\documentclass[a4paper, 11pt]{amsart}
\usepackage[utf8]{inputenc}
\usepackage[T1]{fontenc}
\usepackage[english]{babel}
\usepackage{graphicx}
\usepackage{stmaryrd}
\usepackage{tikz}
\usepackage{tikz-cd} 
\usepackage[all]{xy}
\usepackage{comment}
\usepackage{amsmath,amsfonts,amssymb}
\usepackage[a4paper]{geometry}
\usepackage{amsthm}
\usepackage{float}
\usepackage{enumitem}
\usepackage{hyperref}
\usepackage{xcolor}
\hypersetup{
    colorlinks=true,
    linkcolor={red},
    citecolor={blue},
    urlcolor={blue}}
\newtheorem{te}{Theorem}[section]

\newtheorem{prop}[te]{Proposition}
\newtheorem*{prop2}{Proposition \ref{floor7}}
\newtheorem{co}[te]{Corollary}
\newtheorem{lemme}[te]{Lemma}
\theoremstyle{definition}
\newtheorem{de}[te]{Definition}
\newtheorem{ex}[te]{Example}

\theoremstyle{remark}
\newtheorem{rque}[te]{Remark}
\newenvironment{proof4}{\noindent\textit{Proof of Proposition \ref{cyclique}.~}}{\hfill$\square$\bigbreak} 
\newenvironment{proof6}{\noindent\textit{Proof of Corollary \ref{cyclique2}.~}}{\hfill$\square$\bigbreak}
\newenvironment{proof5}{\noindent\textit{Proof of Proposition \ref{isom}.~}}{\hfill$\square$\bigbreak} 
 
\newenvironment{proof8}{\noindent\textit{Proof of Lemma \ref{truc2bis}.~}}{\hfill$\square$\bigbreak} 
\newenvironment{proof9}{\noindent\textit{Sketch of proof.~}}{\hfill$\square$\bigbreak} 
 
\newenvironment{proof11}{\noindent\textit{Proof of Lemma \ref{lemme3}.~}}{\hfill$\square$\bigbreak} 
\newenvironment{proof12}{\noindent\textit{Proof of Lemma \ref{lemme4}.~}}{\hfill$\square$\bigbreak} 
\newenvironment{proof13}{\noindent\textit{Proof of Lemma \ref{lemme5}.~}}{\hfill$\square$\bigbreak} 

\newenvironment{proof15}{\noindent\textit{Proof of Proposition \ref{disjonctionbis}.~}}{\hfill$\square$\bigbreak} 

\newenvironment{proof17}{\noindent\textit{Proof of Theorem \ref{hypcubulation}.~}}{\hfill$\square$\bigbreak}
\newenvironment{proof18}{\noindent\textit{Proof of Proposition \ref{propcubu}.~}}{\hfill$\square$\bigbreak}
\newenvironment{proof100}{\noindent\textit{Proof of Proposition \ref{100}.~}}{\hfill$\square$\bigbreak} 
 
\newenvironment{proof102}{\noindent\textit{Proof of Lemma \ref{102}.~}}{\hfill$\square$\bigbreak}  
\newlength{\plarg}
\usetikzlibrary{cd}
\calclayout
\title{Hyperbolicity and cubulability are preserved under elementary equivalence}
\author{Simon André}
\date{\today}
\begin{document}
\clearpage
\begin{minipage}{\linewidth}
\begin{abstract}
The following properties are preserved under elementary equivalence, among finitely generated groups: being hyperbolic (possibly with torsion), being hyperbolic and cubulable, and being a subgroup of a hyperbolic group. In other words, if a finitely generated group $G$ has the same first-order theory as a group possessing one of the previous property, then $G$ enjoys this property as well.
\end{abstract}
\maketitle
\end{minipage}

\section{Introduction}

In the middle of the twentieth century, Tarski asked whether all non-abelian finitely generated free groups satisfy the same first-order theory. In \cite{Sel06}, Sela answered Tarski's question in the positive (see also the work of Kharlampovich and Myasnikov, \cite{KM06}) and provided a complete characterization of finitely generated groups with the same first-order theory as the free group $F_2$. Sela extended his work to give in \cite{Sel09} a classification of torsion-free hyperbolic groups up to elementary equivalence. In addition, Sela proved the following striking theorem (see \cite{Sel09}, Theorem 7.10).

\begin{te}[Sela]\label{te1}A finitely generated group with the same first-order theory as a torsion-free hyperbolic group is itself torsion-free hyperbolic.
\end{te}

Recall that a finitely generated group is called hyperbolic (in the sense of Gromov) if its Cayley graph (with respect to any finite generating set) is a hyperbolic metric space: there exists a constant $\delta> 0$ such that any geodesic triangle is $\delta$-slim, meaning that any point that lies on a side of the triangle is at distance at most $\delta$ from the union of the two other sides. At first glance, there doesn't seem to be any reason that hyperbolicity is preserved under elementary equivalence. Sela's theorem above is thus particularly remarkable, and it is natural to ask whether it remains valid if we allow hyperbolic groups to have torsion. We answer this question positively.

\begin{te}\label{te2}A finitely generated group with the same first-order theory as a hyperbolic group is itself hyperbolic.
\end{te}

More precisely, if $\Gamma$ is a hyperbolic group and $G$ is a finitely generated group such that $\mathrm{Th}_{\forall\exists}(\Gamma) = \mathrm{Th}_{\forall\exists}(G)$ (meaning that $\Gamma$ and $G$ satisfy the same first-order sentences of the form $\forall x_1\ldots\forall x_m\exists y_1\ldots\exists y_n \ \varphi(x_1,\dots,x_m,y_1,\dots , y_n)$, where $\varphi$ is quantifier-free), then $G$ is hyperbolic.

Furthermore, we show that being a subgroup of a hyperbolic group is preserved under elementary equivalence, among finitely generated groups. More precisely, we prove the following theorem.

\begin{te}\label{sousgroupe2}Let $\Gamma$ be a group that embeds into a hyperbolic group, and let $G$ be a finitely generated group. If $\mathrm{Th}_{\forall\exists}(\Gamma)\subset\mathrm{Th}_{\forall\exists}(G)$, then $G$ embeds into a hyperbolic group.
\end{te}

Recall that $\mathrm{CAT}(0)$ cube complexes are a particular class of $\mathrm{CAT}(0)$ spaces (see \cite{Sag14} for an introduction) and that a group is called cubulable if it admits a proper and cocompact action by isometries on a $\mathrm{CAT}(0)$ cube complex. Groups which are both hyperbolic and cubulable have remarkable properties and play a leading role in the proof of the virtually Haken conjecture (see \cite{Ago13}). By using a result of Hsu and Wise, we prove:

\begin{te}\label{corollaire}Let $\Gamma$ be a hyperbolic group and $G$ a finitely generated group. Suppose that $\mathrm{Th}_{\forall\exists}(\Gamma) = \mathrm{Th}_{\forall\exists}(G)$. Then $\Gamma$ is cubulable if and only if $G$ is cubulable.
\end{te}

\subsection*{Some remarks}Note that, in Theorem \ref{sousgroupe2}, we do not need to assume that $\Gamma$ is finitely generated. In contrast, in the theorems above, it is impossible to remove the assumption of finite generation on $G$. For instance, Szmielew proved in \cite{Szm55} that $\mathbb{Z}$ and $\mathbb{Z}\times\mathbb{Q}$ have the same first-order theory. Note also that it is not sufficient to assume that $\Gamma$ and $G$ satisfy the same universal sentences, i.e.\ of the form $\forall x_1\ldots\forall x_m \ \varphi(x_1,\dots,x_m)$, where $\varphi$ is quantifier-free. For example, the class of finitely generated groups satisfying the same universal sentences as $F_2$ appears to coincide with the well-known class of non-abelian limit groups, also known as finitely generated fully residually free non-abelian groups, and this class contains some non-hyperbolic groups, such as $F_2\ast\mathbb{Z}^2$.

In Theorem \ref{sousgroupe2} above, in the special case where $\Gamma$ is a free group, we can prove that $G$ is hyperbolic. Indeed, Sela proved in \cite{Sel01} that a limit group is hyperbolic if it does not contain $\mathbb{Z}^2$, and this holds for a finitely generated group $G$ such that $\mathrm{Th}_{\forall\exists}(F_2)\subset\mathrm{Th}_{\forall\exists}(G)$ (see \ref{cyclique2}). More generally, Sela's result remains true if we replace $F_2$ by a finitely generated locally hyperbolic group $\Gamma$ (that is, every finitely generated subgroup of $\Gamma$ is hyperbolic). If $\Gamma$ is such a group, it turns out that the hyperbolic group built in the proof of Theorem \ref{sousgroupe2}, in which $G$ embeds, is in fact locally hyperbolic. So $G$ is locally hyperbolic as well. As a consequence, the following theorem holds.

\begin{te}Let $\Gamma$ be a finitely generated locally hyperbolic group, and let $G$ be a finitely generated group. If $\mathrm{Th}_{\forall\exists}(\Gamma)\subset\mathrm{Th}_{\forall\exists}(G)$, then $G$ is a locally hyperbolic group.
\end{te}

In the rest of the introduction, we give some details about the proofs of the previous theorems.

\subsection*{Strategy of proof}In the language of groups, a system of equations in $n$ variables is a conjunction of formulas of the form $w(x_1,\ldots,x_n)=1$, where $w(x_1,\ldots,x_n)$ stands for an element of the free group $F(x_1,\ldots ,x_n)$. Given a group $\Gamma$, a finite system of equations $\Sigma(x_1,\ldots ,x_n)=1$ admits a non-trivial solution in $\Gamma^n$ if and only if $\Gamma$ satisfies the first-order sentence $\exists x_1\ldots\exists x_n \ (\Sigma(x_1,\ldots,x_n)=1)\wedge ((x_1\neq 1) \vee \ldots \vee (x_n\neq 1))$, if and only if there exists a non-trivial homomorphism from $G_{\Sigma}$ to $\Gamma$, where $G_{\Sigma}=\langle s_1,\ldots ,s_n \ \vert \ \Sigma(s_1,\ldots,s_n)\rangle$. Hence, the study of the set $\mathrm{Hom}(G_{\Sigma},\Gamma)$, for any (finite) system of equations $\Sigma$, is a first step towards understanding the whole first-order theory of the group $\Gamma$. In the case where $\Gamma$ is a torsion-free hyperbolic group, Sela proved that there is a finite description of the set $\mathrm{Hom}(G,\Gamma)$ (called the Makanin-Razborov diagram), for any finitely generated group $G$. His work has subsequently been generalized by Reinfeldt and Weidmann to the case of a hyperbolic group possibly with torsion. Below is the basis of this description (see Theorems \ref{sela2} and \ref{sela2bis} for completeness).

\begin{te}[Sela, Reinfeldt-Weidmann]\label{introduction}Let $\Gamma$ be a hyperbolic group, and let $G$ be a finitely generated one-ended group. There exists a finite set $F\subset G\setminus\lbrace 1\rbrace$ such that, for every non-injective homomorphism $f\in\mathrm{Hom}(G,\Gamma)$, there exists a modular automorphism $\sigma$ of $G$ such that $\ker(f\circ\sigma)\cap F \neq \varnothing$. 
\end{te}

The modular group of $G$, denoted by $\mathrm{Mod}(G)$, is a subgroup of $\mathrm{Aut}(G)$ defined by means of the JSJ decomposition of $G$ (see Section \ref{25} for details). Its main feature is that each modular automorphism acts by conjugation on every non-abelian rigid vertex group of the JSJ decomposition of $G$.

This theorem will be our starting point. First, we will consider a particular case. Let $G=\langle s_1,\ldots ,s_n \ \vert \ \Sigma(s_1,\ldots,s_n)\rangle$ be a finitely presented one-ended group, and let $\Gamma$ be a hyperbolic group. Assume that $\mathrm{Th}_{\forall}(\Gamma)=\mathrm{Th}_{\forall}(G)$ (that is, $\Gamma$ and $G$ satisfy the same universal sentences). Suppose in addition that $G$ is a rigid group (meaning that $G$ does not split non-trivially over a virtually abelian group). Then the modular group of $G$ coincides with the group of inner automorphisms. As a consequence, Theorem \ref{introduction} provides us with a finite set $F=\lbrace w_1(s_1,\ldots ,s_n),\ldots ,w_k(s_1,\ldots,s_n)\rbrace\subset G\setminus\lbrace 1\rbrace$ such that every non-injective homomorphism $f\in\mathrm{Hom}(G,\Gamma)$ kills an element of $F$. We claim that $G$ embeds into $\Gamma$. Otherwise, every homomorphism from $G$ to $\Gamma$ kills an element of $F$; in other words, the group $\Gamma$ satisfies the following first-order sentence: \[\forall x_1\ldots\forall x_n \ (\Sigma(x_1,\ldots ,x_n)=1)\Rightarrow ((w_1(x_1,\ldots ,x_n)=1)\vee \ldots \vee (w_k(x_1,\ldots ,x_n)=1)).\] Since $\mathrm{Th}_{\forall}(\Gamma)=\mathrm{Th}_{\forall}(G)$, this sentence is true in $G$ as well. Taking $x_1=s_1,\ldots ,x_n=s_n$, the previous sentence means that an element of $F$ is trivial, contradicting the fact that $F\subset G\setminus\lbrace 1\rbrace$. So we have proved that $G$ embeds into $\Gamma$. 

At this stage, we cannot conclude that $G$ is hyperbolic, because hyperbolicity is not inherited by finitely presented subgroups (see \cite{Bra99}). However, assuming that the group $\Gamma$ is rigid as well, we can prove in the same way that $\Gamma$ embeds into $G$ (before doing so, it is necessary to observe that Theorem \ref{introduction} is still valid for subgroups of hyperbolic groups, see Corollary \ref{simple}). 

Let us summarize the previous discussion: given a rigid hyperbolic group $\Gamma$ and a rigid finitely presented group $G$, we have shown that if $\mathrm{Th}_{\forall}(\Gamma)=\mathrm{Th}_{\forall}(G)$, then $G$ and $\Gamma$ are isomorphic. In the case where $G$ is not finitely presented anymore, but only finitely generated, this result remains true because hyperbolic groups are equationally noetherian, which means that, for any infinite system of equations $\Sigma(x_1,\ldots,x_n)=1$, there exists a finite subsystem $\Sigma'(x_1,\ldots,x_n)=1$ such that $\Sigma$ and $\Sigma'$ have exactly the same solutions in $\Gamma^n$ (see \cite{RW14}).

What happens when $G$ is not assumed to be rigid anymore ? In this case, the modular group $\mathrm{Mod}(G)$ is in general infinite, and we cannot fully express Theorem \ref{introduction} by means of first-order logic, since the power of expression of first-order logic is not sufficient to express precomposition by an automorphism. The challenge is to express some fragments of Theorem \ref{introduction} that are enough to capture the hyperbolicity (or cubulability) of $\Gamma$ and prove the hyperbolicity of $G$ in turn. To do so, we shall follow a strategy comparable with that used by Perin in \cite{Per11}. 

If $G$ does not embed into $\Gamma$, the idea is to consider the following immediate corollary of Theorem \ref{introduction}: for every homomorphism $f : G\rightarrow \Gamma$, there exists a homomorphism $f' : G\rightarrow\Gamma$ that kills an element of $F$ and that coincides with $f$ in restriction to each non-abelian rigid vertex group of the JSJ splitting of $G$, up to conjugacy by an element of $\Gamma$. One can easily see that this statement is expressible by a first-order sentence satisfied by $\Gamma$ (see Section \ref{32} for details). Since $\Gamma$ and $G$ are elementarily equivalent, this sentence is satisfied by $G$ as well.

Now, taking for $f$ the identity of $G$, we get a special homomorphism $f'$ that kills an element of $F$ and coincides with the identity, up to conjugacy, in restriction to each non-abelian rigid vertex group of the JSJ splitting of $G$. In the torsion-free case, an important part of the work consists in transforming $f'$ into a retraction, which leads to the notion of a hyperbolic tower in the sense of Sela. In the presence of torsion, there are new difficulties and we do not know how to get a retraction (see further details below). This leads to the notion of a quasi-tower (that we introduce in Section \ref{section5}), which is more complicated than the notion of a hyperbolic tower. 

\subsection*{The torsion-free case}The proof that hyperbolicity and cubulability are first-order invariants among torsion-free finitely generated groups is based on the following result proved by Sela in \cite{Sel09} (though it is not explicitly stated).

\begin{te}[Sela]\label{intro3}Let $\Gamma$ be a non-elementary torsion-free hyperbolic group, and let $G$ be a finitely generated group with the same first-order theory as $\Gamma$. There exist a subgroup $\Gamma'$ of $\Gamma$ and a subgroup $G'$ of $G$ such that:
\begin{itemize}
\item[$\bullet$]$\Gamma'$ and $G'$ are isomorphic;
\item[$\bullet$]$\Gamma$ is a hyperbolic tower over $\Gamma'$;
\item[$\bullet$]$G$ is a hyperbolic tower over $G'$.
\end{itemize}
\end{te}

Hyperbolic towers have been introduced by Sela in \cite{Sel01} (see Definition 6.1) to solve Tarski's problem about the elementary equivalence of free groups (see also Kharlampovich and Myasnikov's NTQ groups). More generally, Sela used towers in \cite{Sel09} to classify those finitely generated groups with the same first-order theory as a given torsion-free hyperbolic group. Roughly speaking, a hyperbolic tower is a group obtained by successive addition of hyperbolic floors, and a hyperbolic floor is a group obtained by gluing a retracting surface to another group (see Example \ref{exempleintro} below). We refer the reader to \cite{Per11} for a precise definition of hyperbolic towers. 

\begin{ex}\label{exempleintro}Let $\Sigma$ be a surface with boundary (with Euler characteristic at most -2, or a punctured torus). Denote by $S$ its fundamental group and by $B_1,\ldots, B_n$ its boundary subgroups (well-defined up to conjugacy). Let $H$ be a group and let $h_1,\ldots ,h_n$ be elements of $H$ of infinite order. We define a graph of groups with two vertices labelled by $S$ and $H$, and $n$ edges between them identifying $B_i$ with $\langle h_i\rangle$ for each $1\leq i\leq n$. Call $G$ the fundamental group of this graph of groups. We say that $G$ is a hyperbolic floor over $H$ if there exists a retraction $r : G \rightarrow H$ (such that $r(S)$ is non-abelian).
\begin{figure}[h!]
\includegraphics[scale=0.45]{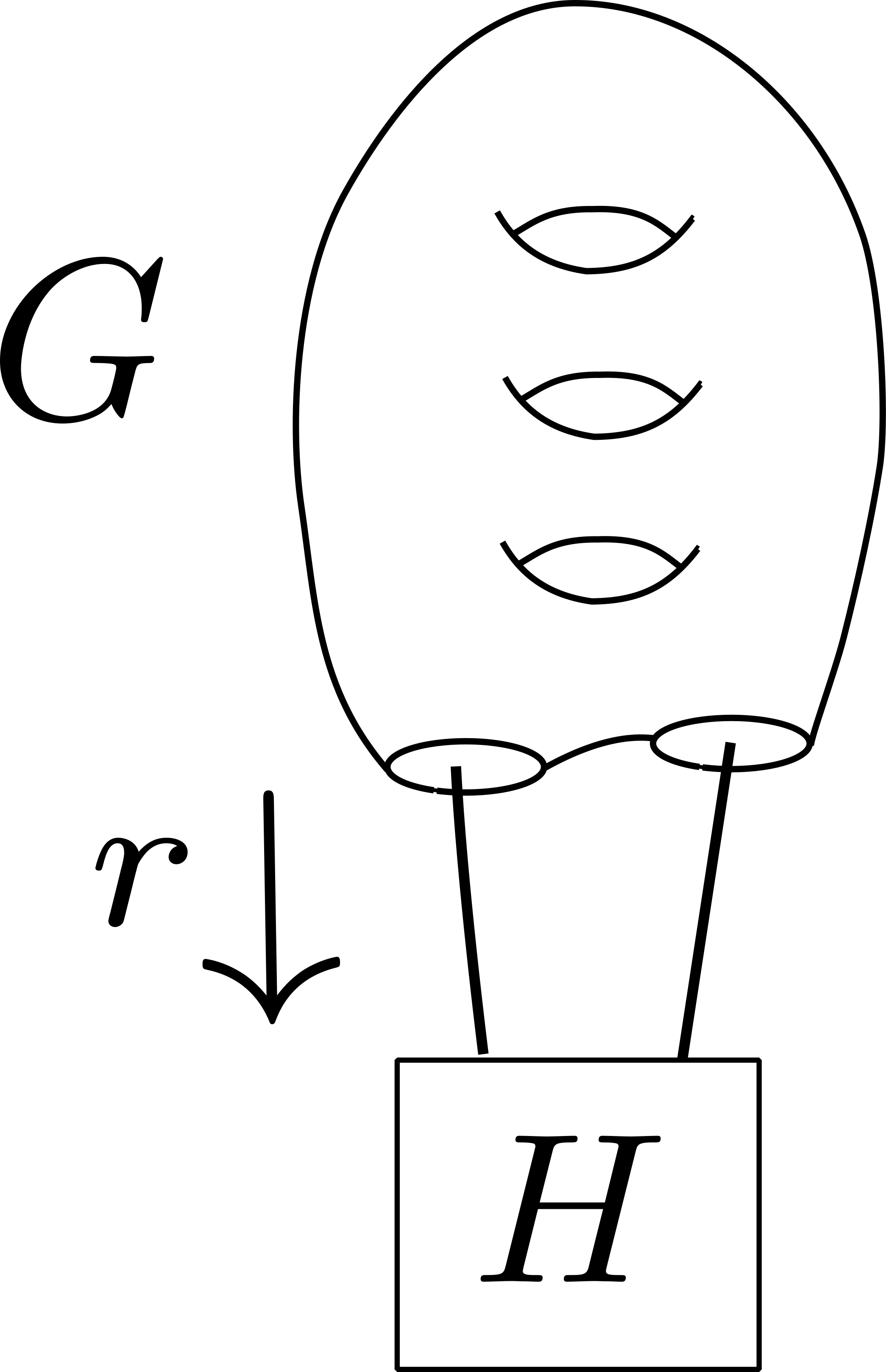}
\caption{The group $G$ is a hyperbolic floor over $H$ in the sense of Sela.}
\end{figure}
\end{ex}
In the previous example, if $G$ is hyperbolic, then $H$ is hyperbolic, as a retract of $G$. Conversely, if $H$ is hyperbolic, it follows from the Bestvina-Feighn combination theorem (see \cite{BF92}) that $G$ is hyperbolic as well. As a consequence, since a hyperbolic tower is a group obtained by successive addition of hyperbolic floors, the following holds: if a group $G$ is a hyperbolic tower over a group $H$, then $G$ is hyperbolic if and only if $H$ is hyperbolic. In the same way, using a combination theorem proved by Hsu and Wise (see \cite{HW15}), we can prove: if a group $G$ is a hyperbolic tower over a group $H$, then $G$ is hyperbolic and cubulable if and only if $H$ is hyperbolic and cubulable.

Now, the fact that hyperbolicity and cubulability are first-order invariants (among torsion-free finitely generated groups) follows immediately from Theorem \ref{intro3}.

\begin{figure}[!h]
\centering
\includegraphics[scale=0.45]{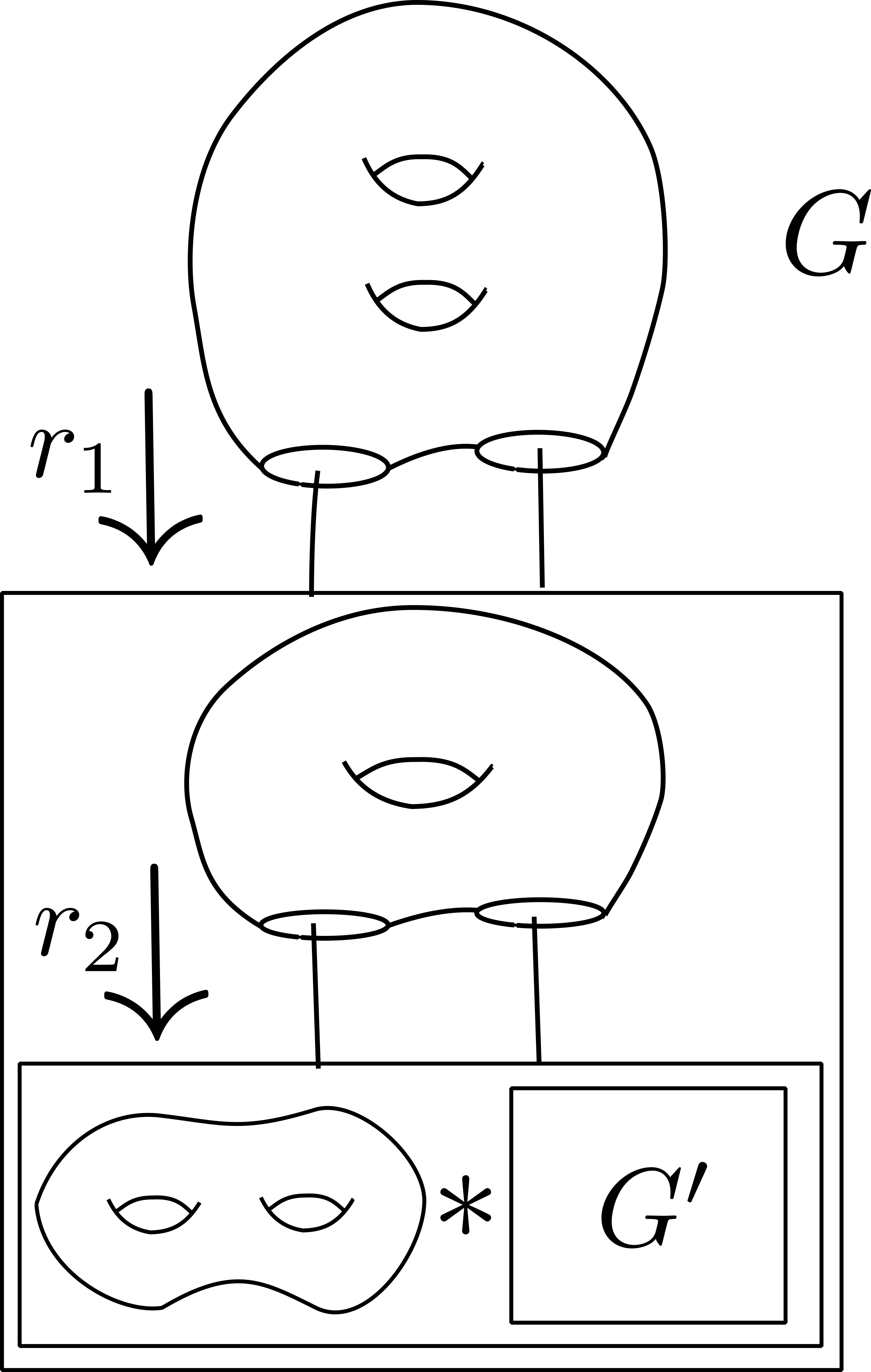}
\caption{The group $G$ is a hyperbolic tower over $G'$ in the sense of Sela.}
\label{cassimple2}
\end{figure}

\subsection*{A new phenomenon arising from the presence of torsion}To give a flavour of the influence of torsion, let us briefly describe a new phenomenon, which occurs only in the presence of torsion. In the case of torsion-free groups, performing a HNN extension over a finite group is the same as doing a free product with $\mathbb{Z}$, and it is well-known that if $G$ is a torsion-free non-elementary hyperbolic group, then $G$ and $G\ast\mathbb{Z}$ have the same universal theory (and even the same first-order theory, as a consequence of the work of Sela). By contrast, it turns out that the situation is very different in the presence of torsion: in general, performing a HNN extension over a finite subgroup, even trivial, modifies the universal theory of a hyperbolic group. Let us consider the following simple example: 

\begin{ex}\label{instable}Let $G=F_2\times\mathbb{Z}/2\mathbb{Z}$. Then the sentence $\forall x\forall y \ (x^2=1)\Rightarrow (xy=yx)$ is satisfied by $G$, but not by $G\ast\mathbb{Z}$. 
\end{ex}

This example shows that, in general, the class of groups with the same universal theory as a given hyperbolic group with torsion is not closed under HNN extensions and amalgams over finite groups. In Section \ref{section 4}, we deal with this problem by proving the following result.

\begin{te}\label{propintro}Every hyperbolic group $\Gamma$ embeds into a hyperbolic group $\overline{\Gamma}$ possessing the property that the class of $\overline{\Gamma}$-limit groups is closed under amalgamated free products and HNN extensions over finite groups.
\end{te}

We will see that this result plays an important role in our proofs.

\subsection*{Quasi-towers}In order to prove our main theorems, we aim to generalize Theorem \ref{intro3} to the case where $\Gamma$ is a hyperbolic group possibly with torsion (see Theorem \ref{isom}). When looking for a generalization of hyperbolic towers in the presence of torsion, we are brought to perform HNN extensions over finite groups, and this leads to new difficulties, as illustrated by Example \ref{instable} above. In Section \ref{section5} we introduce quasi-floors and quasi-towers. It is important to stress that quasi-towers are more complicated than hyperbolic towers in the sense of Sela, the major complication being that if a group $G$ is a quasi-floor over a group $H$, then $H$ is neither a subgroup, nor a quotient of $G$. Roughly speaking, a quasi-tower is a group obtained by successive addition of quasi-floors, and Figure \ref{introduction} below illustrates what a quasi-floor is.

\begin{figure}[!h]
\centering
\includegraphics[scale=0.018]{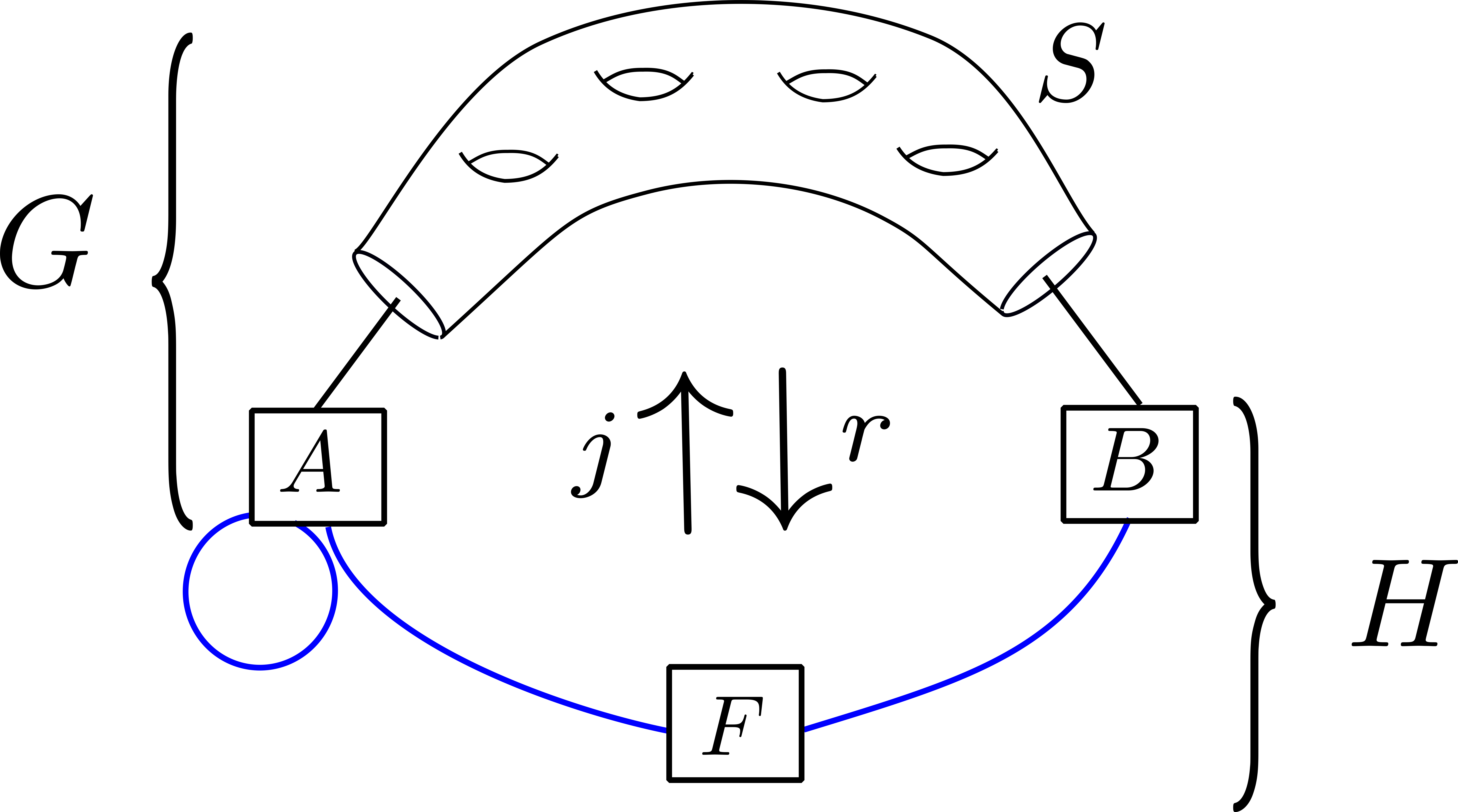}
\caption{The group $G$ is a quasi-floor over $H$. The groups $A$ and $B$ are one-ended, and $F$ is finite. There exist a homomorphism $r : G\rightarrow H$, maybe not surjective, and a homomorphism $j:H\rightarrow G$, maybe not injective. However, $r$ can be viewed as a "piecewise retraction", and $j$ can be viewed as a "piecewise inclusion". Indeed, $r\circ j$ is inner on $A$ and $B$, and $j$ is injective on $A$, $B$ and $F$. The groups $G$ and $H$ are subgroups of a bigger group $G'$ that retracts onto $H$ via an epimorphism $\rho : G' \rightarrow H$ such that $\rho_{\vert G}=r$.}
\label{cassimple2}
\end{figure}

Let us say a few words about the proof of Theorem \ref{intro3}, which we aim to generalize. Let $\Gamma$ be a torsion-free hyperbolic group, and let $G$ be a finitely generated group with the same first-order theory as $\Gamma$. First-order logic provides us with a sequence of hyperbolic floors \[G=G_0\underset{r_1}{\twoheadrightarrow} G_1\underset{r_2}{\twoheadrightarrow}\cdots \underset{r_n}{\twoheadrightarrow}G_n{\twoheadrightarrow}\cdots,\]and this sequence is necessarily finite thanks to the descending chain condition \ref{chaine} for $\Gamma$-limit groups. In the presence of torsion, we cannot use this theorem straightaway for two reasons:
\begin{itemize}
\item[$\bullet$]if a group $G$ is a quasi-floor over a group $H$, the homomorphism $r:G\rightarrow H$ associated with the quasi-floor structure is maybe not surjective,
\item[$\bullet$]and $\mathrm{Th}_{\forall}(G)$ is not equal to $\mathrm{Th}_{\forall}(H)$ in general, since the first-order theory with one quantifier is sensitive to amalgams and HNN extensions over finite groups (see Example \ref{instable}).
\end{itemize}
This second difficulty will be resolved by means of Theorem \ref{propintro}. The problem of non-surjectivity of $r$ will be solved by proving that if \[G=G_0\underset{r_1}{\rightarrow} G_1\underset{r_2}{\rightarrow}\cdots \underset{r_n}{\rightarrow}G_n{\rightarrow}\cdots\]is a sequence of quasi-floors, there exists a subgroup $H_n<G_n$, for each $n$, such that the restriction of $r_n$ to $H_n$ is an epimorphism onto $H_{n+1}$.

\subsection*{Acknowledgements}I am very grateful to my advisor Vincent Guirardel for his help, and for the time he spent carefully reading previous versions of this paper. I would also like to thank Chloé Perin for a talk she gave at "Session états de la recherche de la SMF", where she explained the proof of Theorem \ref{te1}. It has been an important source of inspiration.

\hypersetup{colorlinks=true, linkcolor=black}
\tableofcontents
\hypersetup{colorlinks=true, linkcolor=red}

\section{Preliminaries}

In this section we recall some facts and definitions about the elementary theory of groups, $\Gamma$-limit groups, $K$-$\mathrm{CSA}$ groups, JSJ decompositions. We also prove some results that will be useful in the sequel, and whose proofs are independent from the main body of the paper.

\subsection{The elementary theory of groups}

For detailed background, we refer the reader to \cite{Mar02}.

A first-order formula in the language of groups is a finite formula using the following symbols: $\forall$, $\exists$, $=$, $\wedge$, $\vee$, $\Rightarrow$, $\neq$, $1$ (standing for the identity element), ${}^{-1}$ (standing for the inverse), $\cdot$ (standing for the group multiplication) and variables $x,y,g,z\ldots$ which are to be interpreted as elements of a group. 

A variable is free if it is not bound by any quantifier $\forall$ or $\exists$. A sentence is a formula without free variables.

Given a formula $\varphi(x_1,\ldots ,x_p)$ with $p\geq 0$ free variables, and $p$ elements $g_1,\ldots ,g_p$ of a group $G$, we say that $\varphi(g_1,\ldots ,g_p)$ is satisfied by $G$ if its interpretation is true in $G$. This is denoted by $G\models \varphi(g_1,\ldots ,g_p)$.

The elementary theory of a group $G$, denoted by $\mathrm{Th}(G)$, is the collection of all sentences which are true in $G$. The universal-existential theory of $G$, denoted by $\mathrm{Th}_{\forall\exists}(G)$, is the collection of sentences true in $G$ which are of the form \[\forall x_1\ldots\forall x_p\exists y_1\ldots\exists y_q \ \varphi(x_1,\dots,x_p,y_1,\dots , y_q)\] where $p,q\geq 1$ and $\varphi$ is a quantifier-free formula with $p+q$ free variables. In the same way, we define the universal theory of $G$, denoted by $\mathrm{Th}_{\forall}(G)$, and its existential theory $\mathrm{Th}_{\exists}(G)$.

\subsection{$\Gamma$-limit groups}\label{22}

Let $\Gamma$ and $G$ be two groups. We say that $G$ is fully residually $\Gamma$ if, for every finite subset $F\subset G$, there exists a homomorphism $f : G \rightarrow \Gamma$ whose restriction to $F$ is injective. If $G$ is countable, then $G$ is fully residually $\Gamma$ if and only if there exists a sequence $(f_n)$ of homomorphisms from $G$ to $\Gamma$ such that, for every non-trivial element $g\in G$, $f_n(g)$ is non-trivial for every $n$ large enough. Such a sequence is called a discriminating sequence.

$\Gamma$-limit groups have been introduced by Sela in \cite{Sel09} to study $\mathrm{Hom}(G,\Gamma)$, where $\Gamma$ stands for a torsion-free hyperbolic group, and $G$ stands for a finitely generated group. Sela proved that $\Gamma$-limit groups are exactly those finitely generated fully residually $\Gamma$ groups. Reinfeldt and Weidmann generalized this result when $\Gamma$ is hyperbolic possibly with torsion and, more generally, when $\Gamma$ is equationally noetherian (see \cite{RW14}).

The following easy proposition builds a bridge between group theory and first-order logic.

\begin{prop}Let $G$ and $\Gamma$ be finitely generated groups. Suppose that $G$ is finitely presented or that $\Gamma$ is equationally noetherian. Then $G$ is fully residually $\Gamma$ if and only if $\mathrm{Th}_{\forall}(\Gamma)\subset\mathrm{Th}_{\forall}(G)$.
\end{prop}

The proofs of our main results rely essentially on the following four theorems, proved by Sela for torsion-free hyperbolic groups and later generalized by Reinfeldt and Weidmann to the case of hyperbolic groups possibly with torsion.

\begin{te}[Sela, Reinfeldt-Weidmann] \label{sela2}Let $\Gamma$ be a hyperbolic group and let $G$ be a one-ended finitely generated group. There exist non trivial elements $g_1,\ldots , g_k$ of $G$ such that, for every non-injective homomorphism $f : G\rightarrow\Gamma$, there exist a modular automorphism $\sigma\in\mathrm{Mod}(G)$ and an integer $1\leq \ell \leq k$ such that $f\circ\sigma(g_{\ell})=1$.
\end{te}

\begin{rque}In the case where $G$ is not a $\Gamma$-limit group, the result is obvious.
\end{rque}

The modular group $\mathrm{Mod}(G)$ is a subgroup of $\mathrm{Aut}(G)$ that will be defined in Section \ref{262}, by means of the JSJ decomposition. 

The following easy corollary will be useful in the sequel.

\begin{co}\label{simple}Let $\Gamma$ be a group that embeds into a hyperbolic group, and let $G$ be a one-ended finitely generated group. There exist non trivial elements $g_1,\ldots , g_k$ of $G$ such that, for every non-injective homomorphism $f : G\rightarrow\Gamma$, there exist a modular automorphism $\sigma\in\mathrm{Mod}(G)$ and an integer $1\leq \ell \leq k$ such that $f\circ\sigma(g_{\ell})=1$.
\end{co}

\begin{proof}
Denote by $i$ an embedding of $\Gamma$ into a hyperbolic group $\Omega$. By the previous theorem, there exist non trivial elements $g_1,\ldots , g_k$ of $G$ such that every non-injective homomorphism from $G$ to $\Omega$ kills some $g_{\ell}$, up to precomposition by a modular automorphism. If $f: G \rightarrow \Gamma$ is a non-injective homomorphism, then $i\circ f$ is non-injective as well, so $i\circ f\circ \sigma(g_{\ell})=1$ for some $1\leq \ell \leq k$ and some $\sigma\in\mathrm{Mod}(G)$. Since $i$ is injective, we have $f\circ \sigma(g_{\ell})=1$.
\end{proof}

For every $\gamma\in \Gamma$, we write $\iota_{\gamma}$ for the inner automorphism $x\mapsto\gamma x\gamma^{-1}$.

\begin{te}[Sela, Reinfeldt-Weidmann] \label{sela2bis}Let $\Gamma$ be a hyperbolic group and let $G$ be a one-ended finitely generated group. Suppose that $G$ embeds into $\Gamma$. Then there exists a finite set $\lbrace \varphi_1,\ldots ,\varphi_{\ell}\rbrace$ of monomorphisms from $G$ into $\Gamma$ such that, for every monomorphism $f : G\rightarrow\Gamma$,  there exist a modular automorphism $\sigma\in\mathrm{Mod}(G)$, an integer $1\leq \ell \leq k$ and an element $\gamma\in\Gamma$ such that \[f=\iota_{\gamma}\circ \varphi_{\ell} \circ \sigma. \]
\end{te}

\begin{rque}In contrast with Theorem \ref{sela2}, the previous theorem does not extend to the case where $\Gamma$ is a subgroup of a hyperbolic group.
\end{rque}

\begin{te}[Sela, Reinfeldt-Weidmann]\label{chaine}Let $\Gamma$ be a hyperbolic group. Let $(G_n)$ be a sequence of $\Gamma$-limit groups. If $(f_n : G_n \rightarrow G_{n+1})$ is a sequence of epimorphisms, then $f_n$ is an isomorphism for $n$ sufficiently large.
\end{te}

\begin{te}[Sela, Reinfeldt-Weidmann]\label{typefini}Let $\Gamma$ be a hyperbolic group and $G$ a $\Gamma$-limit group. Then every abelian subgroup of $G$ is finitely generated.
\end{te}

In the next section, we shall prove the following proposition, which is a consequence of the theorem above.

\begin{prop}Let $\Gamma$ be a hyperbolic group and $G$ a finitely generated group. Suppose that $\mathrm{Th}_{\forall\exists}(\Gamma)\subset\mathrm{Th}_{\forall\exists}(G)$. Then every abelian subgroup of $G$ is virtually cyclic.
\end{prop}

But before proving this proposition, we need to define $K$-$\mathrm{CSA}$ groups.

\subsection{Virtual cyclicity of abelian subgroups}

\subsubsection{$K$-$\mathrm{CSA}$ groups} 

A group is said to be $\mathrm{CSA}$ if its maximal abelian subgroups are malnormal. It is well-known that every torsion-free hyperbolic group is $\mathrm{CSA}$. However, it is not true anymore in the presence of torsion. To overcome this problem, Guirardel and Levitt defined $K$-$\mathrm{CSA}$ groups in \cite{GL16} (Definition 9.7).

\begin{de}\label{CSA}
\normalfont
A group $G$ is called a $K$-$\mathrm{CSA}$ group (where $K>0$) if the following conditions hold:
\begin{itemize}
\item[$\bullet$] every finite subgroup of $G$ has order bounded by above by $K$ (hence, an element $g$ has infinite order if and only if $g^{K!}\neq e$);
\item[$\bullet$] every element $g$ of infinite order is contained in a unique maximal virtually abelian subgroup of $G$, denoted by $M(g)$. Moreover $M(g)$ is $K$-virtually torsion-free abelian (i.e.\ $M(g)$ has a torsion-free abelian subgroup of index less than $K$);
\item[$\bullet$] $M(g)$ is equal to its normalizer.
\end{itemize}
\end{de}

We recall some useful facts about $K$-$\mathrm{CSA}$ groups. The proofs can be found in \cite{GL16}.

\begin{prop}\label{hyp}Every hyperbolic group is $K$-$\mathrm{CSA}$ for some $K>0$.
\end{prop}

\begin{prop}\label{CSA2}Let $G$ be a $K$-$\mathrm{CSA}$ group.
\begin{enumerate}
\item If $g,h\in G$ have infinite order, the following conditions are equivalent:
\begin{enumerate}
\item $M(g)=M(h)$.
\item $g^{K!}$ and $h^{K!}$ commute.
\item $\langle g,h\rangle$ is virtually abelian.
\end{enumerate}
\item Let $H$ be an infinite virtually abelian subgroup of $G$. Then $H$ is contained in a unique maximal virtually abelian subgroup of $G$, denoted by $M(H)$. This group is almost malnormal: if $M(H)^g\cap M(H)$ is infinite, then $g$ belongs to $M(H)$. Moreover, for every element $h$ of $H$ of infinite order, $M(H)=M(h)$.
\end{enumerate}
\end{prop}

\begin{prop}\label{définissable}Let $G$ be a $K$-$\mathrm{CSA}$ group and $g$ an element of $G$ of infinite order. The subgroup $M(g)$ is definable without quantifiers with respect to $g$. In other words, there exists a first-order formula $\psi_K(x,y)$ without quantifiers such that \[M(g)=\lbrace h\in G \ \vert \ \psi_K(h,g)\rbrace.\]
\end{prop}

\begin{proof}
First, let us remark that $M(g)=\lbrace h\in G \ \vert \ \langle g,h\rangle \ \text{is K-virtually abelian}\rbrace$. Indeed, if $\langle g,h\rangle$ is $K$-virtually abelian, then $\langle g,h\rangle\subset M(g)$ by maximality of $M(g)$. Conversely, if $h\in M(g)$ then $\langle g,h\rangle$ is a subgroup of $M(g)$, which is $K$-virtually abelian, so $\langle g,h\rangle$ is $K$-virtually abelian. We now prove that there exists a first-order formula $\psi_K(x,y)$ with two free variables such that $\langle g,h\rangle$ is $K$-virtually abelian if and only if $\psi_K(g,h)$ is true in $G$. Let $\pi : F_2=\langle x,y\rangle\rightarrow G$ be the epimorphism sending $x$ to $g$ and $y$ to $h$. If $A$ is a subgroup of $\langle g,h\rangle$ of index less than $K$, there exists a subgroup $B$ of $\langle x,y\rangle$ of index less than $K$ such that $A=\pi(B)$. Denote by $H_1,\ldots ,H_n$ the $n$ subgroups of $F_2$, of index $\leq K$. For each $1\leq i \leq n$, let $(w_{i,j}(x,y))_{1\leq j\leq n_i}$ be a finite generating set of $H_i$. We can define $\psi_K(g,h)$ by\[\psi_K(g,h)=\bigvee_{i=1}^n\bigwedge_{k=1}^{n_i}\bigwedge_{\ell=1}^{n_i}[w_{i,k}(g,h),w_{i,\ell}(g,h)]=1.\]
\end{proof}

One can prove that the property $K$-$\mathrm{CSA}$ is defined by a set of universal formulas (see \cite{GL16}, Proposition 9.9). Since every hyperbolic group is $K$-$\mathrm{CSA}$ for some $K>0$ (see above), the following holds:

\begin{prop}\label{universelle}Let $\Gamma$ be a hyperbolic group. There exists a constant $K>0$ such that every $\Gamma$-limit group is $K$-$\mathrm{CSA}$.
\end{prop}

\subsubsection{Abelian subgroups are virtually cyclic} First, recall the following well-known result:

\begin{lemme}\label{QFA}There exists a $\forall\exists$-sentence $\phi$ such that, if $G$ is a finitely generated torsion-free abelian group, $G\models\phi$ if and only if $G$ is cyclic.
\end{lemme}

\begin{proof}
Since ${\mathbb{Z}}^n/2\mathbb{Z}^n$ has $2^n$ elements, it results from the pigeonhole principle that the following $\forall\exists$-sentence is verified by $\mathbb{Z}^n$ if and only if $n=1$:\[\forall x_1\forall x_2\forall x_3\exists x_4 \ (x_1=x_2 x_4^2)\vee (x_1=x_3 x_4^2)\vee (x_2=x_3 x_4^2).\]\end{proof}

The following results will be important in the sequel.

\begin{prop}\label{cyclique}Let $\Gamma$ be a hyperbolic group and  $G$ a finitely generated group such that $\mathrm{Th}_{\forall\exists}(\Gamma)\subset\mathrm{Th}_{\forall\exists}(G)$. If $g$ is an element of $G$ of infinite order, then $M(g)$ is virtually cyclic.
\end{prop}

\begin{co}\label{cyclique2}Let $\Gamma$ be a hyperbolic group and $G$ a finitely generated group such that $\mathrm{Th}_{\forall\exists}(\Gamma)\subset\mathrm{Th}_{\forall\exists}(G)$. Then every abelian subgroup of $G$ is virtually cyclic.
\end{co}

\begin{proof6}
By Proposition \ref{universelle}, $G$ is $K$-$\mathrm{CSA}$ for some $K$. Let $H$ be an infinite abelian subgroup of $G$. By Proposition \ref{CSA2}, $H$ is contained in a unique maximal virtually abelian subgroup of $G$, denoted by $M(H)$, and $M(H)=M(h)$ for every element $h$ of $H$ of infinite order (note that such an element exists since $H$ is abelian, finitely generated (according to Lemma \ref{typefini}) and infinite). According to Proposition \ref{cyclique} above, $M(h)$ is virtually cyclic, hence $H$ is virtually cyclic.
\end{proof6}

\begin{proof4}
By Proposition \ref{universelle}, $\Gamma$ and $G$ are $K$-$\mathrm{CSA}$ for some $K$. Let $g$ be an element of $G$ of infinite order. Since the group $M(g)$ is $K$-virtually torsion-free abelian, it has a normal torsion-free abelian subgroup $N$ of index dividing $K!$. For every element $h$ of $M(g)$ the element $h^{K!}$ belongs to $N$. Denote by $N(g)$ the subgroup of $M(g)$ generated by $\left\lbrace h^{K!} \ \vert \ h\in M(g)\right\rbrace$. It is a subgroup of $N$, so it is torsion-free abelian. 

It is enough to show that $N(g)$ is cyclic. Then we will be able to conclude that $M(g)$ is virtually cyclic. Indeed, if $N(g)$ is cyclic, so is $\langle x^{K!},y^{K!}\rangle$ for every $x,y\in M(g)$. As a consequence there is no pair of elements of $M(g)$ generating a subgroup isomorphic to $\mathbb{Z}^2$. Since $M(g)$ is virtually abelian and finitely generated (according to Lemma \ref{typefini}), it is thus virtually cyclic.

For every integer $\ell\geq 1$, let $N_{\ell}(g)=\left\lbrace h_1^{K!}\cdots h_{\ell}^{K!} \ \vert \ h_1,\ldots ,h_{\ell}\in M(g) \right\rbrace$. Since $N(g)$ is finitely generated, there exist an integer $r$ and some elements $g_1,\ldots ,g_r$ of $M(g)$ such that $N(g)$ is generated by $\left\lbrace g_1^{K!},\ldots ,g_r^{K!}\right\rbrace$. We claim that $N(g)=N_r(g)$. In order to see this, remark that, since $N(g)$ is abelian, every element $h\in N(g)$ can be written as follows: \[h={\left(g_1^{K!}\right)}^{n_1}\cdots{\left(g_r^{K!}\right)}^{n_r}={\left(g_1^{n_1}\right)}^{K!}\cdots{\left(g_r^{n_r}\right)}^{K!},\] where $n_1,\ldots,n_r$ lie in $\mathbb{Z}$. This proves that $N(g)\subset N_r(g)$, and the reverse inclusion is immediate. Then, recall that there exists a first-order formula without quantifiers $\psi(x,y)$ such that $M(g)=\lbrace h\in G \ \vert \ \psi(h,g)\rbrace$ (see Proposition \ref{définissable}). Hence \[N_r(g)=\left\lbrace h\in G \ \vert \ \exists h_1 \ldots \exists h_{r} \ \left(h=h_1^{K!}\cdots h_{r}^{K!} \wedge \psi(h_1,g) \wedge \ldots \wedge \psi(h_{r},g)\right)\right\rbrace.\]

It remains to prove that $N_r(g)$ is cyclic. Recall that, by Lemma \ref{QFA}, a finitely generated torsion-free abelian group is cyclic if and only if it satisfies $\forall x_1\forall x_2\forall x_3\exists x_4 \phi(x_1,x_2,x_3,x_4)$, where $\phi(x_1,x_2,x_3,x_4):(x_1=x_2 x_4^2)\vee (x_1=x_3 x_4^2)\vee (x_2=x_3 x_4^2)$. Since $\Gamma$ is a hyperbolic group, every torsion-free abelian subgroup of $\Gamma$ is cyclic, so satisfies the previous sentence. We can write a $\forall\exists$-sentence $\varphi_{r}$ satisfied by $\Gamma$, with the following interpretation: for every element $\gamma$ of $\Gamma$ of infinite order, $N_r(\gamma)$ is cyclic. Below is the sentence $\varphi_r$, where $\underline{h_i}$ stands for $(h_{i,1},\ldots,h_{i,r})$ and $x_i:=h_{i,1}^{K!}\cdots h_{i,r}^{K!}$.
{\small
\[\varphi_r:\forall\gamma\forall \underline{h_1} \forall \underline{h_2}\forall \underline{h_3}\exists \underline{h_4}\left(\bigwedge_{i=1}^{3}\bigwedge_{j=1}^{r}\psi(h_{i,j},\gamma)\wedge\left(\gamma^{K!}\neq 1\right)\right)
   \Rightarrow \left(\bigwedge_{j=1}^{r}\psi(h_{4,j},\gamma)\wedge \phi(x_1,x_2,x_3,x_4)\right)\]}
   
Since $\mathrm{Th}_{\forall\exists}(\Gamma)\subset\mathrm{Th}_{\forall\exists}(G)$, the sentence $\varphi_r$ is true in $G$ as well. It follows that $N(g,r)$ is cyclic. This concludes the proof.
\end{proof4}

\subsection{Generalized Baumslag's lemma} We shall generalize a criterion proved by Baumslag in the case of free groups (see \cite{Bau62} Proposition 1 or \cite{Bau67} Lemma 7) that will be useful to show that some sequences of homomorphisms taking values in a hyperbolic group are discriminating. A proof in the hyperbolic case can be found in \cite{Ol93}, Lemma 2.4. We include a proof for completeness.

If $g$ is an element of infinite order of a hyperbolic group $G$, we denote by $g^{+}$ and $g^{-}$ the attracting and repellings fixed points of $g$ on the boundary $\partial G$ of $G$.

We begin with a preliminary lemma.

\begin{lemme}\label{4}Let $G$ be a hyperbolic group and $S$ a finite generating set of $G$. Let $g,h,x$ be elements of $G$ such that $g$ and $h$ have infinite order. For every $p\geq 1$, denote by $\alpha_p$ a geodesic path between $g^{p}$ and $g^{p-1}$ in $\mathrm{Cay}(G,S)$, and $\beta_p$ a geodesic path between $xh^{p-1}$ and $xh^{p}$ in $\mathrm{Cay}(G,S)$. Denote by $\gamma$ a path joining $1$ and $x$. If $g^{+}\neq x\cdot h^{+}$, then there exist two constants $\lambda$ and $k$ such that for every integers $p\geq 1$ and $q\geq 1$, $w_{p,q}=\alpha_p\cdots\alpha_1\cdot\gamma\cdot\beta_1\cdots\beta_q$ is a $(\lambda,k)$-quasi-geodesic. 
\end{lemme}

\begin{proof}
It is well-known that $p\in\mathbb{Z}\mapsto g^p$ and $p\in\mathbb{Z}\mapsto h^p$ are quasi-geodesics, i.e.\ there exist constants $\lambda_g,k_g$ and $\lambda_h,k_h$ such that for every $p$, $\alpha_p\cdots\alpha_1$ is a $(\lambda_g,k_g)$-quasi-geodesic joining $g^p$ and $1$, and $\beta_1\cdots\beta_p$ is a $(\lambda_h,k_h)$-quasi-geodesic joining $x$ and $xh^p$. An easy computation gives \[\mathrm{length}(w_{p,q})\leq \max(\lambda_g,\lambda_h)(d(1,g^{p}) + d(1,xh^{q}))+(k_g+k_h+\mathrm{length}(\gamma)+\lambda_hd(1,x)).\]Since $g^{+}\neq x\cdot h^{+}$, there exists a constant $C$ such that the Gromov product $(g^p,xh^q)_1$ is less than $C$, that is \[d(1,g^p)+d(1,xh^q)-d(g^p,xh^q)\leq 2C.\]It follows that $\mathrm{length}(w_{p,q})\leq \lambda d(g^p,xh^q)+k$ for some constants $\lambda$ and $k$.
\end{proof}

One can now prove the following criterion generalizing that of Baumslag.

\begin{prop}\label{baumslag}Let $a_0,a_1,\ldots ,a_m, c_1,\ldots ,c_m$ be elements of a hyperbolic group. Let $w(p_1,\ldots,p_m)= a_0 c_1^{p_1} a_1 c_2^{p_2}a_2 \cdots a_{m-1} c_m^{p_m}a_{m}$ with $p_1,\ldots,p_m\geq 0$. Suppose that the two following conditions hold:
\begin{enumerate}
\item every $c_i$ has infinite order;
\item for every $i\in\llbracket 1,m-1\rrbracket$, $c_i^{+}\neq a_i\cdot c_{i+1}^{+}$.
\end{enumerate}
Then there exists a constant $C$ such that $w(p_1,\ldots,p_m)\neq 1$ for every $p_1,\ldots,p_m\geq C$.
\end{prop}

\begin{proof}
We will prove that there exist constants $\lambda,k$ satisfying the following property: for every integer $n$, there exists an integer $p(n)$ such that for every $p_1,\ldots,p_m\geq p(n)$, $w(p_1,\ldots,p_m)$ (viewed as a path in $\mathrm{Cay}(G,S)$, for a given finite generating set $S$) is a local $(\lambda,k,n)$-quasi-geodesic, i.e.\ every subword of $w(p_1,\ldots,p_m)$ whose length is less than $n$ is a $(\lambda,k)$-quasi-geodesic. Then, it will follow from \cite{CDP90} Theorem 1.4 that there exist three constants $L$, $\lambda'$ and $k'$ such that for every $p_1,\ldots,p_m\geq p(\lceil L\rceil)$, $w(p_1,\ldots,p_m)$ is a $(\lambda',k')$-quasi-geodesic. Hence: \[d(1,w(p_1,\ldots,p_m))\geq 1/{\lambda'} (\mathrm{length}(w(p_1,\ldots,p_m))-k').\] 
Moreover $\lim\limits_{p_1,\ldots,p_m \rightarrow +\infty}\mathrm{length}(w(p_1,\ldots,p_m))= +\infty$, so $d(1,w(p_1,\ldots,p_m))\geq 1$ for every $p_1,\ldots,p_m$ large enough. As a consequence, $w(p_1,\ldots,p_m)\neq 1$ for every $p_1,\ldots,p_m$ large enough.

To prove the existence of constants $\lambda$ and $k$, let's look at subwords of $w(p_1,\ldots,p_m)$. For every $n$, every subword of $w(p_1,\ldots,p_m)$ whose length is less that $n$ is a subword of a word of the form $c_i^{k}$ or $c_i^{k}a_i c_{i+1}^{k}$ where $i\in\llbracket 1,m-1\rrbracket$.

Now, it suffices to show that the words above are quasi-geodesic with constants that do not depend on $k$. This follows from Lemma \ref{4}.
\end{proof}

If $G$ is a hyperbolic group, each element $g$ of infinite order is contained in a unique maximal virtually abelian subgroup of $G$, denoted by $M(g)$, namely the stabilizer of the pair of points $\lbrace g^{+},g^{-}\rbrace$. The straightforward following corollary of Proposition \ref{baumslag} is easier to use in practice.

\begin{co}\label{baumslag2}Let $a_0,a_1,\ldots ,a_m$ and $c$ be elements of a hyperbolic group. Let $(\varepsilon_i)$ in $\lbrace -1,+1\rbrace^m$. Let $w(p)= a_0{c} ^{\varepsilon_1 p} a_1 {c} ^{\varepsilon_2 p}a_2 \cdots a_{m-1} {c}^{\varepsilon_m p}a_{m}$ with $p\geq 0$. Suppose that the two following conditions hold:
\begin{enumerate}
\item $c$ has infinite order;
\item for every $i\in\llbracket 1,m-1\rrbracket$, $a_i\notin M(c)$.
\end{enumerate}
Then there exists a constant $C$ such that $w(p)\neq 1$ for $p\geq C$.
\end{co}

\subsection{The canonical JSJ splitting and the modular group}\label{25}

\subsubsection{The canonical JSJ splitting}

\begin{de}\label{FBO}A group $G$ is called a finite-by-orbifold group if it is an extension \[F\rightarrow G \rightarrow \pi_1(O)\]where $O$ is a compact hyperbolic 2-orbifold possibly with boundary, and $F$ is an arbitrary finite group called the fiber. We call extended boundary subgroup of $G$ the preimage in $G$ of a boundary subgroup of $\pi_1(O)$. We define in the same way extended conical subgroups. In the case where $O$ has only conical singularities, i.e.\ has no mirrors, we say that $G$ is a conical finite-by-orbifold group.
\end{de}

\begin{de}\label{QH}A vertex $v$ of a graph of groups is said to be quadratically hanging (denoted by QH) if its stabilizer $G_v$ is a finite-by-orbifold group $F\rightarrow G \rightarrow \pi_1(O)$ such that $O$ has non-empty boundary, and such that any incident edge group is finite or contained in an extended boundary subgroup. We also say that $G_v$ is QH. 
\end{de}

A one-ended finitely generated $K$-$\mathrm{CSA}$ group $G$ has a canonical JSJ splitting over virtually abelian groups (see \cite{GL16}, Theorem 9.14). In the case where all finitely generated abelian subgroups of $G$ are virtually cyclic, $G$ has a canonical JSJ decomposition $\Delta$ over $\mathcal{Z}$, the class of virtually cyclic groups with infinite center. We refer the reader to \cite{GL16}, Section 9.2, for a construction of $\Delta$ using the tree of cylinders. The benefit of this decomposition $\Delta$ is that its QH vertex groups have no mirrors. In the sequel, all one-ended groups will be $K$-$\mathrm{CSA}$ without $\mathbb{Z}^2$, and the expression "the JSJ splitting" or "the $\mathcal{Z}$-JSJ splitting" will always refer to the canonical JSJ splitting over $\mathcal{Z}$ obtained by the tree of cylinders. Below are the properties of $\Delta$ that will be useful in the sequel.
\begin{itemize}
\item[$\bullet$]The graph $\Delta$ is bipartite, with every edge joining a vertex carrying a virtually cyclic group to a vertex carrying a non-virtually-cyclic group.
\item[$\bullet$]There are two kinds of vertices of $\Delta$ carrying a non-cyclic group: rigid ones, and QH ones. If $v$ is a QH vertex of $\Delta$, every incident edge group $G_e$ coincides with an extended boundary subgroup of $G_v$. Moreover, given any extended boundary subgroup $B$ of $G_v$, there exists a unique incident edge $e$ such that $G_e=B$.
\item[$\bullet$]The action of $G$ on the associated Bass-Serre tree $T$ is acylindrical in the following strong sense: if an element $g\in G$ of infinite order fixes a segment of length $\geq 2$ in $T$, then this segment has length exactly 2 and its midpoint has virtually cyclic stabilizer.
\item Let $v$ be a vertex of $T$, and let $e,e'$ be two distinct edges incident to $v$. If $G_v$ is not virtually cyclic, then the group $\langle G_e,G_{e'}\rangle$ is not virtually cyclic.
\end{itemize}

\subsubsection{The modular group}\label{262}

Let $G$ be a one-ended finitely generated $K$-$\mathrm{CSA}$ group that does not contain $\mathbb{Z}^2$. The modular group $\mathrm{Mod}(G)$ of $G$ is the subgroup of $\mathrm{Aut}(G)$ composed of all automorphisms that act by conjugaison on non-QH vertex groups of the $\mathcal{Z}$-JSJ splitting, and on finite subgroups of $G$, and that act trivially on the underlying graph of the $\mathcal{Z}$-JSJ splitting. 

\subsection{Stallings-Dunwoody splitting}\label{SD}

Let $G$ be a finitely generated group. Suppose that there exists a constant $K$ such that every finite subgroup of $G$ has order less than $K$. Then $G$ splits over finite groups as a graph of groups all of whose vertex groups are finite or one-ended. Such a splitting is called a Stallings-Dunwoody splitting of $G$; it is not unique, but the conjugacy classes of one-ended vertex groups do not depend on the splitting. A one-ended subgroup of $G$ that appears as a vertex group of a Stallings-Dunwoody splitting is called a one-ended factor of $G$.

If $G$ is a $\Gamma$-limit group, where $\Gamma$ is hyperbolic, one can prove that there exists a uniform bound on the order of finite subgroups of $G$ (see \cite{RW14} Lemma 1.18). As a consequence, $G$ has a Stallings-Dunwoody splitting.

\subsection{Preliminaries on orbifolds}

In the sequel, all orbifolds are compact, 2-dimensional, hyperbolic and conical (i.e.\ without mirrors).

\subsubsection{Cutting an orbifold into elliptic components}

\begin{de}\label{essential}A set $\mathcal{C}$ of simple closed curves on a conical orbifold is said to be essential if its elements are non null-homotopic, two-sided, non boundary-parallel, pairwise non parallel, and represent elements of infinite order (in other words, no curve of $\mathcal{C}$ circles a singularity). 
\end{de}

\begin{prop}\label{cutting}Let $O$ be a hyperbolic orbifold. Suppose that $S=\pi_1(O)$ acts minimally on a tree $T$ in such a way that its boundary elements are elliptic. Then there exists an essential set $\mathcal{C}$ of curves on $O$, and a surjective $S$-equivariant map $f : T_{\mathcal{C}}\rightarrow T$, where $T_{\mathcal{C}}$ stands for the Bass-Serre tree associated with the splitting of $S$ dual to $\mathcal{C}$. In other words:
\begin{itemize}
\item[$\bullet$] every element of $S$ corresponding to a loop of $\mathcal{C}$ fixes an edge of $T$;
\item[$\bullet$] every fundamental group of a connected component of $O\setminus\mathcal{C}$ is elliptic.
\end{itemize}
\end{prop}

The proposition above is proved in \cite{MS84} for surfaces (see Theorem 3.2.6), and the generalization to compact hyperbolic 2-orbifolds of conical type is straightforward. Then, Proposition \ref{cutting} extends to finite-by-orbifold groups through the following observation: if $G$ is a finite extension $F\hookrightarrow G\twoheadrightarrow \pi_1(O)$ acting minimally on a tree $T$, then the action factors through $G/F\simeq\pi_1(O)$, because $F$ acts as the identity on $T$. Indeed, since $F$ is finite, it fixes a point $x$ of $T$. Since $F$ is normal, the non-empty subtree of $T$ pointwise-fixed by $F$ is invariant under the action of $G$. Since this action is minimal, $F$ fixes $T$ pointwise.

\subsubsection{Non-pinching homomorphisms}

\begin{de}Let $G$ and $G'$ be conical finite-by-orbifold groups. A homomorphism from $G$ to $G'$ is called a morphism of finite-by-orbifold groups if it sends each extended boundary subgroup injectively into an extended boundary subgroup, and if it is injective on finite subgroups.
\end{de}

\begin{de}\label{pinch}Let $G$ be a conical finite-by-orbifold group $F\hookrightarrow G\overset{q}{\twoheadrightarrow} \pi_1(O)$. Let $p$ be a homomorphism from $G$ to a group $G'$. Let $\alpha$ be a two-sided and non-boundary-parallel simple loop on $O$ representing an element of infinite order, and let $C_{\alpha}=q^{-1}(\alpha)\simeq F\rtimes\mathbb{Z}$ (well defined up to conjugacy). The curve $\alpha$ (or $C_{\alpha}$) is said to be pinched by $p$ if $p(C_{\alpha})$ is finite. The homomorphism $p$ is said to be non-pinching if it does not pinch any two-sided simple loop. Otherwise, $p$ is said to be pinching. 
\end{de}

\begin{prop}\label{pinchbis}Let $O$ and $O'$ be conical hyperbolic orbifolds. Let $G$ and $G'$ be their fundamental groups. Let $p:G\rightarrow G'$ be a non-pinching morphism of orbifolds. Suppose that $p(G)$ is not contained in a conical or boundary subgroup of $G'$. Then $[G':p(G)]<+\infty$.
\end{prop}

\begin{proof}
Suppose that $p(G)$ has infinite index in $G'$. Then $p(G)$ is the fundamental group of a geometrically finite hyperbolic conical orbifold of infinite volume. Therefore, $p(G)$ splits as a graph of groups $\Delta$ whose edge groups are trivial, whose vertex groups are conical or boundary subgroups of $G'$, and such that each boundary subgroup of $p(G)$ is elliptic in the Bass-Serre tree $T$ of $\Delta$. This splitting is non-trivial because $p(G)$ is not contained in an extended boundary subgroup or an extended conical subgroup of $G'$. By definition of a morphism of orbifolds, each boundary subgroup of $G$ is contained in a boundary subgroup of $p(G)$, so is elliptic in $T$. Then it follows from Proposition \ref{cutting} that there exists a simple loop on $O$ pinched by $p$. This is a contradiction.
\end{proof}
Let $G$ be the fundamental group of a conical hyperbolic orbifold with boundary. We will associate to $G$ a number, denoted by $k(G)$ and called the complexity of $G$, such that the following proposition holds.

\begin{prop}\label{100}Let $O$ and $O'$ be conical hyperbolic orbifolds, with non-empty boundary. Denote by $G$ and $G'$ their fundamental groups. Let $p:G\rightarrow G'$ be a morphism of orbifolds. If $p(G)$ has finite index in $G'$, then $k(G)\geq k(G')$, with equality if and only if $p$ is an isomorphism.
\end{prop}

We need two definitions.

\begin{de}[Euler characteristic]Let $O$ be a conical orbifold with $m$ conical points of orders $p_1,\ldots ,p_m$, and let $\Sigma$ be the underlying surface. The Euler characteristic of $O$ is\[\chi(O):=\chi(\Sigma)-\sum_{i=1}^{m} \left(1-\dfrac{1}{p_i}\right).\] A compact orbifold is hyperbolic if and only if $\chi(O)<0$.
\end{de}

\begin{de}[Complexity of a virtually free group]Let $G$ be a virtually free group, and let $\Delta$ be a Stallings decomposition of $G$. Denote by $V(\Delta)$ the set of vertices of $\Delta$, and by $E(\Delta)$ its set of edges. We define the complexity $k(\Delta)$ of $\Delta$ as follows: \[k(\Delta):=\sum_{e\in E(\Delta)}\dfrac{1}{\vert G_e\vert}-\sum_{v\in V(\Delta)}\dfrac{1}{\vert G_v\vert}.\]
This number does not depend on $\Delta$. We define the complexity of $G$ as $k(G):=k(\Delta)$ for any decomposition $\Delta$ of $G$ as a graph of groups with finite edge groups and finite vertex groups.
\end{de}

The Euler characteristic and the complexity $k$ defined above are linked through the following lemma.

\begin{lemme}\label{101}Let $O$ be a conical hyperbolic orbifold. Assume that $O$ possesses at least one boundary component, so that $\pi_1(O)$ is virtually free. Then $k(\pi_1(O))=-\chi(O)$.
\end{lemme}

\begin{proof}Let $\mathcal{C}$ be a maximal set of properly embedded arcs in $O$ that are pairwise non parallel and non boundary-parallel. Let $n$ be the cardinality of $\mathcal{C}$. Let $m$ be the number of connected components of $O\setminus\mathcal{C}$, denoted by $O_1,\ldots ,O_m$. Each connected component $O_i$ of $O\setminus\mathcal{C}$ is an annulus or a disc with one singular point. One can see that $\chi(O)$ is equal to $\sum_{1\leq i\leq m}\chi(O_i)-n$, and that $k(\pi_1(O))$ is equal to $\sum_{1\leq i\leq m} k(\pi_1(O_i))+n$ (since $k(A\ast B)=k(A)+k(B)+1$ for any virtually free groups $A$ and $B$). So the lemma follows from the following observation: for each $1\leq i\leq m$, $k(\pi_1(O_i))=-\chi(O_i)$. 
\end{proof}

Now, Proposition \ref{100} follows easily from the lemma below.

\begin{lemme}\label{102}Let $G$ and $G'$ be two free products of finite groups and of a free group. Let $\Delta$ and $\Delta'$ be two Stallings decomposition of $G$ and $G'$ respectively. Assume that edge groups of $\Delta$ and $\Delta'$ are trivial. Let $\phi : G\twoheadrightarrow G'$ be an epimorphism that is injective on finite subgroups of $G$. Then $k(G)\geq k(G')$, with equality if and only if $\phi$ is injective.
\end{lemme}

\begin{proof100}Denote by $d$ the index of $p(G)$ in $G'$. There exists a covering orbifold $O''$ of $O'$ of degree $d$ such that $p(G)\simeq \pi_1(O'')$. As a consequence, $\chi(O'')=d\chi(O')$, so $k(p(G))=dk(G')$ thanks to Lemma \ref{101}. According to Lemma \ref{102}, $k(G)\geq k(p(G))$. Now, assume that $k(G)=k(G')$. Then $d=1$, so $p$ is surjective, and $p$ is injective by Lemma \ref{102}. Hence, $p$ is an isomorphism.
\end{proof100}

\begin{proof102}If $\Delta$ is reduced to a point, then $\phi$ is obviously injective. From now on, we will suppose that $\Delta$ has at least two vertices. Let $T,T'$ be Bass-Serre trees of $\Delta,\Delta'$ respectively. We build a $\phi$-equivariant map $f:T\rightarrow T'$ in the following way: for every vertex $v$ of $T$, there exists a vertex $v'$ of $T'$ such that $\phi(G_v)=G'_{v'}$. Moreover, $v'$ is unique since $\phi$ is injective on finite subgroups, and edge groups of $\Delta'$ are trivial. We let $f(v)=v'$. Next, if $e$ is an edge of $T$, with endpoints $v$ and $w$, there exists a unique path $e'$ from $f(v)$ to $f(w)$ in $T'$. We let $f(e)=e'$. Let us denote by $d'$ the natural distance function on $T'$. Up to subdivising the edges of $T$, we can assume that, for every adjacent vertices $v,w\in T$, $d'(f(v),f(w))\in\lbrace 0,1\rbrace$. We will prove that $\phi$ can be written as a composition $i\circ \pi_n\circ \cdots\circ \pi_0$, with $n\geq 0$, $\pi_0=\mathrm{id}$ and $i$ injective, such that $k(\pi_{\ell}\circ \cdots \circ \pi_0(G))\geq k(\pi_{\ell+1}\circ \cdots \circ \pi_0(G))$ for every $0\leq \ell<n$ (if $n>0$), with equality if and only if $\pi_{\ell+1}=\mathrm{id}$.

\textbf{Step 1.} Assume that there exist two adjacent vertices $v,w$ such that $d'(f(v),f(w))=0$. Let $e$ be the edge between $v$ and $w$. We collapse $e$ in $T$, as well as all its translates under the action of $G$. Collapsing $e$ gives rise to a new vertex $x$ labelled by $\langle G_v,G_w\rangle$ if $v$ and $w$ are not in the same orbit, or $\langle G_v,g\rangle$ if $w=g\cdot v$. Call $S_1$ the resulting tree. Now, let $N$ be the kernel of the restriction of $\phi$ to $G_x$, let $T_1:=S_1/\langle\langle N\rangle\rangle$, $G_1:=G/\langle\langle N\rangle\rangle$ and $h_1: T \twoheadrightarrow T_1$, $\pi_1 : G\twoheadrightarrow G_1$ the associated surjections. The homomorphism $\phi$ factors through $\pi_1$ as $\phi=\phi_1\circ \pi_1$, and the map $f$ factors through $h_1$ as $f=f_1\circ h_1$. Since $T'$ has finite vertex groups, the stabilizer $G_x/N\simeq \phi(G_x)$ of $h_1(x)$ in $T_1$ is finite, so $T_1/G_1$ is a Stallings splitting of $G_1$. Thus, $k(G_1)=k(T/G_1)$ by definition. Let us compare $k(G_1)$ with $k(G)$. It is not hard to see that 
\begin{equation*}
k(G)-k(G_1)= \begin{cases}
             1-\dfrac{1}{\vert G_v\vert}-\dfrac{1}{\vert G_w\vert}+\dfrac{1}{\vert G_x/N\vert} & \text{if} \ v \ \text{and} \ w \ \text{are not in the same orbit}  \\          
             1-\dfrac{1}{\vert G_v\vert}+\dfrac{1}{\vert G_x/N\vert} & \text{if} \ v \ \text{and} \ w \ \text{are in the same orbit} \\
       \end{cases}.
\end{equation*}
Hence, if $v$ and $w$ are in the same orbit, it is clear that $k(G)>k(G_1)$. If $v$ and $w$ are not in the same orbit, there are four distinct cases: if $\vert G_v\vert \geq 2$ and $\vert G_w\vert \geq 2$, it is clear that $k(G)>k(G_1)$; if $\vert G_v\vert =1$ and $\vert G_w\vert \geq 2$ (respectively $\vert G_w\vert =1$ and $\vert G_v\vert \geq 2$), then $G_x=G_w$ (respectively $G_x=G_v$) and $k(G)\geq k(G_1)$ with equality if and only if $N=\lbrace 1\rbrace$, i.e.\ $\pi_1=\mathrm{id}$; if $\vert G_v\vert=\vert G_w\vert =1$, then $G_x=N=\lbrace 1\rbrace$, i.e.\ $\pi_1=\mathrm{id}$.

If the $\phi_1$-equivariant map $f_1 : T_1\rightarrow T'$ collapses some edge, we repeat the previous operation. Since $T$ has only finitely many orbits of edges under the action of $G$, the procedure terminates after finitely many steps. So we can assume, without loss of generality, that $f_1$ sends adjacent vertices on adjacent vertices.

\vspace{1mm}

\textbf{Step 2.} Assume that $f_1$ folds some pair of edges, as pictured below.

\begin{center}
\begin{tikzpicture}[scale=1]
\node[draw,circle, inner sep=1.7pt, fill, label=below:{$w$}] (A1) at (2,0) {};
\node[draw,circle, inner sep=1.7pt, fill, label=below:{${w'}$}] (A2) at (2,2) {};
\node[draw,circle, inner sep=1.7pt, fill, label=below:{$v$}] (A3) at (0,1) {};
\node[draw=none, label=below:{e}] (B1) at (1,0.5) {};
\node[draw=none, label=below:{e'}] (B2) at (1,2.2) {};
\node[draw=none, label=below:{}] (B3) at (7,2) {};
\node[draw,circle, inner sep=1.7pt, fill, label=above:{$f_1(w)=f_1(w')$}] (A4) at (8,1) {};
\node[draw,circle, inner sep=1.7pt, fill, label=below:{$f_1(v)$}] (A5) at (6,1) {};

\draw[-,>=latex] (A3) to (A1) ;
\draw[-,>=latex] (A3) to (A2);
\draw[-,>=latex] (A4) to (A5);
\draw[->,>=latex, dashed] (3,1) to (5,1);
\end{tikzpicture}
\end{center}
Let us fold $e$ and $e'$ together in $T_1$, as well as all their translates under the action of $G_1$. Note that $e$ and $e'$ are not in the same $G_v$-orbit since $T'$ has trivial edge stabilizers and $\phi$ is injective on $G_v$. Folding $e$ and $e'$ together gives rise to a new vertex $x$ labelled by $\langle G_w,G_w'\rangle$ if $w$ and $w'$ are not in the same orbit, or $\langle G_w,g\rangle$ if ${w'}=g\cdot w$. Call $S_2$ the resulting tree. Now, let $N$ be the kernel of the restriction of $\phi_1$ to $(G_1)_x$, let $T_2:=S_2/\langle\langle N\rangle\rangle$, $G_2:=G_1/\langle\langle N\rangle\rangle$ and $h_2: T_1 \twoheadrightarrow T_2$, $\pi_2 : G_1\twoheadrightarrow G_2$ the associated surjections. The homomorphism $\phi_1$ factors through $\pi_2$ as $\phi_1=\phi_2\circ \pi_2$, and the map $f_1$ factors through $h_2$ as $f_1=f_2\circ h_2$. Since $T'$ has finite vertex groups, the stabilizer $(G_1)_x/N\simeq \phi_1((G_1)_x)$ of $h_2(x)$ in $T_2$ is finite, so $T_2/G_2$ is a Stallings splitting of $G_2$. As in the first step, we can see that $k(G_1)>k(G_2)$. Again, we can repeat this operation only finitely many times since $T$ has only finitely many orbits of edges under the action of $G$. At the end, with obvious notations, we get a $\phi_n$-equivariant map $f_n : T_n\rightarrow T'$ that is locally injective, so injective. It remains to prove that $\phi_n$ is injective: if $\phi_n(g)=1$, then for every vertex $v$ of $T_n$, $f_n(gv)=f_n(v)$, so $gv=v$. Since $G_n$ acts on $T_n$ with trivial edge stabilizers, we get $g=1$.\end{proof102}

We deduce the following result from Proposition \ref{pinchbis} and Proposition \ref{100}.

\begin{prop}\label{147}Let $O$ and $O'$ be conical hyperbolic orbifolds, with non-empty boundary. Denote by $G$ and $G'$ their fundamental groups. If $p:G\rightarrow G'$ is a non-pinching morphism of orbifolds such that $p(G)$ is not contained in a conical or boundary subgroup of $G'$, then $k(G)\geq k(G')$, with equality if and only if $p$ is an isomorphism. 
\end{prop}

We will now generalize the proposition above to finite extensions of conical hyperbolic orbifolds. First, we need the following lemma.

\begin{lemme}Let $O$ and $O'$ be conical hyperbolic orbifolds. Let $G$ and $G'$ be two finite extensions $F\hookrightarrow G\twoheadrightarrow \pi_1(O)$ and $F'\hookrightarrow G'\twoheadrightarrow \pi_1(O')$. If $p : G \rightarrow G'$ is a homomorphism whose restriction to $F$ is injective and whose image is infinite, then $p(F)\subset F'$. As a consequence, $p$ induces a homomorphism $q$ from $\pi_1(O)$ to $\pi_1(O')$.

\begin{center}
\begin{tikzcd}
F \arrow[r, hook]  \arrow[d,dashrightarrow]
& S \arrow[r, "\pi" ] \arrow[d,"f"]
& \pi_1(O) \arrow[d,dashrightarrow, "q"] \\
F' \arrow[r, hook]
& S' \arrow[r, "\pi'" ] 
& \pi_1(O') \\
\end{tikzcd}
\end{center}
\end{lemme}

\begin{proof}
Firstly, we make the following observation: if $A$ is a finite subgroup of $G'$ which is not contained in $F'$, then the normalizer $N_{G'}(A)$ of $A$ in $G'$ is finite. Indeed, if $A$ is not contained in $F'$, then $\pi(A)$ is a non-trivial finite subgroup of $O'$, so its normalizer is a finite cyclic group (since we can see $\pi(A)$ as an elliptic subgroup of $\mathrm{PSL}(2,\mathbb{R})$ acting on $\mathbb{H}^{2}$). Now, if $p(F)$ is not contained in $F'$, then $p(G)\subset N_{G'}(p(F))$, which is finite. This is a contradiction.
\end{proof}

\begin{prop}\label{complexity}Let $O$ and $O'$ be conical hyperbolic orbifolds, with non-empty boundary. Let $G$ and $G'$ be two finite extensions $F\hookrightarrow G\twoheadrightarrow \pi_1(O)$ and $F'\hookrightarrow G'\twoheadrightarrow \pi_1(O')$. Let $p:G\rightarrow G'$ be a non-pinching morphism of finite-by-orbifold groups such that $p(G)$ is not contained in an extended conical or boundary subgroup of $G'$. Then $k(G)\geq k(G')$, with equality if and only if $p$ is an isomorphism.
\end{prop}

\begin{proof}
$k(\pi_1(O))=\vert F\vert k(G)$ and $k(\pi_1(O'))=\vert F'\vert k(G')$. By the previous lemma, $p$ induces a morphism of orbifolds $q : \pi_1(O)\rightarrow \pi_1(O')$. According to Proposition \ref{147}, $k(\pi_1(O))\geq k(\pi_1(O'))$, so $\vert F\vert k(G)\geq \vert F'\vert k(G')$. But $p(F)\subset F'$, and $p$ is injective in restriction to $F$, so $\vert F'\vert\geq \vert F\vert$. As a consequence, $k(G)\geq k(G')$. Moreover, if $k(G)=k(G')$, then $k(\pi_1(O))\geq k(\pi_1(O'))$, so $q : \pi_1(O)\rightarrow \pi_1(O')$ is an isomorphism by Proposition \ref{147}, and $p_{\vert F}:F\rightarrow F'$ is an isomorphism, so $p$ is an isomorphism.
\end{proof}

\section{How to extract information from the JSJ using first-order logic}

\subsection{Preliminary examples}

Let $\Gamma$ be a hyperbolic group, and let $G$ be a finitely generated group such that $\mathrm{Th}_{\forall\exists}(\Gamma)=\mathrm{Th}_{\forall\exists}(G)$. For convenience, suppose that $G$ is one-ended. We saw in the previous section that $G$ does not contain $\mathbb{Z}^2$, and has a canonical JSJ splitting $\Delta$ over $\mathcal{Z}$ (the class of virtually cyclic groups with infinite center). Our aim is to prove that $G$ is a hyperbolic group. Since the JSJ splitting is acylindrical, and since the QH vertex groups are hyperbolic by definition, it suffices to prove that all non-QH vertex groups of $\Delta$ are hyperbolic (as a consequence of the Bestvina-Feighn combination theorem, see proposition \ref{BF92}). 

Hence, the main issue is to manage to extract information about non-QH vertex groups of $\Delta$ by means of first-order logic. We shall give two motivating examples. The simplest example arises when $\Gamma$ and $G$ are both rigid (that is, they are one-ended, their canonical JSJ splittings over $\mathcal{Z}$ are reduced to a point, and they are not finite-by-orbifold). 

\begin{ex}Let $\Gamma$ be a hyperbolic group, and let $G$ be a finitely generated group such that $\mathrm{Th}_{\forall\exists}(\Gamma)=\mathrm{Th}_{\forall\exists}(G)$. Suppose that $\Gamma$ and $G$ are rigid (i.e.\ do not split non-trivially over a virtually cyclic group). We shall prove that $\Gamma$ and $G$ are isomorphic. Since $G$ is a $\Gamma$-limit group, there exists a discriminating sequence $(f_n)$ of homomorphisms from $G$ to $\Gamma$ (see Section \ref{22}). By definition of the modular group, $\mathrm{Mod}(G)=\mathrm{Inn}(G)$, so it results easily from Sela's shortening argument \ref{sela2} that $f_n$ is necessarily injective for $n$ large enough. Hence, $G$ embeds into $\Gamma$. Now, using Corollary \ref{simple}, we prove in the same way that $\Gamma$ embeds into $G$. As a one-ended hyperbolic group, $\Gamma$ is co-Hopfian, so $G\simeq\Gamma$.
\end{ex}

The second example below is a little more complicated, and much more instructive. First of all, note that we cannot express the full statement of the shortening argument \ref{simple} in first-order logic, since precomposition by a modular automorphism is not expressible by a first-order formula in general. To deal with this problem, let's consider the following corollary of the shortening argument \ref{simple}, which follows immediately from the definition of the modular group.

\begin{co}\label{simple2}Let $\Gamma$ be a group that embeds into a hyperbolic group, and let $G$ be a one-ended finitely generated group. There exists a finite set $F\subset G\setminus\lbrace 1\rbrace$ with the following property: for every non-injective homomorphism $f : G\rightarrow\Gamma$, there exists a homomorphism $f' : G\rightarrow\Gamma$ such that 
\begin{itemize}
\item[$\bullet$]$\ker(f')\cap F\neq\varnothing$;
\item[$\bullet$]$f$ and $f'$ coincide, up to conjugacy by an element of $\Gamma$, on non-QH vertex groups of the $\mathcal{Z}$-JSJ splitting $\Delta$ of $G$, and on finite subgroups of $G$.
\end{itemize}
We say that $f$ and $f'$ are $\Delta$-related (see Definition \ref{reliés2} below).
\end{co}

Below is our second example.

\begin{ex}\label{exemple}Let $\Gamma$ be a hyperbolic group, and let $G$ be a finitely generated group such that $\mathrm{Th}_{\forall\exists}(\Gamma)=\mathrm{Th}_{\forall\exists}(G)$. Suppose that $\Gamma$ and $G$ are one-ended, and that there is no QH vertex in the JSJ decompositions of $\Gamma$ and $G$. Then $\Gamma$ and $G$ are isomorphic. In particular, $G$ is hyperbolic.
\end{ex}

\begin{proof9}
First, we prove that $G$ embeds into $\Gamma$. Argue by contradiction and suppose that every homomorphism from $G$ to $\Gamma$ is non-injective. By Corollary \ref{simple2}, there exists a finite set $F\subset G$ such that, for every homomorphism $f : G \rightarrow\Gamma$, there exists $f' : G\rightarrow\Gamma$ that kills an element of $F$ and that coincides with $f$ up to conjugacy on every vertex group of the canonical $\mathcal{Z}$-JSJ splitting $\Delta$ of $G$. One easily sees that this fact can be expressed by a $\forall\exists$-sentence verified by $\Gamma$ (see the next section for details). Since $\mathrm{Th}_{\forall\exists}(\Gamma)\subset\mathrm{Th}_{\forall\exists}(G)$, this sentence is also verified by $G$, and its interpretation in $G$ yields the following: for every endomorphism $f : G \rightarrow G$, there exists an endomorphism $f' : G\rightarrow G$ that kills an element of $F$ and that coincides with $f$ up to conjugacy on every vertex group of the canonical $\mathcal{Z}$-JSJ splitting $\Delta$ of $G$. Taking $f=\mathrm{id}_G$, we get an endomorphism of $G$ that kills an element of $F$ and that is inner in restriction to every vertex group. But we will prove further that such an endomorphism is necessarily injective (see Proposition \ref{lemmeperin}). This contradicts the fact that $f'$ kills an element of $F$. Hence, we have shown that $G$ embeds into $\Gamma$. Now, using Corollary \ref{simple2}, we prove in the same way that $\Gamma$ embeds into $G$. As a one-ended hyperbolic group, $\Gamma$ is co-Hopfian, so $G\simeq\Gamma$.
\end{proof9}

New difficulties arise when $\Gamma$ and $G$ are not supposed to be one-ended, or when the canonical $\mathcal{Z}$-JSJ splittings of $\Gamma$ and $G$ contain QH vertices. However, the example above highlights the crucial role played by endomorphims of $G$ that coincide up to conjugacy with the identity of $G$ on non-QH vertex groups. This example also brings out the key idea to obtain these special homomorphisms, due to Sela-Perin, that consists in expressing a consequence of the shortening argument \ref{sela2} by a $\forall\exists$-sentence that $\Gamma$ verifies (assuming $G$ is one-ended). Since $\Gamma$ and $G$ have the same $\forall\exists$-theory, $G$ satisfies this sentence as well, and its interpretation in $G$ provides us with a special endomorphism of $G$. This example leads us to the definition of related homomorphisms.

\subsection{Related homomorphisms}\label{32}

The following definition is similar (but slightly different) to Definition 5.9 and Definition 5.15 of \cite{Per11}.

\begin{de}[Related homomorphisms]\label{reliés2}
\normalfont
Let $G$ be a one-ended finitely generated $K$-$\mathrm{CSA}$ group and let $G'$ be a group. Let $\Delta$ be the canonical JSJ splitting of $G$ over $\mathcal{Z}$. Let $f$ and $f'$ be two homomorphisms from $G$ to $G'$. We say that $f$ and $f'$ are $\Delta$-related if the two following conditions hold:
\begin{itemize}
\item[$\bullet$]for every non-QH vertex $v$ of $\Delta$, there exists an element $g_v\in G'$ such that \[{f'}_{\vert G_v}=\iota_{g_v}\circ f_{\vert G_v};\]
\item[$\bullet$]for every finite subgroup $F$ of $G$, there exists an element $g\in G'$ such that \[{f'}_{\vert F}=\iota_{g}\circ f_{\vert F}.\]
\end{itemize}
Note that being $\Delta$-related is an equivalence relation on $\mathrm{Hom}(G,G')$. 
\end{de}

\begin{rque} 
The second condition above can be reformulated as follows: for every QH vertex $v$, and for every finite subgroup $F$ of $G_v$, there exists an element $g\in G'$ such that $f_{\vert F}=\iota_{g}\circ {f'}_{\vert F}$. Indeed, every finite group has Serre's property (FA), so every finite subgroup $F$ of $G$ is contained in a conjugate of some vertex group $G_w$ of $\Delta$. Furthermore, if $w$ is a non-QH vertex, it follows from the first condition that there exists an element $g\in G'$ such that $f_{\vert F}=\iota_{g}\circ {f'}_{\vert F}$. 
\end{rque}

\begin{de}[Preretraction]\label{pre}
\normalfont
Let $G$ be a finitely generated $K$-$\mathrm{CSA}$ group and let $H$ be a one-ended subgroup of $G$. Let $\Delta$ be the canonical JSJ splitting of $H$ over $\mathcal{Z}$. A preretraction from $H$ to $G$ is a homomorphism $H\rightarrow G$ that is $\Delta$-related to the inclusion of $H$ into $G$.
\end{de}

The following lemma shows that being $\Delta$-related can be expressed in first-order logic.

\begin{lemme}[compare with \cite{Per11} Lemma 5.18]\label{deltarelies}
Let $G$ be a finitely generated group that possesses a canonical $\mathcal{Z}$-JSJ splitting $\Delta$, and let $\lbrace g_1,\ldots ,g_n\rbrace$ be a generating set of $G$. Let $G'$ be a group. There exists an existential formula $\varphi(x_1,\ldots , x_{2n})$ with $2n$ free variables such that, for every $f,f'\in \mathrm{Hom}(G,G')$, $f$ and $f'$ are $\Delta$-related if and only if $G'$ verifies $\varphi\left(f(g_1),\ldots , f(g_n),f'(g_1),\ldots ,f'(g_n)\right).$
\end{lemme}

\begin{proof}
Firstly, remark that there exist finitely many (say $p\geq 1$) conjugacy classes of finite subgroups of QH vertex groups of $\Delta$ (indeed, a QH vertex group possesses finitely many conjugacy classes of finite subgroups, and $\Delta$ has finitely many vertices). Denote by $F_1,\ldots ,F_p$ a system of representatives of those conjugacy classes. Denote by $R_1,\ldots ,R_m$ the non-QH vertex groups of $\Delta$. Remark that these groups are finitely generated since $G$ and the edge groups of $\Delta$ are finitely generated. Denote by $\lbrace A_i\rbrace_{1\leq i\leq p+m}$ the union of $\lbrace F_i\rbrace_{1\leq i\leq p}$ and $\lbrace R_i\rbrace_{1\leq i\leq m}$. For every $i\in\llbracket 1,m+p\rrbracket$, let $\lbrace a_{i,1},\ldots ,a_{i,k_i}\rbrace$ be a finite generating set of $A_i$. For every $i\in\llbracket 1,m+p\rrbracket$ and $j\in\llbracket 1,k_i\rrbracket$, there exists a word $w_{i,j}$ in $n$ letters such that $a_{i,j}=w_{i,j}(g_1,\ldots ,g_n)$. Let \[\psi(x_1,\dots,x_{2n}):\exists u_1\ldots\exists u_m \bigwedge_{i=1}^{m+p}\bigwedge_{j=1}^{k_i}w_{i,j}(x_1,\ldots ,x_n)=u_iw_{i,j}(x_{n+1},\ldots ,x_{2n}){u_i}^{-1}. \]

Since $f(a_{i,j})=w_{i,j}(f(g_1),\ldots ,f(g_n))$ and $f'(a_{i,j})=w_{i,j}\left(f'(g_1),\ldots ,f'(g_n)\right)$ for every $i\in\llbracket 1,m+p\rrbracket$ and $j\in\llbracket 1,k_i\rrbracket$, the homomorphisms $f$ and $f'$ are $\Delta$-related if and only if the sentence $\psi\left(f(g_1),\ldots , f(g_n),f'(g_1),\ldots ,f'(g_n)\right)$ is satisfied by $G'$.
\end{proof}

\subsection{Centered graph of groups}

We need to define relatedness in a more general context. In order to deal with groups that are not assumed to be one-ended, we define below the notion of a centered graph of groups. We denote by $\overline{\mathcal{Z}}$ the class of groups that are either finite or virtually cyclic with infinite center.

\begin{de}[Centered graph of groups]\label{graphecentre}A graph of groups over $\overline{\mathcal{Z}}$, with at least two vertices, is said to be centered if the following conditions hold:
\begin{itemize}
\item[$\bullet$]the underlying graph is bipartite, with a QH vertex $v$ such that every vertex different from $v$ is adjacent to $v$;
\item[$\bullet$]every incident edge group $G_e$ coincides with an extended boundary subgroup or with an extended conical subgroup of $G_v$ (see Definition \ref{FBO});
\item[$\bullet$]given any extended boundary subgroup $B$, there exists a unique incident edge $e$ such that $G_e$ is conjugate to $B$ in $G_v$;
\item[$\bullet$]if an element of infinite order fixes a segment of length $\geq 2$ in the Bass-Serre tree of the splitting, then this segment has length exactly 2 and its endpoints are translates of $v$.
\end{itemize}
The vertex $v$ is called the central vertex.
\end{de}

\begin{figure}[!h]
\includegraphics[scale=0.4]{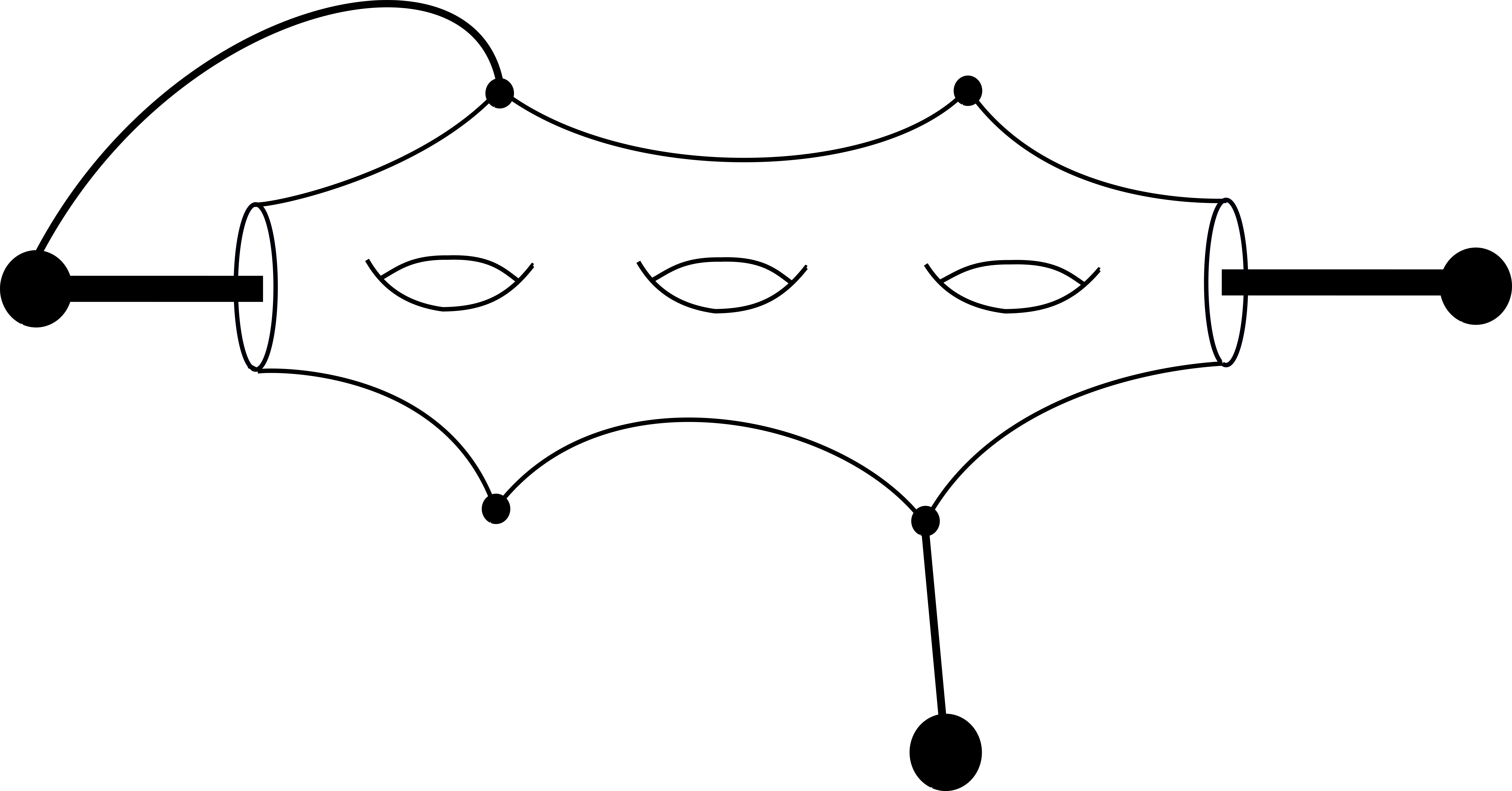}
\caption{A centered graph of groups. Edges with infinite stabilizer are depicted in bold.}
\end{figure}

The following proposition will be crucial in the sequel.

\begin{prop}\label{hyperbolicité}Let $G$ be a group that splits as a centered graph of groups, with central vertex $v$. Then the following assertions are equivalent.
\begin{itemize}
\item[$\bullet$]$G$ is hyperbolic
\item[$\bullet$]For every vertex $w\neq v$, $G_w$ is hyperbolic.
\end{itemize}
\end{prop}

This is an easy consequence of the two following well-known results.

\begin{prop}[\cite{Bow98}, Proposition 1.2]\label{bow}If a hyperbolic group splits over quasi-convex subgroups, then every vertex group is quasi-convex (so hyperbolic).
\end{prop}

Recall that a subgroup $H$ of a group $G$ is called almost malnormal if $H\cap gHg^{-1}$ is finite for every $g$ in $G\setminus H$. Note that edge groups in a centered graph of groups are almost malnormal in the central vertex group.

\begin{prop}[\cite{BF92}, corollary of the combination theorem]\label{BF92}
~\

Let $G=A\ast_CB$ be an amalgamated product such that $A$ and $B$ are hyperbolic and $C$ is virtually cyclic and almost malnormal in $A$ or $B$. Then $G$ is hyperbolic. 

Let $G=A\ast_C$ be an HNN extension such that $A$ is hyperbolic and $C$ is virtually cyclic and almost malnormal in $A$. Then $G$ is hyperbolic.
\end{prop}

We also define relatedness for centered splittings.

\begin{de}[Related homomorphisms]\label{reliés}
\normalfont
Let $G$ and $G'$ be two groups. Suppose that $G$ possesses a centered splitting $\Delta$, with central vertex $v$. Let $f$ and $f'$ be two homomorphisms from $G$ to $G'$. We say that $f$ and $f'$ are $\Delta$-related if the two following conditions hold:
\begin{itemize}
\item[$\bullet$]for every vertex $w\neq v$, there exists an element $g_w\in G'$ such that \[{f'}_{\vert G_w}=\iota_{g_w}\circ f_{\vert G_w};\]
\item[$\bullet$]for every finite subgroup $F$ of $G$, there exists an element $g\in G'$ such that \[{f'}_{\vert F}=\iota_{g}\circ f_{\vert F}.\]
\end{itemize}
Note that being $\Delta$-related is an equivalence relation on $\mathrm{Hom}(G,G')$. 
\end{de}

\begin{rque}Let $G$ be a one-ended finitely generated $K$-$\mathrm{CSA}$ group and let $\Delta$ be the canonical JSJ splitting of $G$ over $\mathcal{Z}$. Let $v$ a be a QH vertex of $\Delta$, and let $\mathrm{star}(v)$ be the subgraph of $\Delta$ composed of $v$ and all its incident edges. Let us denote by $\Lambda$ the splitting of $G$ obtained by collapsing to a point each connected component of the complement of $\mathrm{star}(v)$ in $\Delta$. This new splitting is obviously a centered splitting of $G$. Moreover, if two homomorphisms are $\Delta$-related, then they are $\Lambda$-related.
\end{rque}

\section{Torsion-saturated groups}\label{section 4}

In the presence of torsion, the universal theory of a hyperbolic group is not closed under HNN extensions and amalgamated free products over finite groups. This is a simple example: the universal sentence $\forall x\forall y \ (x^2=1)\Rightarrow (xy=yx)$, which means that any involution is central, is verified by $F_2\times\mathbb{Z}/2\mathbb{Z}$ but not by $(F_2\times\mathbb{Z}/2\mathbb{Z})\ast_{\lbrace 1\rbrace}=(F_2\times\mathbb{Z}/2\mathbb{Z})\ast\mathbb{Z}$. For the same reason, $(F_2\times\mathbb{Z}/2\mathbb{Z})\ast(F_2\times\mathbb{Z}/2\mathbb{Z})$ does not have the same universal theory as $F_2\times\mathbb{Z}/2\mathbb{Z}$.

To control this phenomenon, we will prove that every hyperbolic group $G$ embeds into a hyperbolic group $\overline{G}$ that possesses the following property: the class of $\overline{G}$-limit groups is closed under HNN extensions and amalgamated free products over finite groups.

\begin{de}\label{sat0}We say that a group $G$ is torsion-saturated if the two following conditions hold.
\begin{enumerate}
\item For every isomorphism $\alpha : F_1\rightarrow F_2$ between finite subgroups $F_1,F_2$ of $G$, there exists an element $g\in G$ such that $g xg^{-1}=\alpha(x)$ for every $x\in F_1$.
\item For every finite subgroup $F$ of $G$, there exists an infinite subset $\lbrace g_1,g_2,\ldots\rbrace$ of $G$ such that $g_n$ has infinite order, $M(g_n)=\langle g_n\rangle\times F$ for every $n$, and $M(g_n)\neq M(g_m)$ whenever $n\neq m$.
\end{enumerate}
\end{de}

\begin{rque}The second condition is equivalent to the following: there exist two elements $a,b\in G$ generating a free group, and such that $M(a)=\langle a\rangle\times F$ and $M(b)=\langle b\rangle\times F$.
\end{rque}

We shall see that every hyperbolic group embeds into a torsion-saturated hyperbolic group (see Theorem \ref{sat2} below). The main interest of torsion-saturated groups resides in the following result.

\begin{te}\label{sat}Let $G$ be a torsion-satured hyperbolic group. Then the class of $G$-limit groups is closed under HNN extensions and amalgamated free products over finite groups.
\end{te}

\begin{proof}
Let $H$ be a $G$-limit group. Let $\alpha: F_1\rightarrow F_2$ be an isomorphism between two finite subgroups $F_1,F_2$ of $H$. We shall prove that the HNN extension $H\ast_{\alpha}$ is a $G$-limit group. According to Lemma \ref{lemme3} below, there exists an isomorphism $\beta: F'_1\rightarrow F'_2$ between finite subgroups $F'_1,F'_2$ of $G$ such that $H\ast_{\alpha}$ is a $G\ast_{\beta}$-limit group. But $G\ast_{\beta}$ is a $G$-limit group thanks to Lemma \ref{lemme5} below, so $H\ast_{\alpha}$ is a $G$-limit group. It remains to treat the case of an amalgamated free product. Let $A,B$ be $G$-limit groups. Let $F$ be a finite group that embeds into $A$ and $B$. According to Lemma \ref{lemme4} below, there exists a finite subgroup $F'$ of $G$ such that $A\ast_FB$ is a $G\ast_{F'}$-limit group, and $G\ast_{F'}$ is a $G$-limit group by Lemma \ref{lemme5}, so $A\ast_FB$ is a $G$-limit group as well.
\end{proof}

\begin{lemme}\label{lemme3}Let $G$ be a hyperbolic group and let $H$ be a $G$-limit group. Let $\alpha: F_1\rightarrow F_2$ be an isomorphism between finite subgroups $F_1,F_2$ of $H$. Then there exists an isomorphism $\beta: F'_1\rightarrow F'_2$ between finite subgroups $F'_1,F'_2$ of $G$ such that $H\ast_{\alpha}$ is a $G\ast_{\beta}$-limit group.
\end{lemme}

\begin{lemme}\label{lemme4}Let $G$ be a hyperbolic group. Suppose that the first condition of Definition \ref{sat0} holds, that is: for every isomorphism $\alpha : F_1\rightarrow F_2$ between finite subgroups $F_1,F_2$ of $G$, there exists an element $g\in G$ such that $g xg^{-1}=\alpha(x)$ for every $x\in F_1$. Let $A,B$ be $G$-limit groups. Let $F$ be a finite group that embeds into $A$ and $B$. Then there exists a finite subgroup $F'$ of $G$ such that $A\ast_FB$ is a $G\ast_{F'}$-limit group.
\end{lemme}

\begin{lemme}\label{lemme5}Let $G$ be a hyperbolic group, and let $\alpha : F_1\rightarrow F_2$ be an isomorphism between finite subgroups $F_1,F_2$ of $G$. Suppose that the two following conditions hold:
\begin{enumerate}
\item there exists an element $g\in G$ such that $g xg^{-1}=\alpha(x)$ for every $x\in F_1$;
\item there exists an infinite subset $E=\lbrace g_1,g_2,\ldots \rbrace\subset G$ such that $g_n$ has infinite order, $M(g_n)=\langle g_n\rangle\times F_1$ for every $n$, and $M(g_n)\neq M(g_m)$ whenever $n\neq m$.
\end{enumerate}
Then $G\ast_{\alpha}$ is a $G$-limit group.
\end{lemme}

\begin{proof11}
Let $(\phi_n)$ be a discriminating sequence of homomorphisms from $H$ to $G$. Without loss of generality, we can suppose that $\phi_n$ is injective in restriction to $F_1$ and $F_2$. We will construct a discriminating sequence $(\rho_n)$ of homomorphisms from $H\ast_{\alpha}$ to $G\ast_{\beta}$, for some $\beta$ that will be defined below. Since $G$ has finitely many conjugacy classes of finite subgroups, there exist two finite subgroups $F'_1,F'_2$ of $G$ such that, up to extracting a subsequence, $\phi_n(F_1)$ is conjugate to $F'_1$ and $\phi_n(F_2)$ is conjugate to $F'_2$ for every $n$. Up to composing $\phi_n$ by an inner automorphism of $G$, one can assume that $\phi_n(F_1)=F'_1$ and $\iota_{g_n}\circ\phi_n(F_2)=F'_2$ for some $g_n\in G$. Denote by $\beta_n$ the isomorphism from $F'_1$ to $F'_2$ making the following diagram commute:

\begin{center}
\begin{tikzcd}
F_1 \arrow[rr, "{\phi_n}_{\vert F_1}"] \arrow[d, "\alpha"]
& & F'_1 \arrow[d, "\beta_n" ] \\
F_2 \arrow[rr, "\iota_{g_n}\circ{\phi_n}_{\vert F_2}" ]
& & F'_2
\end{tikzcd}
\end{center}
Since $\mathrm{Isom}(F'_1,F'_2)$ is finite, there exists an isomorphism $\beta$ between $F'_1$ and $F'_2$ such that (up to extracting a subsequence) $\beta_n =\beta$ for every $n$. Let $t$ and $u$ be the stable letters of $H\ast_{\alpha}$ and $G\ast_{\beta}$, i.e.\ $txt^{-1}=\alpha(x)$ for all $x\in F_1$ and $uyu^{-1}=\beta(y)$ for all $y\in F'_1$. For every $n$, we define a map $\rho_n$ from $H\ast_{\alpha}$ to $G\ast_{\beta}$ as follows:

\begin{equation*}
\rho_n(x) = \begin{cases}
             \phi_n(x) & \text{if} \ x\in H \\
            g_n^{-1}u & \text{if} \ x=t
       \end{cases}.
\end{equation*} 
The map $\rho_n$ clearly extends to a homomorphism since the diagram commutes, and we claim that the sequence $(\rho_n)$ is discriminating. Let $x$ be a non-trivial element of $H\ast_{\alpha}$. If $x$ lies in $H$, it is obvious that $\rho_n(x)$ is non-trivial for every $n$ large enough since $(\phi_n)$ is discriminating. Assume now that $x\notin H$. Then $x$ can be written in reduced form as $x=h_0t^{\varepsilon_1}h_1t^{\varepsilon_2}h_2\cdots$ with $n>0$, $h_i\in H$, $\varepsilon_i=\pm 1$, $h_i\notin F_1$ if $\varepsilon_i=-\varepsilon_{i+1}=1$ and $h_i\notin F_2$ if $\varepsilon_i=-\varepsilon_{i+1}=-1$. One has:\[\rho_n(x)=\phi_n(h_0){(g_nu)}^{\varepsilon_1}\phi_n(h_1){(g_nu)}^{\varepsilon_2}\phi_n(h_2)\cdots  .\]
One has to prove that $\rho_n(x)\neq 1$ for every $n$ large enough. In order to apply Britton's lemma, one verifies that, for every subword of $w$ of the form $uvu^{-1}$ with $v$ not involving $u$, $v$ does not lie in $F'_1$, and that for every subword of $w$ of the form $u^{-1}vu$ with $v$ not involving $u$, $v$ does not lie in $F'_2$.
\begin{itemize}
\item[$\bullet$]If $uvu^{-1}$ is a subword of $w$ with $v$ not involving $u$, then $v$ is of the form $\phi_n(h_i)$ with $h_i\notin F_1$, and $\phi_n(h_i)\notin F'_1$ for every $n$ large enough because $(\phi_n)$ is discriminating.
\item[$\bullet$]Similarly, if $u^{-1}vu$ is a subword of $w$ with $v$ not involving $u$, then $v$ is of the form ${g_n}\phi_n(h_i) g_n^{-1}$ with $h_i\notin F_2$, and $g_n\phi_n(h_i) g_n^{-1}\notin F'_2=g_n\phi_n(F_2) g_n^{-1}$ for every $n$ large enough because $(\iota_{g_n}\circ\phi_n)$ is discriminating.
\end{itemize}
Hence, it follows from Britton's lemma that $\rho_n(x)\neq 1$ for every $n$ large enough.\end{proof11}

\begin{rque}Note that in the previous proof, the hypothesis that $G$ is hyperbolic is only used to ensure that $G$ has finitely many classes of finite subgroups.
\end{rque}

\begin{proof12}
Let $(\phi_n)$ be a discriminating sequence from $A$ to $G$, and let $(\psi_n)$ be a discriminating sequence from $B$ to $G$. For every $n$ large enough, $\phi_n(F)$ and $\psi_n(F)$ are isomorphic to $F$. By hypothesis, there exists an element $g_n\in G$ making the following diagram commute:

\begin{center}
\begin{tikzcd}
 & & \phi_n(F) \arrow[dd, "\iota_{g_n}" ] \\
F \arrow[rru, "{\phi_n}"] \arrow[rrd, "{\psi_n}"'] & &  \\
 & & \psi_n(F)
\end{tikzcd}
\end{center}
Let $\chi_n=\iota_{g_n^{-1}}\circ \psi_n$. So $\chi_n$ and $\phi_n$ coincide on $F$. Since $(\phi_n)$ and $(\chi_n)$ are discriminating, we can assume that $\phi_n$ and $\chi_n$ are injective on $F$. Since $G$ has only finitely many conjugacy classes of finite subgroups, we can assume that $\phi_n(F)=\chi_n(F)=F'$ for every $n$, for some finite subgroup $F'$ of $G$, up to extracting a subsequence and precomposing by an inner automorphism.

Let $G\ast_{F'}=\langle G, t \ \vert \ [t,x]=1 \ \forall x\in F' \rangle$ be the HNN-extension of $G$ over the identity of $F'$. For every $n$, we define a homomorphism $\rho_n$ from $A\ast_FB$ to $G\ast_{F'}$ as follows:
\begin{equation*}
\rho_n(x) = \begin{cases}
             \phi_n(x) & \text{if} \ x\in A \\
             \iota_{t}\circ \chi_n(x) & \text{if} \ x\in B
       \end{cases}.
\end{equation*} 
We claim that the sequence $(\rho_n)$ is discriminating, i.e.\ that for every non-trivial element $x\in A\ast_FB$, $\rho_n(x)\neq 1$ for every $n$ large enough. If $x\in F\setminus\lbrace 1\rbrace$, then $\rho_n(x)=\phi_n(x)$, so $\rho_n(x)$ is non-trivial for every $n$ large enough since $(\phi_n)$ is discriminating. Assume now that $x\notin F$. Then $x$ can be written in a reduced form $a_1b_1a_2b_2\cdots a_kb_k$ with $a_i\in A\setminus F$ and $b_i\in B\setminus F$ (except maybe $a_1$ and $b_k$). The following holds: \[\rho_{n}(x)=\phi_n(a_1)t\chi_n(b_1)t^{-1}\phi_n(a_2)t\chi_n(b_2)t^{-1}\cdots \phi_n(a_k)t\chi_n(b_k)t^{-1}.\]
Since $(\phi_n)$ and $(\chi_n)$ are discriminating, $\phi_n(a_i)\notin F'$ and $\chi_n(b_i)\notin F'$ for every $n$ large enough  (except maybe $\phi_n(a_1)$ and $\chi_n(b_k)$). Hence, it follows from Britton's lemma that $\rho_n(x)\neq 1$ for every $n$ large enough.\end{proof12}

\begin{proof13}
Let $G=\langle S \ \vert \ R\rangle$ be a presentation of $G$, and let \[G\ast_{\alpha}=\langle S,t \ \vert \ R, \ txt^{-1}=\alpha(x) \ \forall x\in F_1 \rangle\] be a presentation of $G\ast_{\alpha}$. By hypothesis, there exists an element $g\in G$ such that $g xg^{-1}=\alpha(x)$ for every $x\in F_1$. Up to replacing $t$ by $g^{-1}t$, we can assume that $F_1=F_2$ and that $\alpha$ is the identity of $F_1$. The presentation of $G\ast_{\alpha}$ becomes $\langle S,t \ \vert \ R, \ [t,x]=1 \ \forall x\in F_1 \rangle$.

For every $g_n\in E$, and for every integer $p$, we define a map $\phi_{p,n}$ from $G\ast_{\alpha}$ to $G$ by 
\begin{equation*}
\phi_{p,n}:\begin{cases}
             z\mapsto z & \text{if} \ z\in G \\
             t\mapsto g_n^p
       \end{cases}.
\end{equation*}
The map $\phi_{p,n}$ clearly extends to a homomorphism since  $\iota_t$ and $\iota_{g_n^p}$ coincide on $F_1$ for every $p$, because $M(g_n)=\langle g_n\rangle\times F_1$ by hypothesis.

Denote by $B_m$ the ball of radius $m$ in $G\ast_{\alpha}$ (for a given generating set). We shall prove the existence of two sequences $(n_m)\in\mathbb{N}^{\mathbb{N}}$ and $(p_{n_m})\in\mathbb{N}^{\mathbb{N}}$ such that $\phi_{p_{n_m},n_m}(x)\neq 1$ for every $x\in B_m\setminus\lbrace 1\rbrace$, for every $m$. Let $x\in B_m\setminus \lbrace 1\rbrace$. If $x$ lies in $G$, $\phi_{p,n}(x)\neq 1$ for all $p$ and $n$. Assume now that $x$ does not belong to $G$. Then $x$ can be written in a reduced form as \[x=y_{0}t^{\varepsilon_{1}}y_{1}t^{\varepsilon_{2}}\cdots t^{\varepsilon_{k}}y_{k}\] with $k>0$, $\varepsilon_i=\pm 1$, $y_{i}\notin F_1$ if $\varepsilon_{i}=-\varepsilon_{i+1}$. We claim that for $p$ and $n$ sufficiently large, the homomorphism $\phi_{p,n}$ verifies $\phi_{p,n}(x)\neq 1$. In order to prove this, we will use Baumslag's lemma \ref{baumslag2} with $c=g_n$. We have\[\phi_{p,n}(x)=y_{0} g_n^{\varepsilon_{1}p}y_{1} g_n^{\varepsilon_{2}p}\cdots g_n^{\varepsilon_{k}p}y_{k}.\]

First, let us rewrite $\phi_{p,n}(x)$ under a more convenient form. Let $i\in\llbracket 1,k-1\rrbracket$, and suppose that $\varepsilon_i=\varepsilon_{i+1}=1$, and that $y_i$ lies in $F_1$. Then, we replace the subword $g_n^p y_i g_n^py_{i+1}$ by $g_n^{2p}y_iy_{i+1}$. In the case where $\varepsilon_{i+2}=-1$, note that $y_iy_{i+1}$ does not belong to $F_1$, since $y_i\in F_1$ and $y_{i+1}\notin F_1$. In the case where $\varepsilon_{i+2}=1$, we repeat the previous operation, and so on. Similarly, if $\varepsilon_i=\varepsilon_{i+1}=-1$ and $y_i$ lies in $F_1$, we replace the subword $g_n^{-p}y_i g_n^{-p}y_{i+1}$ of $\phi_{p,n}$ by $g_n^{-2p}y_iy_{i+1}$, and so on. At the end of this process, we have\[\phi_{p,n}(x)=z_{0} g_n^{\varepsilon_{1}n_1p}z_{1} g_n^{\varepsilon_{2}n_2p}\cdots g_n^{\varepsilon_{\ell}n_{\ell}p}z_{\ell},\]with $n_1,\ldots,n_{\ell}\in\mathbb{N}^{\ast}$, and $z_{i}\notin F_1$ for every $i\in\llbracket 1,\ell-1\rrbracket$.

We can now use Baumslag's lemma \ref{baumslag2} with $c=g_n$. We claim that there exists $n(x)$ such that $z_i$ does not belong to $M(g_n)$ for every $n\geq n(x)$ and for every $i\in\llbracket 1,\ell-1\rrbracket$. Indeed, suppose that $z_i$ lies in $M(g_n)=\langle g_n\rangle\times F_1$, for some $n$. Since $z_i$ does not belong to $F_1$, it has infinite order, so $M(z_i)$ is well-defined and $M(z_i)=M(g_n)$. Then the claim follows from the fact that $M(g_m)\neq M(g_n)$ if $n\neq m$ (by hypothesis).

Let $n_m:=\max\lbrace n(x) \ \vert \ x\in B_m\rbrace$. By Corollary \ref{baumslag2}, there exists an integer $p_{n_m}$ such that $\ker(\phi_{p_{n_m},n_m})\cap B_m=\lbrace 1\rbrace$. This concludes the proof.
\end{proof13}

We conclude this section by proving that every hyperbolic group embeds into a torsion-saturated hyperbolic group.

\begin{te}\label{sat2}Every hyperbolic group embeds into a torsion-saturated hyperbolic group.
\end{te}

\begin{proof}Let $G$ be a hyperbolic group and denote by $F_1,\ldots ,F_m$ a system of representatives of the conjugacy classes of finite subgroups of $G$. We build a graph of groups as follows: begin with a vertex $v$ labelled by $G$, then for every pair $\lbrace i,j\rbrace\subset \llbracket 1,m\rrbracket$ (including $\lbrace i,i\rbrace$), if $F_i$ and $F_j$ are isomorphic, then for every isomorphism $\alpha\in\mathrm{Isom}(F_i,F_j)$ add two edges from $v$ to itself labelled by $\alpha$. Denote by $\overline{G}$ the fundamental group of this graph of groups. It is hyperbolic by Bestvina-Feighn's combination theorem \cite{BF92}. Let us prove that it is torsion-saturated. By definition, for every isomorphism $\alpha : F_1 \rightarrow F_2$ between finite subgroups of $\overline{G}$, there exists an element $g$ of $\overline{G}$ such that $gxg^{-1}=\alpha(x)$, for all $x\in F_1$. Hence the first condition of Definition \ref{sat0} is satisfied by $\overline{G}$. It remains to verify that the second condition holds. Let $F$ be a finite subgroup of $\overline{G}$. We can assume that $F=F_i$ for some $1\leq i\leq m$. Let $a$ and $b$ be the two stable letters associated with the two HNN extensions over the identity of $F_i$. The group $\langle a,b\rangle$ is free, and $M(a)=\langle a\rangle\times F$, $M(b)=\langle b\rangle\times F$. This concludes the proof.
\end{proof}

\section{Quasi-floors and quasi-towers}\label{section5}

Hyperbolic floors and hyperbolic towers have been introduced by Sela in \cite{Sel01} to solve Tarski's problem about the elementary equivalence of free groups (see also Kharlampovich and Myasnikov's NTQ groups). Since we want to deal with torsion, we need new definitions. We introduce below quasi-floors and quasi-towers. 

\subsection{Definitions}

\begin{de}[Quasi-floor]\label{quasi-floor}Let $G$ and $H$ be two groups. Let $\Delta$ be a centered splitting of $G$. Let $V_G$ be the set of vertices of $\Delta$, and $v$ the central vertex. Suppose that $H$ splits as a graph of groups with finite edge groups, and denote by $V_H$ the set of vertices of this splitting. We say that $G$ is a quasi-floor over $H$ if there exist two homomorphisms $r : G \rightarrow H$ and $j : H \rightarrow G$, a partition $V_H=V_H^1\sqcup V_H^2$ and a bijection $s : V_G\setminus \lbrace v\rbrace \rightarrow V_H^1$ such that the following conditions hold:
\begin{itemize}
\item[$\bullet$]$j\circ r$ is $\Delta$-related to the identity of $G$;
\item[$\bullet$]for every $w$ in $V_G\setminus \lbrace v\rbrace$, $r(G_w)={H_{s(w)}}$;
\item[$\bullet$]for every $u\in V_H^2$, $H_u$ is finite and $j$ is injective on $H_u$.
\end{itemize}
If, moreover, there exists a one-ended subgroup $A$ of $G$ such that $A\cap \ker(r)\neq \lbrace 1\rbrace$, the quasi-floor is said to be strict.
\end{de}

\begin{rque}\label{rem}Recall that the first condition means that $j\circ r$ is inner on every vertex group $G_w$, with $w\neq v$, and on every finite subgroup of $G_v$ (see Definition \ref{reliés}). It is not hard to see that the three conditions in the previous definition imply the following:
\begin{itemize}
\item[$\bullet$]$j$ is injective in restriction to any one-ended subgroup of $H$;
\item[$\bullet$]$r\circ j$ is inner on every $H_u$, with $u\in V_H^1$;
\item[$\bullet$]for every $u$ in $V_H^1$, there exists $g_u\in G$ such that $j(H_u)=G_{s^{-1}(u)}^{g_u}$.
\end{itemize}
Note in particular that $r$ sends $G_w$ isomorphically on $H_{s(w)}$ (for every $w\neq v$), and that $j$ sends $H_u$ isomorphically on a conjugate of $G_{s^{-1}(u)}$ (for every $u\in V_H^1$). 
\end{rque}

\begin{figure}[h!]
\includegraphics[scale=0.3]{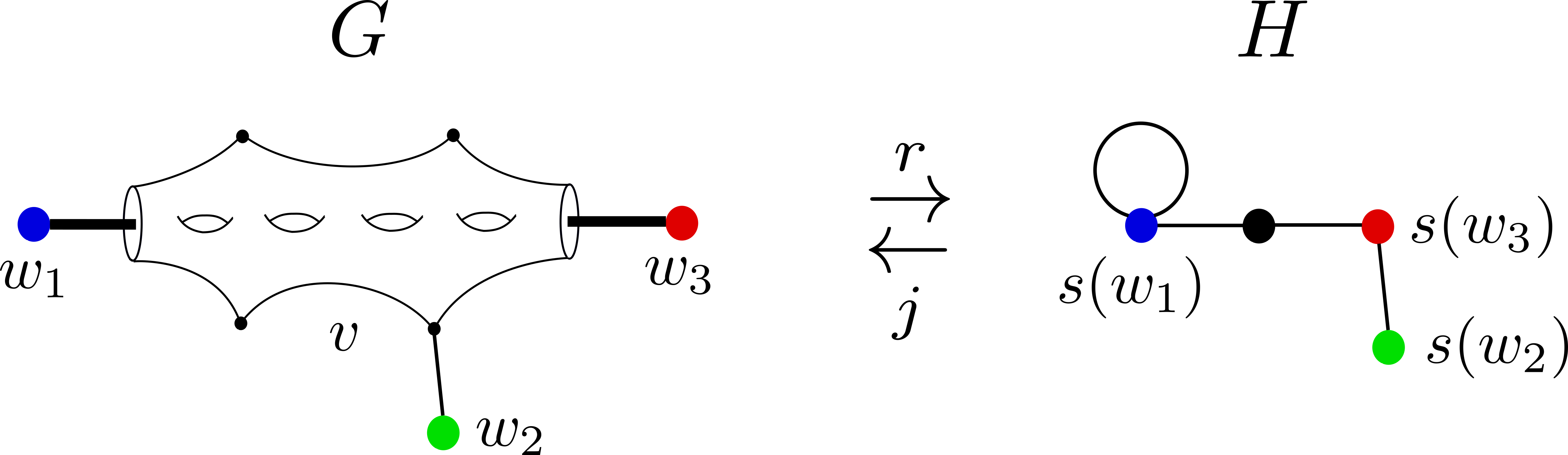}
\caption{$G$ is a quasi-floor over $H$. On this figure, $\vert V_H^1\vert =3$ and $\vert V_H^2\vert =1$. The black vertex (on the right) is labelled by a finite group. Edges with infinite stabilizer are depicted in bold.}
\end{figure}

Sometimes it is convenient to think of $G$ and $H$ as subgroups of a bigger group $G'$ that retracts onto $H$ via an epimorphism $\rho : G' \rightarrow H$ such that $\rho_{\vert G}=r$, as illustrated below. However, it should be noted that the existence of $\rho$ does not guarantee the existence of a homomorphism $j : H\rightarrow G$ as in the previous definition.

\begin{figure}[!h]
\centering
\includegraphics[scale=0.02]{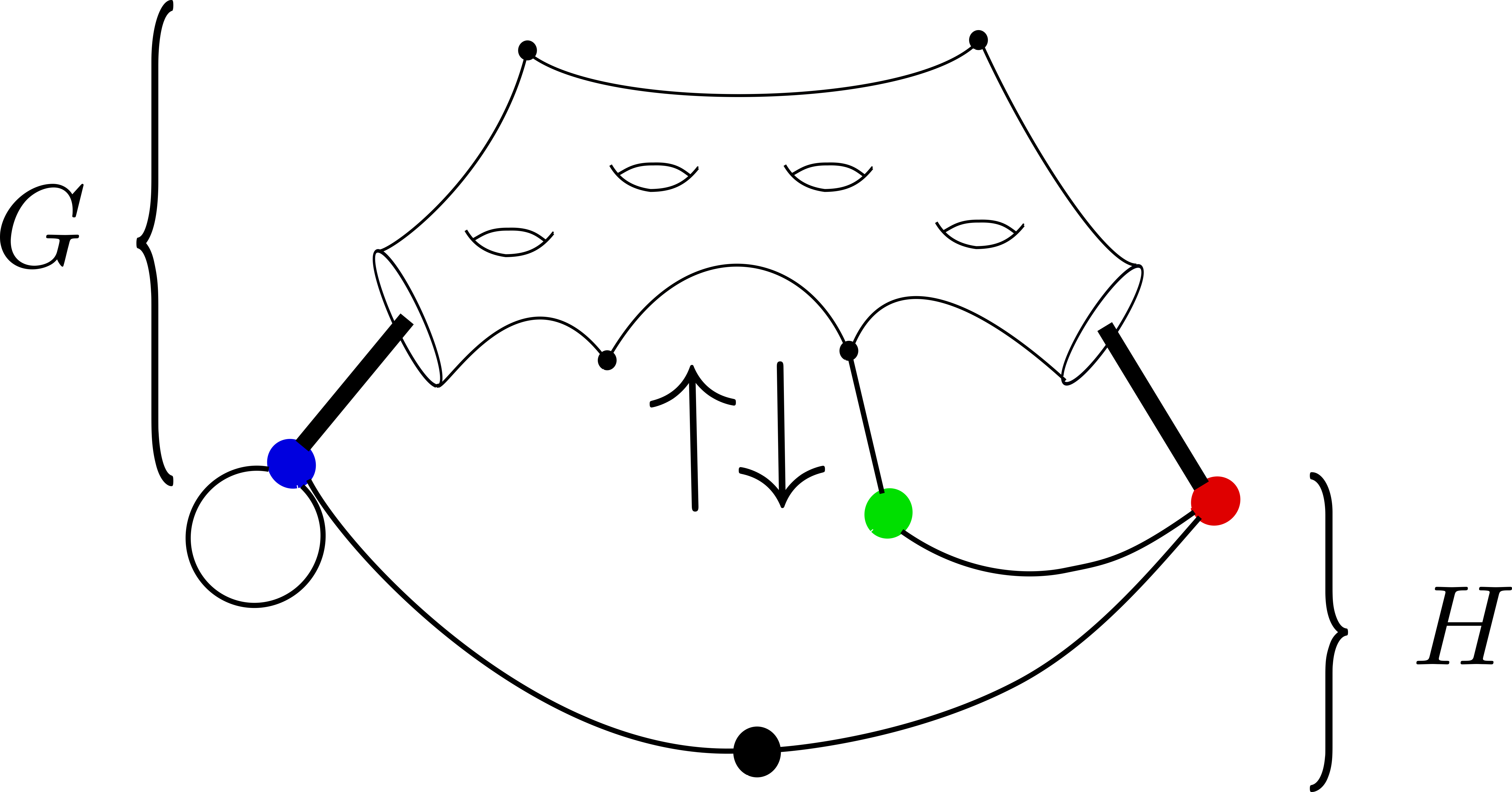}
\caption{The groups $G$ and $H$ can be viewed as subgroups of a bigger group that retracts onto $H$. Edges with infinite stabilizer are depicted in bold.}
\label{cassimple2}
\end{figure}

We stress that a hyperbolic floor in the sense of Sela (see \cite{Per11}, Definition 5.4) is a quasi-floor in the sense of the previous definition.

\begin{ex}Let $G$ be a group, and $H$ a subgroup of $G$. Suppose that $G$ is a hyperbolic floor over $H$ in the sense of Sela. Then $G$ is a quasi-floor over $H$ in the sense of the previous definition. In this particular case, the set $V_H^2$ is empty, $j$ is the inclusion of $H$ into $G$, and $r$ is a retraction from $G$ onto $H$, i.e.\ $r\circ j$ is the identity of $H$ (see Figure \ref{plusdidee} below).
\vspace{2mm}

\begin{figure}[h!]
\includegraphics[scale=1]{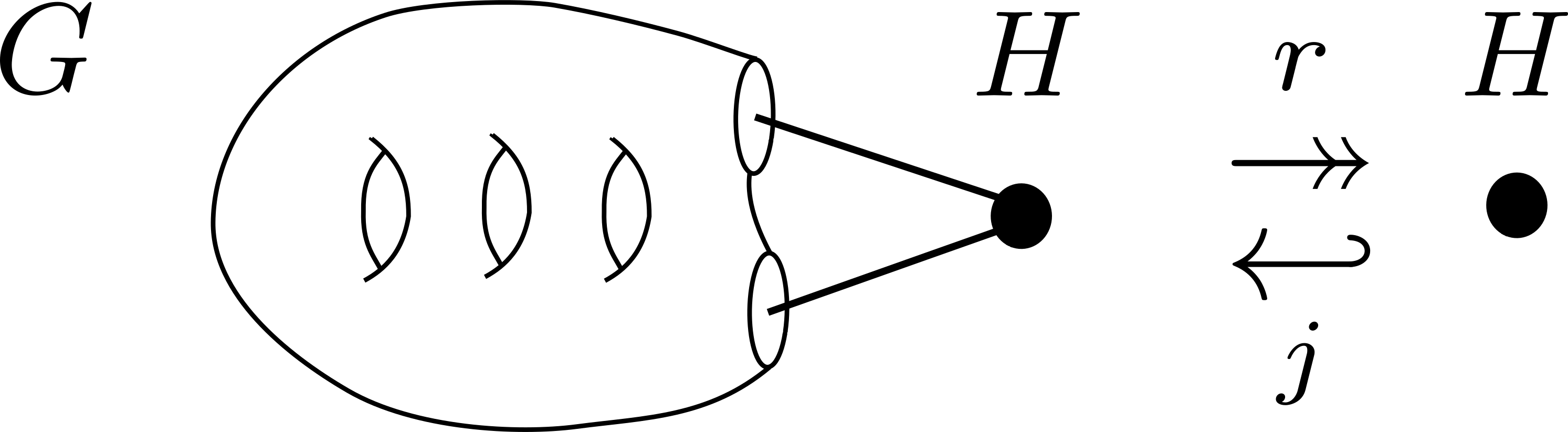}
\caption{$G$ is a hyperbolic floor over $H$ in the sense of Sela, $r\circ j=\mathrm{id}_H$.}
\label{plusdidee}
\end{figure}
\end{ex}

It should be noted that in the definition of a quasi-floor, no assumption is made about the image of the QH group $G_v$, whereas in the definition of a hyperbolic floor, $r(G_v)$ is assumed to be non-abelian. It turns out that this hypothesis is not necessary to prove that hyperbolicity is preserved under elementary equivalence.

In the definition of a quasi-floor, $r$ is not a retraction (it is not even surjective), $j$ is not injective and $H$ is not a subgroup of $G$. However, $r$ can be viewed as a "piecewise retraction" from $G$ to $H$, and $j$ can be viewed as a "piecewise inclusion" from $H$ to $G$. Indeed, as we mentionned in Remark \ref{rem}, it follows easily from Definition \ref{quasi-floor} that $r\circ j$ is inner on every vertex group $H_u$ with $u\in V_H^1$ (in particular on every one-ended subgroup of $H$), and that $j$ is injective on every vertex group $H_u$ with $u\in V_H$ (in particular on every one-ended subgroup of $H$).

It is important to emphasize that, maybe, our construction of a quasi-floor (in Section \ref{section72}) might be modified to ensure that $r$ is an epimorphism, and that $j$ is a monomorphism. However, it seemed to us that this could give rise to a number of new technical difficulties. These complications will be avoided by using Theorem \ref{sat2} stating that every hyperbolic group embeds into a torsion-saturated hyperbolic group (see Definition \ref{sat0}).

Here below is an example of a quasi-floor in the presence of torsion.

\begin{ex}\label{exempleqt}Let $A=\left(\langle x\rangle\times F(x_1,x_2)\right)\ast (\langle y\rangle\times F(x_3,x_4))$, with $x$ of order 6 and $y$ of order 10. $F(x_i,x_j)$ stands for the free group on two generators $x_i$ and $x_j$. Let $\Sigma$ be the orientable surface of genus two with two boundary components, and let $S=\pi_1(\Sigma)$. Call $\langle b_1\rangle$ and $\langle b_2\rangle$ its two boundary subgroups. Let $B=\langle z\rangle\times S$ with $z$ of order 2. Let us define a graph of groups with two vertices labelled by $A$ and $B$, and two edges linking these vertices, identifying the extended boundary subgroup $\langle z\rangle\times \langle b_1\rangle$ with $\langle x^3,[x_1,x_2]\rangle < A$ by $z\mapsto x^3,b_1\mapsto [x_2,x_1]$, and the extended boundary subgroup $\langle z\rangle\times \langle b_2\rangle$ with $\langle y^5,[x_3,x_4]\rangle < A$ by $z\mapsto y^5,b_2\mapsto [x_4,x_3]$. Call $G$ the fundamental group of this graph of groups.

First, note that $G$ cannot be a quasi-floor over $A$. To see that, remark that each involution of $G$ commutes with an element of order 3 and with an element of order 5, whereas there are two conjugacy classes of involutions in $A$: those commuting with an element of order 3, and those commuting with an element of order 5. Hence, there cannot exist any homomorphism $G\rightarrow A$ that is injective on finite subgroups.

We shall prove that $G$ is a quasi-floor over $A\ast_{\left\langle x^3\right\rangle \simeq \left\langle y^5\right\rangle}$. Here is a presentation of $G$:
\[G=\Biggl\langle 
       \begin{array}{l|cl}
                        &tx^3t^{-1}=y^5, x^6=y^{10}=1,  \\
           x,y,x_1,x_2,x_3,x_4,s_1,s_2,s_3,s_4,t & [s_3,s_4][s_1,s_2][x_2,x_1]t^{-1}[x_4,x_3]t=1, \\
                        & [x,x_1]=[x,x_2]=[y,x_3]=[y,x_4]=1    \\                                       
        \end{array}
     \Biggr\rangle\]
The group $G$ retracts onto 

    \[H=\Biggl\langle 
       \begin{array}{l|cl}
                        &tx^3t^{-1}=y^5, \ x^6=y^{10}=1,  \\
           x,y,x_1,x_2,x_3,x_4,t & \\ 
                        & [x,x_1]=[x,x_2]=[y,x_3]=[y,x_4]=1    \\                                       
        \end{array}
     \Biggr\rangle\simeq A\ast_{\left\langle x^3\right\rangle \simeq \left\langle y^5\right\rangle} \]
via the epimorphism $r : G \rightarrow H$ defined as follows:
\begin{equation*}
r: \begin{cases}
             a\mapsto a & \text{if} \ a\in\lbrace x,y,x_1,x_2,x_3,x_4,t\rbrace \\
             s_i\mapsto x_i & \text{if} \ 1\leq i\leq 2 \\
             s_i\mapsto t^{-1}x_it^{1} & \text{if} \ 3\leq i\leq 4\\
       \end{cases}.
\end{equation*}
Hence, $G$ is a quasi-floor over $H=A\ast_{\left\langle x^3\right\rangle \simeq \left\langle y^5\right\rangle}$. See Figure \ref{plusdidee2} below.

\begin{figure}[h!]
\includegraphics[scale=1]{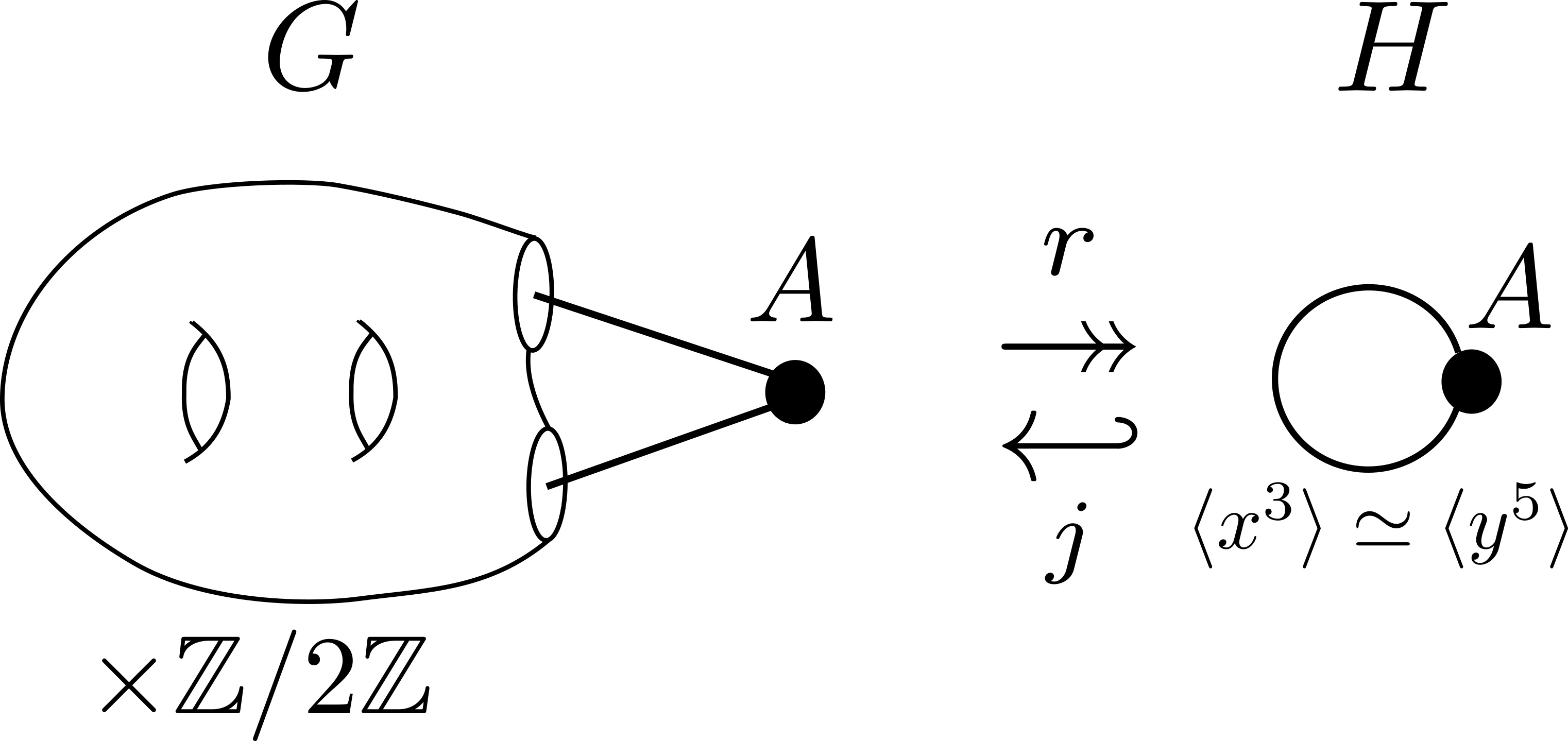}
\caption{$G$ is a quasi-floor over $H$, but not over $A$.}
\label{plusdidee2}
\end{figure}
\end{ex}

We now define quasi-towers, which are obtained by successive addition of quasi-floors.

\begin{de}[Quasi-tower]\label{quasi-tower}Let $G$ and $H$ be two groups. We say that $G$ is a quasi-tower over $H$ if there exists a finite sequence of groups $(G_m)_{0\leq m\leq n}$, with $n\geq 1$, $G_0=G$ and $G_n=H$, such that for every integer $m$ in $\llbracket 1 , n\rrbracket$, the group $G_{m-1}$ is a quasi-floor over $G_{m}$. If, moreover, every quasi-floor is strict, then the quasi-tower is said to be strict.
\end{de}

We end this subsection with two definitions that are comparable to Sela's elementary prototype (Definition 7.3 of \cite{Sel09}) and Sela's elementary core (Definition 7.5 of \cite{Sel09}).

\begin{de}[Quasi-prototype]\label{prototype}A quasi-prototype is a group $G$ that is not a strict quasi-floor over any group $H$.
\end{de}

\begin{ex}A one-ended hyperbolic group whose $\mathcal{Z}$-JSJ splitting does not contain any QH vertex is a quasi-prototype.
\end{ex}

\begin{de}[Quasi-core]\label{core}
Let $G$ be a group. If $G$ is not a quasi-prototype, a quasi-core of $G$ is a group $C$ satisfying the two following conditions:
\begin{itemize}
\item[$\bullet$] $G$ is a strict quasi-tower over $C$.
\item[$\bullet$] $C$ is a quasi-prototype.
\end{itemize}
If $G$ is a quasi-prototype, we define $G$ as the only quasi-core of $G$.
\end{de}

\begin{rque}Note that if a quasi-core exists, it is not unique \textit{a priori}.
\end{rque}

\subsection{Inheritance of hyperbolicity}

Here is an easy but essential proposition.

\begin{prop}\label{héritage}Let $G$ and $H$ be two groups. Suppose that $G$ is a quasi-tower over $H$. The following hold:
\begin{itemize}
\item[$\bullet$]$G$ is hyperbolic if and only if $H$ is hyperbolic.
\item[$\bullet$]$G$ embeds into a hyperbolic group if and only if $H$ embeds into a hyperbolic group.
\item[$\bullet$]$G$ is hyperbolic and cubulable if and only if $H$ is hyperbolic and cubulable.
\end{itemize}
\end{prop}

The proof of the third claim is postponed to Section \ref{64} (see Proposition \ref{propcubu}).

\begin{proof}
We shall prove the proposition in the case where $G$ is a quasi-floor over $H$, the general case follows immediately by induction. Let $\Delta_G$ and $\Delta_H$ be the splittings of $G$ and $H$ associated with the quasi-floor structure. By definition, $\Delta_G$ is a centered graph of groups. Let $V_G$ be its set of vertices, and let $v$ be the central vertex. Denote by $V_H^1$ the set of vertices of $\Delta_H$ whose stabilizers are infinite. By definition, there exists a bijection $s : V_G\setminus \lbrace v\rbrace \rightarrow V_H^1$ such that $G_w\simeq H_{s(w)}$ for every $w\in V_G\setminus \lbrace v\rbrace$.

We prove the first claim. Suppose that $H$ is hyperbolic. Then, by Proposition \ref{bow}, $H_u$ is hyperbolic for every vertex $u$ of $\Delta_H$. As a consequence, the vertex groups of $\Delta_G$ are hyperbolic. Since each edge group of $\Delta_G$ is virtually cyclic and almost malnormal in $G_v$, it follows from the Bestvina-Feighn combination theorem (see Proposition \ref{BF92}) that $G$ is hyperbolic. Conversely, if $G$ is hyperbolic, we prove in exactly the same way that $H$ is hyperbolic.

Now, let us prove the second claim (see Figure \ref{plonge} below). Suppose that $H$ embeds into a hyperbolic group $\Gamma$. In particular, each $H_u$ embeds into $\Gamma$. As a consequence, each $G_w$ embeds into $\Gamma$, for $w\in V_G\setminus\lbrace v\rbrace$. We construct a graph of groups $\Delta_G^{\Gamma}$ from $\Delta_G$ by replacing each vertex group $G_w$ by $\Gamma$. Recall that $\Delta_G$ is 2-acylindrical by definition of a centered graph of groups \ref{graphecentre}, so $\Delta_G^{\Gamma}$ is 2-acylindrical. Call $\Omega$ the fundamental group of $\Delta_G^{\Gamma}$. It is clear that $G$ embeds into $\Omega$. Moreover, $\Omega$ is hyperbolic by the combination theorem of Bestvina and Feighn. Conversely, if $G$ embeds into a hyperbolic group, we prove in the same way that $H$ embeds into a hyperbolic group.
\end{proof}
\vspace{1mm}

\begin{figure}[h!]
\includegraphics[scale=0.3]{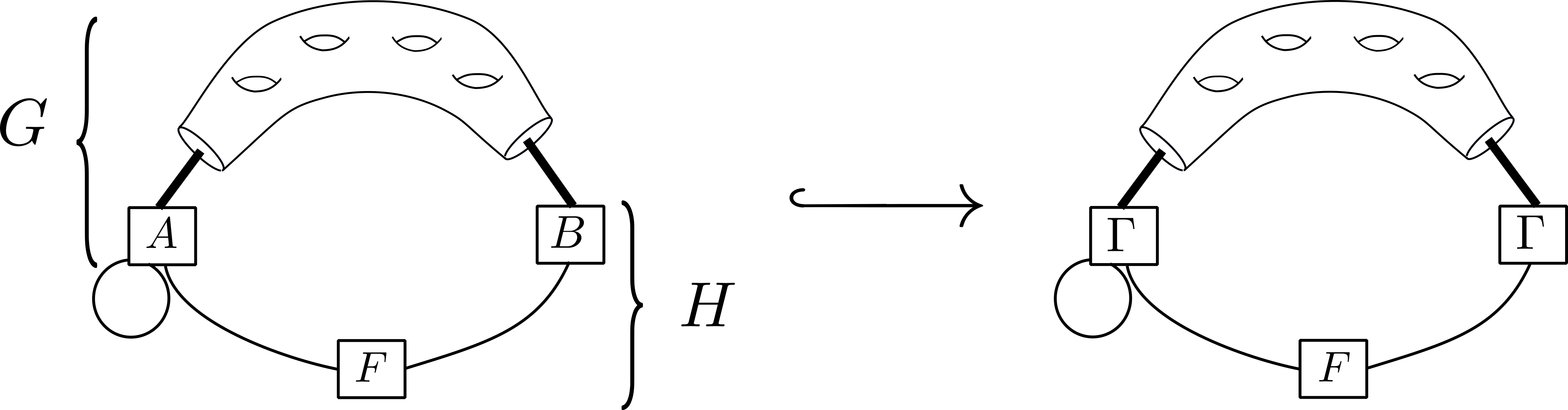}
\caption{The group $G$ is a quasi-floor over $H$. If $G$ embeds into a hyperbolic group $\Gamma$, then $H$ embeds into a hyperbolic group. Conversely, if $H$ embeds into a hyperbolic group $\Gamma$, then $G$ embeds into a hyperbolic group. Edges with infinite stabilizer are depicted in bold.}
\label{plonge}
\end{figure}

\subsection{Every $\Gamma$-limit group has a quasi-core}\label{section}

We shall prove the following result:

\begin{prop}Let $\Gamma$ be a hyperbolic group, and $G$ a $\Gamma$-limit group. Then $G$ has a quasi-core $C$. Moreover, $G$ is hyperbolic if and only if $C$ is hyperbolic.
\end{prop}

The second part of the previous proposition is an immediate consequence of Proposition \ref{héritage} above. It remains to prove the first part, that is:

\begin{prop}\label{truc2}If $\Gamma$ is a hyperbolic group, then every $\Gamma$-limit group has a quasi-core.
\end{prop}

In other words, the previous proposition claims that a $\Gamma$-limit group is either a quasi-prototype, either a strict quasi-tower over a quasi-prototype. The proposition is an easy consequence of the following lemma.

\begin{lemme}\label{truc2bis}There does not exist any infinite sequence $(G_n)_{n\geq 0}$ of finitely generated groups such that $G_0$ is a $\Gamma$-limit group and, for every integer $n$, $G_n$ is a strict quasi-floor over $G_{n+1}$.  
\end{lemme}

First, note that if $\Gamma$ is a torsion-free hyperbolic group, it follows from the descending chain condition \ref{chaine} that there does not exist any infinite sequence $(G_n)_{n\geq 0}$ such that $G_0$ is a $\Gamma$-limit group and $G_n$ is a hyperbolic floor over $G_{n+1}$ (in the sense of Sela). 

In the presence of torsion, however, Definition \ref{quasi-floor} has two drawbacks that seem to be obstacles to the use of the descending chain condition: if $G$ is a $\Gamma$-limit group, and if $G$ is a quasi-floor over $H$, then in general $H$ is neither a $\Gamma$-limit group, nor a quotient of $G$.

We remedy the first issue using the fact that every hyperbolic group embeds into a torsion-saturated hyperbolic group (see Theorem \ref{sat2}).

\begin{prop}\label{chapeau0}Let $\Gamma$ be a hyperbolic group, and $G$ a $\Gamma$-limit group. Let $\overline{\Gamma}$ be a torsion-saturated hyperbolic group containing $\Gamma$. Suppose that $G$ is a quasi-floor over a finitely generated group $H$. Then $H$ is a $\overline{\Gamma}$-limit group.
\end{prop}

\begin{proof}
Let $\Delta_H$ be the splitting of $H$ over finite groups associated to the structure of a quasi-floor. It follows from the definition of a quasi-floor that every vertex group of $\Delta_H$ embeds into $G$ (see Remark \ref{rem}), so is a $\Gamma$-limit group. By Theorem \ref{sat}, the class of $\overline{\Gamma}$-limit groups is closed under HNN extensions and amalgamated free products over finite groups, so $H$ is a $\overline{\Gamma}$-limit group.
\end{proof}

The following corollary is immediate.

\begin{co}\label{chapeau1}Let $\Gamma$ be a hyperbolic group. Let $\overline{\Gamma}$ be a torsion-saturated hyperbolic group containing $\Gamma$. Let $(G_n)_{n\geq 0}$ be a sequence of finitely generated groups such that $G_0$ is a $\Gamma$-limit group and $G_n$ is a quasi-floor over $G_{n+1}$ for every $n$. Then every $G_n$ is a $\overline{\Gamma}$-limit group.
\end{co}

Then, the following proposition remedies the lack of surjectivity in the definition of a quasi-floor, under the assumption that each quasi-floor is a $\Gamma$-limit group.

\begin{prop}\label{truc}Let $\Gamma$ be a hyperbolic group. There does not exist any infinite sequence $(G_n)_{n\geq 0}$ of groups such that, for every integer $n$, $G_n$ is a strict quasi-floor over $G_{n+1}$ and $G_n$ is a $\Gamma$-limit group.
\end{prop}

\begin{proof}
Let $(G_n)_{n\geq 0}$ be a sequence of groups such that, for every $n$, $G_n$ is a strict quasi-floor over $G_{n+1}$ and $G_n$ is a $\Gamma$-limit group. Denote by $r_n : G_n \rightarrow G_{n+1}$ and $j_n : G_{n+1} \rightarrow G_{n}$ the associated homomorphisms, for every $n$. Let $G'_n=r_n\circ\cdots \circ r_0(G_0)$. We shall apply the descending chain condition \ref{chaine} to the sequence $(G'_n)_{n\geq 0}$. It follows from the definition of a strict quasi-floor that there exists a one-ended subgroup $A_n$ of $G_n$ such that $A_n\cap\ker(r_{n+1})\neq \lbrace 1\rbrace$. In addition, the restriction of $r_n\circ\cdots \circ r_0\circ j_0\circ \cdots \circ j_n$ to $A_n$ is a conjugation. Consequently, $r_n\circ\cdots \circ r_0( j_0\circ \cdots \circ j_n(A_n))\cap\ker(r_{n+1})\neq \lbrace 1\rbrace$. So the restriction of $r_{n+1}$ to $G'_n=r_n\circ\cdots \circ r_0(G_0)$ is not injective. Hence, the descending chain condition \ref{chaine} implies that the sequence $(G'_n)$ is finite, so the sequence $(G_n)$ is finite as well.
\end{proof}

We can now prove Lemma \ref{truc2bis}, which implies Proposition \ref{truc2}.

\vspace{1mm}

\begin{proof8}Let $(G_n)_{n\geq 0}$ be a sequence of finitely generated groups such that $G_0$ is a $\Gamma$-limit group and, for every integer $n$, $G_n$ is a strict quasi-floor over $G_{n+1}$. We aim to prove that this sequence is finite. By Corollary \ref{chapeau1}, every $G_n$ is a $\overline{\Gamma}$-limit group, where $\overline{\Gamma}$ stands for a torsion-saturated hyperbolic group containing $\Gamma$. Then it follows from Proposition \ref{truc} that the sequence $(G_n)_{n\geq 0}$ is finite.
\end{proof8}

\subsection{Quasi-floor and relatedness}Here is a result that will be useful in the next section.

\begin{prop}\label{prop reliés}Let $G$ and $H$ be two finitely generated groups. Suppose that $G$ is a quasi-floor over $H$. By definition, $G$ splits as a centered splitting $\Delta_G$. Let $r : G \rightarrow H$ and $j : H\rightarrow G$ be the homomorphisms associated with the quasi-floor structure. Let $K$ be a one-ended group having a $\mathcal{Z}$-JSJ splitting, denoted by $\Delta_K$. Then,
\begin{enumerate}
\item for every monomorphism $f : K\hookrightarrow G$, $j\circ r\circ f$ and $f$ are $\Delta_K$-related;
\item for every monomorphism $f : K\hookrightarrow H$, $r\circ j\circ f$ and $f$ are $\Delta_K$-related.
\end{enumerate}
\end{prop}

\begin{proof}We shall prove the first assertion. We have to prove that $f$ and $j\circ r\circ f$ coincide, up to conjugacy by an element of $G$, on every non-QH vertex group of $\Delta_K$, as well as on every finite subgroup of $K$. The condition concerning finite subgroups is obvious, since $j\circ r$ is inner on each finite subgroup of $G$. Now, let $R$ be a non-QH vertex group of $\Delta_K$. We shall prove that $f(R)$ belongs to a non-QH vertex of $\Delta_G$. By definition, $R$ is elliptic in every tree on which $K$ acts with edge stabilizers in $\mathcal{Z}$ or finite. Therefore, $f(R)$ is elliptic in the Bass-Serre tree $T$ of $\Delta_G$ (indeed, the preimage under $f$ of every $\mathcal{Z}$-subgroup (resp. finite subgroup) of $G$ is a $\mathcal{Z}$-subgroup (resp. finite subgroup) of $K$, because $f$ is injective). We make the following observation: let $S$ be a finite-by-orbifold group, and let $B$ be a subgroup of $S$. If $B$ is elliptic in every splitting of $S$ over $\mathcal{Z}$, then $B$ is finite or lies in an extended boundary subgroup (see \cite{GL16} Corollary 5.24 (5)). As a consequence, if $R$ is rigid, $f(R)$ lies in a conjugate of a non-QH vertex group of $\Delta_G$. In the case where $R$ is virtually cyclic, then $f(R)$ may possibly lie in a boundary subgroup of the central QH vertex group of $\Delta_G$. In any case, $f(R)$ lies in a conjugate of a non-QH vertex group of $\Delta_G$. Hence, since $j\circ r$ is $\Delta_G$-related to the identity of $G$, there exists an element $g$ in $G$ such that \[{j\circ r \circ f}_{\vert R}={\iota_g\circ f}_{\vert R}.\]

This finishes the proof of the first assertion.

We now prove the second assertion. By definition, there exists a splitting $\Delta_H$ of $H$ over finite groups such that $r\circ j$ is inner on every one-ended vertex group of this splitting (see Remark \ref{rem}). Since $f(K)\simeq K$ is one-ended, it is contained in a one-ended vertex of $\Delta_H$. Thus, $r\circ j$ is inner on $f(K)$.
\end{proof}

\section{Proofs of the theorems}\label{section6}

In this section, we shall prove our main results by admitting Proposition \ref{étage}, whose proof is postponed to Section \ref{section72} for sake of clarity.

\addtocontents{toc}{\setcounter{tocdepth}{2}}

\subsection{How to build a quasi-floor using first-order logic}

We shall prove the following proposition by means of first-order logic.

\begin{prop}\label{disjonctionbis}Let $\Gamma$ be a group that embeds into a hyperbolic group, and let $G$ be a finitely generated group such that $\mathrm{Th}_{\forall\exists}(\Gamma)\subset\mathrm{Th}_{\forall\exists}(G)$. Suppose that $\Gamma$ is a quasi-tower over a group $\Gamma'$, and that $G$ is a quasi-tower over a group $G'$. Let $H$ be a one-ended subgroup of $G'$. Then one of the following holds:
\begin{itemize}
\item[$\bullet$]$H$ embeds into $\Gamma'$,
\item[$\bullet$]or there exists a non-injective preretraction $H \rightarrow G'$, that is a homomorphism related to the inclusion of $H$ into $G'$ (see Definition \ref{pre}).
\end{itemize}
\end{prop}

But before proving the proposition above, let us explain how it will be used in the sequel. In Section \ref{section72}, we shall prove the following proposition, whose proof is quite technical.

\begin{prop}\label{étage}Let $G$ be a finitely generated group possessing a $\mathcal{Z}$-JSJ splitting. Let $H$ be a one-ended factor of $G$ (defined in Section \ref{SD}) that is not finite-by-orbifold. If there exists a non-injective preretraction $H\rightarrow G$, then $G$ is a strict quasi-floor.
\end{prop}

Propositions \ref{disjonctionbis} and \ref{étage} above are complementary and explain how to build quasi-floors. The first one states the existence of a non-injective preretraction under certain conditions, while the second one asserts that a strict quasi-floor can be built using a non-injective preretraction. Combining these two propositions, we get:

\begin{prop}\label{disjonctionbis2}Let $\Gamma$ be a group that embeds into a hyperbolic group, and let $G$ be a finitely generated group such that $\mathrm{Th}_{\forall\exists}(\Gamma)\subset\mathrm{Th}_{\forall\exists}(G)$. Let $G'$ be a quasi-core of $G$. Suppose that $\Gamma$ is a quasi-tower over a group $\Gamma'$. Let $H$ be a one-ended factor of $G'$ that is not finite-by-orbifold. Then $H$ embeds into $\Gamma'$,
\end{prop}

\begin{proof15}By hypothesis, $G$ is a quasi-tower over a group $G'$. Denote by $j:G'\rightarrow G$ the associated homomorphism. Let $H$ be a one-ended subgroup of $G'$, and let $\Delta$ be its canonical $\mathcal{Z}$-JSJ splitting. Note that $j_{\vert H}$ is injective (see Remark \ref{rem}). We suppose that $H$ does not embed into $\Gamma'$. We will prove the existence of a non-injective homomorphism $H \rightarrow G'$ related to the inclusion of $H$ into $G'$.

\vspace{1mm}

\textbf{Step 1.} There exist non-trivial elements $h_1,\ldots ,h_k\in H$ such that every homomorphism $f:H\rightarrow\Gamma$ is $\Delta$-related to a homomorphism $f':H\rightarrow\Gamma$ that kills some $h_{\ell}$.

\vspace{1mm}

\textit{Proof of Step 1.} Since $\Gamma$ is a quasi-tower over $\Gamma'$, there exists by definition a finite sequence of groups $(\Gamma_m)_{0\leq m \leq n}$ such that $\Gamma_0=\Gamma$ and $\Gamma_n=\Gamma'$, and for every $1\leq m\leq n$ there exist two homomorphisms $r_m : \Gamma_{m-1} \rightarrow \Gamma_{m}$ and $j_m :\Gamma_{m}\rightarrow \Gamma_{m-1}$ such that $j_m\circ r_m$ is related to the identity of $G_{m-1}$. Let $r_0=j_0=\mathrm{id}_{\Gamma}$. 

Let $f$ be a homomorphism from $H$ to $\Gamma$. First, we shall prove that $f$ is related to a non-injective homomorphism $u$. Since there is no monomorphism from $H$ into $\Gamma'$, $r_{n}\circ\cdots\circ r_0\circ f:H\rightarrow\Gamma'$ is non-injective. Let $N$ be the smallest integer such that $r_{N}\circ\cdots\circ r_0\circ f$ is non-injective. If $N=0$, $f$ is non-injective, so we define $u:=f$. If $N\geq 1$, then $f$ is injective, so $j_1\circ r_1\circ f$ is related to $f$ by Lemma \ref{prop reliés}. If $N=1$, we define $u:=j_1\circ r_1\circ f$. If $N\geq 2$, then $r_1\circ f$ is injective, so $j_2\circ r_2\circ r_1\circ f$ is related to $r_1\circ f$ by Lemma \ref{prop reliés}. Thus $j_1\circ j_2\circ r_2\circ r_1\circ f$ is related to $j_1\circ r_1\circ f$, which is itself related to $f$. An easy induction proves that we can take $u:=j_{0}\circ\cdots j_{N}\circ r_{N}\circ\cdots\circ r_0\circ f$. Now, it follows from Corollary \ref{simple2} that there exists a homomorphism $f'$ that kills $h_{\ell}$ for some $\ell\in\llbracket 1,k\rrbracket$ and that is related to $u$, which is itself related to $f$. By transitivity, $f'$ is related to $f$. 

\vspace{1mm}

\textbf{Step 2.} There exists a non-injective homomorphism $p : H\rightarrow G$ that is related to $j_{\vert H}$.

\vspace{1mm}

\textit{Proof of Step 2.} The proof consists in expressing by a first-order sentence the statement of Step 1, and by interpreting this sentence in $G$. Let $H=\langle s_1,\ldots , s_n \ \vert \ w_1,w_2,\ldots\rangle$ be a (possibly infinite) presentation of $H$. Let $W_i=\lbrace w_1,\ldots ,w_i\rbrace$ for every $i\geq 0$. Denote by $H_i$ the finitely presented group $\langle s_1,\ldots ,s_n \ \vert \ W_i\rangle$. By hypothesis, $\Gamma$ embeds into a hyperbolic group $\Omega$. As a hyperbolic group, $\Omega$ is equationally noetherian (see \cite{RW14}, Corollary 6.13). Then $\Gamma$ and $G$ are equationally noetherian, as $\Omega$-limit groups (see \cite{OH07} Corollary 2.10). As a consequence, the sets $\mathrm{Hom}(H,\Gamma)$ and $\mathrm{Hom}(H,G)$ are respectively in bijection with $\mathrm{Hom}(H_i,\Gamma)$ and $\mathrm{Hom}(H_i,G)$ for $i$ large enough (see Lemma 5.2 of \cite{RW14}). Hence, there exists an integer $i$ such that $\mathrm{Hom}(H,\Gamma)$ (resp. $\mathrm{Hom}(H,G)$) is in bijection with the $n$-tuples in $\Gamma^n$ (resp. $G^n$) solutions of the following system of equations, denoted by $\Sigma_i(x_1,\ldots,x_n)$:
\begin{equation*}
 \Sigma_i(x_1,\ldots,x_n) : \left\{
      \begin{aligned}
        w_1(x_1,\ldots &,x_n)=1 \\
        \vdots & \\
        w_i(x_1,\ldots &,x_n)=1 \\
      \end{aligned} 
    \right. .
\end{equation*} 

Let $f$ and $f'$ be homomorphisms from $H$ to $\Gamma$. Recall that there exists an existential formula $\phi(x_1,\ldots , x_{2n})$ with $2n$ free variables such that \[\Gamma\models\phi\left(f(s_1),\ldots,f(s_n),f'(s_1),\ldots,f'(s_n)\right)\]if and only if $f$ and $f'$ are $\Delta$-related (see Lemma \ref{deltarelies}).

We can write a $\forall\exists$ first-order sentence $\mu$, verified by $\Gamma$, whose interpretation in $\Gamma$ is the statement of Step 1: for every homomorphism $f:H\rightarrow\Gamma$, there exists a homomorphism $f':H\rightarrow\Gamma$ that is $\Delta$-related to $f$ and that kills some $h_{\ell}$. The sentence $\mu$ is the following:

\[\mu : \forall x_1 \ldots\forall x_n \exists y_1 \ldots\exists y_n \  \left(\rule{0cm}{1cm}\Sigma_i(x_1,\ldots,x_n)\Rightarrow
 \left(\rule{0cm}{1cm}
       \begin{array}{lc}
                        & \Sigma_i(y_1,\ldots,y_n)  \\
        & \wedge \ \phi(x_1,\ldots,x_n,y_1,\ldots,y_n) \\
                        &   \wedge\underset{i\in\llbracket 1,k\rrbracket}{\bigvee} h_i(y_1,\ldots,y_n)=1 \\                                       
        \end{array}
     \right)\right).\]

Since $\mu\in\mathrm{Th}_{\forall\exists}(\Gamma)\subset\mathrm{Th}_{\forall\exists}(G)$, the sentence $\mu$ is true in $G$ as well. The interpretation of $\mu$ in $G$ is the following: for every homomorphism $f:H\rightarrow G$, there exists a homomorphism $f':H\rightarrow G$ that is $\Delta$-related to $f$ and that kills some $h_{\ell}$.

By taking $f=j_{\vert H}$, we get a non-injective homomorphism $p : H\rightarrow G$ that is related to $j_{\vert H}$.

\vspace{1mm}

\textbf{Step 3.} We have proved the existence of a non-injective homomorphism $p : H\rightarrow G$ that is related to $j_{\vert H}$. Let $r : G\rightarrow G'$ be the homomorphism associated with the structure of a quasi-tower. $r\circ p: H\rightarrow G'$ is related to $(r\circ j)_{\vert H}$, which is related to the inclusion of $H$ into $G'$ by the second assertion of Lemma \ref{prop reliés} (note that this assertion is stated for a quasi-floor, but it extends obviously to the case of a quasi-tower). This concludes the proof.
\end{proof15}

\subsection{Being a subgroup of a hyperbolic group is a first-order invariant}

We shall prove that the property of being a subgroup of a hyperbolic group is preserved under elementary equivalence, among finitely generated groups. More precisely, we will prove the following theorem.

\begin{te}\label{sousgroupe}Let $\Gamma$ be a group that embeds into a hyperbolic group $\Omega$, and let $G$ be a finitely generated group. If $\mathrm{Th}_{\forall\exists}(\Gamma)\subset\mathrm{Th}_{\forall\exists}(G)$, then $G$ embeds into a hyperbolic group.
\end{te}

\begin{proof}
Since $G$ is a $\Omega$-limit group, it possesses a quasi-core $G'$ (according to Proposition \ref{truc2}). Notice that every one-ended factor of $G'$ that is finite-by-orbifold is hyperbolic, as a consequence of Corollary \ref{cyclique2}. If each one-ended factor of $G'$ is finite-by-orbifold, then $G'$ is hyperbolic, so $G$ embeds into a hyperbolic group by Proposition \ref{héritage}. Otherwise, let $H_1,\ldots, H_p$ be the one-ended factors of $G'$ that are not finite-by-orbifold. It follows from Proposition \ref{disjonctionbis2} that each $H_k$ embeds into $\Gamma$, so into $\Omega$. Let $\Omega'$ be the group obtained by replacing by $\Omega$ each $H_k$ in a Stallings-Dunwoody splitting of $G'$. It is clear that $G'$ embeds into $\Omega'$. In addition, $\Omega'$ is hyperbolic. So by Proposition \ref{héritage}, $G$ embeds into a hyperbolic group $\Omega''$.
\end{proof}

Let us observe that, if $\Omega$ is locally hyperbolic, then $\Omega''$ is locally hyperbolic as well. As a consequence, the following theorem holds.

\begin{te}Let $\Gamma$ be a locally hyperbolic group, and let $G$ be a finitely generated group. If $\mathrm{Th}_{\forall\exists}(\Gamma)\subset\mathrm{Th}_{\forall\exists}(G)$, then $G$ is a locally hyperbolic group.
\end{te}

\begin{rque}\label{remarque}We stress that the theorem above can be derived more directly from Sela's shortening argument, without using quasi-towers. In \cite{Sel01}, Sela proved (Corollary 4.4) that a limit group is hyperbolic if and only if it does not contain $\mathbb{Z}^2$. Therefore, a finitely generated group $G$ satisfying $\mathrm{Th}_{\forall\exists}(F_2)\subset\mathrm{Th}_{\forall\exists}(G)$ is hyperbolic (thanks to Corollary \ref{cyclique2}). In fact, Sela's proof shows that $G$ is locally hyperbolic, and it turns out that this proof remains valid if we replace $F_2$ by a given locally hyperbolic group. 
\end{rque}

\subsection{Being hyperbolic is a first-order invariant}\label{section62}Let $G$ be a finitely generated group with the same first-order theory as a hyperbolic group $\Gamma$. In this section, we shall prove that $G$ is hyperbolic. Let $G'$ be a quasi-core of $G$, and $\Gamma'$ a quasi-core of $\Gamma$ (these groups exist thanks to Proposition \ref{truc2}). According to Proposition \ref{héritage}, it is enough to prove that $G'$ is hyperbolic. Denote by $G'_1,\ldots ,G'_n$ the one-ended vertex groups of a Stallings-Dunwoody decomposition $\Delta_{G'}$ of $G'$ that are not finite-by-orbifold. It is enough to prove that $G'_1,\ldots ,G'_n$ are hyperbolic, because the finite-by-orbifold vertex groups of $\Delta_{G'}$ are hyperbolic. Indeed, by Theorem \ref{sousgroupe}, $G$ embeds into a hyperbolic group, so $G'$ embeds into a hyperbolic group as well (thanks to Proposition \ref{héritage}). 

Denote by $\Gamma'_1,\ldots ,\Gamma'_m$ the one-ended vertex groups of a Stallings-Dunwoody splitting of $\Gamma'$ that are not finite-by-orbifold. We aim to prove that each $G'_k$ is isomorphic to some $\Gamma'_{\ell}$. Since $G$ embeds into a hyperbolic group, it follows from Proposition \ref{disjonctionbis2} that there exist two applications $\tau : \llbracket 1,n\rrbracket\rightarrow\llbracket 1,m\rrbracket$ and $\sigma : \llbracket 1,m\rrbracket\rightarrow\llbracket 1,n\rrbracket$ such that $G'_k$ embeds into $\Gamma'_{\tau(k)}$ for every $k\in\llbracket 1,n\rrbracket$, and $\Gamma'_{\ell}$ embeds into $G'_{\sigma(\ell)}$ for every $\ell\in\llbracket 1,m\rrbracket$. 

We will prove that we can choose $\tau$ and $\sigma$ so that $\sigma\circ\tau$ is a permutation of $\llbracket 1,n\rrbracket$. It will be enough to conclude (see Proposition \ref{isom}). For example, suppose that $n=m=1$. Then $G'_1$ embeds into $\Gamma'_1$ and $\Gamma'_1$ embeds into $G'_1$. Since $\Gamma$ is assumed to be hyperbolic, $\Gamma'$ is hyperbolic as well, by Proposition \ref{héritage}. As a one-ended hyperbolic group, $\Gamma'_1$ is co-Hopfian, so $G'_1$ is isomorphic to $\Gamma'_1$. Recall that the co-Hopf property for one-ended hyperbolic groups has been proved by Sela in the torsion-free case (see \cite{Sel97}) and by Moioli in the general case, in his PhD thesis (see \cite{Moi13}). 

In order to prove the existence of $\tau$ and $\sigma$ such that $\sigma\circ\tau$ is a permutation of $\llbracket 1,n\rrbracket$, we need to strenghten Proposition \ref{disjonctionbis2}.

\newpage

\begin{prop}\label{disjonctionbis2bisbis}Let $\Gamma$ be a group, and let $G$ be a finitely generated group. Suppose that 
\begin{itemize}
\item[$\bullet$]$\Gamma$ embeds into a hyperbolic group $\Omega$;
\item[$\bullet$]$\mathrm{Th}_{\forall\exists}(\Gamma)\subset\mathrm{Th}_{\forall\exists}(G)$;
\item[$\bullet$]$\Gamma$ is a quasi-tower over a group $\Gamma'$. 
\end{itemize}
As a $\Omega$-limit group, $G$ has a quasi-core $G'$. Denote by $G'_1,\ldots ,G'_n$ the one-ended vertex groups of a Stallings-Dunwoody splitting of $G'$ that are not finite-by-orbifold, and denote by $\Gamma'_1,\ldots ,\Gamma'_m$ the one-ended vertex groups of a Stallings-Dunwoody splitting of $\Gamma'$ that are not finite-by-orbifold. 
\begin{itemize}
\item[$\bullet$]There exists a homomorphism $f : G' \rightarrow \Gamma'$ and an application $\tau : \llbracket 1,n\rrbracket\rightarrow\llbracket 1,m\rrbracket$ such that, for every $k\in\llbracket 1,n\rrbracket$, the restriction of $f$ to $G'_k$ is injective and $f(G_k)$ is contained in ${\Gamma'}_{\tau(k)}^{\gamma_k}$ for some $\gamma_k\in\Gamma'$. 
\item[$\bullet$]Denote by $I\subset\llbracket 1,n\rrbracket$ the set of indices $i\in\llbracket 1,n\rrbracket$ such that $f(G'_i)={\Gamma'}_{\tau(i)}^{\gamma_{i}}$ and $\Gamma'_{\tau(i)}$ is hyperbolic. If $i$ lies in $I$, then for every $j\neq i$, and for every $\gamma\in \Gamma'$, $f(G'_{j})\not\subset f(G'_{i})^{\gamma}$. As a consequence, for every $i\in I$ and $j\in\llbracket 1,n\rrbracket$, $\tau(i)=\tau(j)\Leftrightarrow i=j$.
\end{itemize}
\end{prop}

\begin{proof}
Suppose for the sake of contradiction that the proposition is false. Then, for every homomorphism $f : G' \rightarrow \Gamma'$, one of the following holds:
\begin{itemize}
\item[$\bullet$] there exists $k\in\llbracket 1,n\rrbracket$ such that $f_{\vert G'_k}$ is non-injective,
\item[$\bullet$] or $I\neq \varnothing$, the restriction of $f$ to each $G'_{k}$ is injective, and there exist $i\in I$,$j\in\llbracket 1,n\rrbracket$ with $i\neq j$, and $\gamma,\gamma'\in\Gamma'$ such that $f(G'_j)\subset f(G'_i)^{\gamma}={\Gamma'}_{\tau(i)}^{\gamma'}$.
\end{itemize}
Let's find a contradiction.

\vspace{1mm}

\textbf{Step 1.} We will prove that there exist some finite subsets $X_1\subset G'_1,\ldots ,X_n\subset G'_n$ containing only elements of infinite order, and a finite set $F\subset G'\setminus\lbrace 1 \rbrace$ such that, for every homomorphism $f:G'\rightarrow\Gamma$, one of the following claims is true:  
\begin{itemize}[label=$\bullet$]
\item there exists $k\in\llbracket 1,n\rrbracket$ such that $f_{\vert G'_k}$ is $\Delta_k$-related to a homomorphism $f':G'_k\rightarrow\Gamma$ that kills an element of $F$ (where $\Delta_k$ stands for the $\mathcal{Z}$-JSJ splitting of $G'_k$),
\item or there exist an element $\gamma\in\Gamma$ and two elements $x_k\in X_k$ and $x_{\ell}\in X_{\ell}$ (with $k\neq \ell$) such that $f(x_k)=\gamma f(x_{\ell}) \gamma^{-1}$.
\end{itemize} 

\vspace{1mm}

\textit{Proof of Step 1.} In a first time, we will define the set $F$ and the sets $X_1,\ldots,X_n$. 

\vspace{1mm}

\textbf{Definition of $F$.} Since $\Gamma$ embeds into a hyperbolic group, for each $k\in\llbracket 1,n\rrbracket$, Corollary \ref{simple} to Sela's shortening argument \ref{sela2} provides us with a finite set $F_k\subset G'_k$ such that every non-injective homomorphism from $G'_k$ to $\Gamma$ kills an element of $F_k$, up to precomposition by a modular automorphism of $G'_k$. We let $F:=F_1\cup\cdots\cup F_n$.

\vspace{1mm}

\textbf{Definition of $X_1,\ldots,X_n$.} In the case where $I$ is empty, let $X_1=\cdots = X_n=\varnothing$. Now, assume that $I$ is non-empty. For each $k\in\llbracket 1,n\rrbracket$, since $G'_{k}$ is not finite-by-orbifold, there exists at least one non-QH vertex group $A_{k}$ in the $\mathcal{Z}$-JSJ splitting $\Delta_{k}$ of $G'_{k}$. Fix an element of infinite order $x_{k}\in A_{k}$. For each $\ell\in\llbracket 1,n\rrbracket\setminus\lbrace k\rbrace$, if $G'_{\ell}$ is not hyperbolic, let $X_{k,\ell}:=\varnothing$; if $G'_{\ell}$ is hyperbolic, let $Y_{k,\ell}:=\lbrace f(x_k) \ \vert \ f\in \mathrm{Mono}(G'_k,G'_{\ell})\rbrace$ and let $X_{k,\ell}$ be a set of representatives for the orbits of $Y_{k,\ell}$ under the action of $G'_{\ell}$ by conjugation. Note that $X_{k,\ell}$ is finite thanks to Sela's shortening argument \ref{sela2bis}. Last, we let \[X_{k}:=\lbrace x_{k}\rbrace\cup\bigcup_{1\leq \ell\neq k\leq n}X_{k,\ell}.\]

Now, we will prove that these sets have the expected property. The group $\Gamma$ being a quasi-tower over $\Gamma'$, there exists a finite sequence of groups $(\Gamma_p)_{0\leq p \leq N}$ such that ${\Gamma}_0=\Gamma$ and ${\Gamma}_N=\Gamma'$, and such that $\Gamma_{p-1}$ is a quasi-floor over $\Gamma_{p}$, for each $p\in\llbracket 1,N\rrbracket$. Let $r_p : \Gamma_{p-1} \rightarrow \Gamma_{p}$ and $j_p :\Gamma_{p}\rightarrow {\Gamma}_{p-1}$ be the homomorphisms associated with the quasi-floor structure, for each $p\in\llbracket 1,N\rrbracket$.

Let $f\in{\mathrm{Hom}(G',\Gamma)}$. If there exists a vertex group $G'_k$ such that ${f}_{\vert G'_k}$ is non-injective, then it follows from Corollary \ref{simple} that there exists $f':G'_k\rightarrow \Gamma$ that is $\Delta_k$-related to ${f}_{\vert G'_k}$ and kills an element of $F$. Otherwise, if there exists a vertex group $G'_k$ such that $(r_1\circ f)_{\vert G'_k}$ is non-injective, then $(j_1\circ r_1\circ f)_{\vert G'_k}$ is non-injective as well, so it follows from Corollary \ref{simple} that there exists $f':G'_k\rightarrow\Gamma$ that is $\Delta_k$-related to $(j_1\circ r_1\circ f)_{\vert G'_k}$ and kills an element of $F$. But $(j_1\circ r_1\circ f)_{\vert G'_k}$ is $\Delta_k$-related to ${f}_{\vert G'_k}$ according to Proposition \ref{prop reliés}, so $f'$ is $\Delta_k$-related to ${f}_{\vert G'_k}$ and satisfies the first claim. Otherwise, we look at $r_2\circ r_1\circ f$, etc. 

If, after $N$ steps, we have not found any $f'$ satisfying the first claim, then the morphism $f'=r_N\circ\cdots\circ r_1\circ f\in\mathrm{Hom}(G',\Gamma')$ is injective on each $G'_j$. So, by hypothesis, the set $I$ is non-empty and there exist $k, \ell\in\llbracket 1,n\rrbracket$ with $k\neq \ell$, $i\in I$ and $\gamma,\gamma'\in\Gamma'$ such that $f'(G'_k)$ is contained in $\iota_{\gamma}(f'(G'_{\ell}))=\iota_{\gamma'}(\Gamma'_{i})$. As a consequence, $G'_{\ell}$ is hyperbolic and \[\left({f'}_{\vert G'_{\ell}}\right)^{-1}\circ\iota_{\gamma^{-1}}\circ \left(f_{\vert G'_k}\right)\in\mathrm{Mono}(G'_k,G'_{\ell}).\] Hence, by definition of $X_k$ and $X_{\ell}$, there exist $x_k\in X_k$ and $x_{\ell}\in X_{\ell}$ such that $f'(x_k)=\iota_{\gamma}(f'(x_{\ell}))$. So $j_1\circ\cdots\circ j_N\circ f'(x_k)=\iota_{\gamma'}\circ j_1\circ\cdots\circ j_N\circ f'(x_{\ell})$ for some $\gamma'\in \Gamma$. But $j_1\circ\cdots\circ j_N\circ f'=j_1\circ\cdots\circ j_N\circ r_N\circ\cdots\circ r_1\circ f$ is related to $f$ thanks to Proposition \ref{prop reliés}. Therefore, $f(x_k)=\iota_{\gamma''}\circ f(x_{\ell})$ for some $\gamma''\in \Gamma$. This concludes the proof of the first step.

\vspace{1mm}

\textbf{Step 2.} The statement of Step 1 is expressible by a $\forall\exists$-sentence, denoted by $\phi$, which is true in $\Gamma$ (as in the proof of Proposition \ref{disjonctionbis}, Step 2). Since $\mathrm{Th}_{\forall\exists}(\Gamma)\subset\mathrm{Th}_{\forall\exists}(G)$, $\phi$ is true in $G$ as well. Thus, for every $f\in{\mathrm{Hom}(G',G)}$, one of the following claims is true:
\begin{itemize}[label=$\bullet$]
\item there exist a vertex group $G'_k$ together with a homomorphism $f':G'_k\rightarrow G$ which is $\Delta_k$-related to $f_{\vert G'_k}$ and kills an element of $F$,
\item or there exist an element $g\in G$ and two elements $x_k\in X_k$ and $x_{\ell}\in X_{\ell}$ (with $k\neq \ell$) such that $f(x_k)=g f(x_{\ell}) g^{-1}$, with $x_k$ of infinite order.
\end{itemize}

By definition, $G$ is a quasi-tower over $G'$. Let $r : G \rightarrow G'$ and $j: G'\rightarrow G$ be the two homomorphisms associated with this structure of a quasi-tower. 

Taking $f:=j$, the second claim above is false. Otherwise, $j(G'_k)\cap j(G'_{\ell})^g$ is infinite, so $r\circ j(G'_k)\cap r\circ j(G'_{\ell})^{r(g)}$ is infinite, since $r$ is injective in restriction to $j(G'_k)$. But $r\circ j$ is inner on $G'_k$ and on $G'_{\ell}$. Hence, there exists an element $h\in G'$ such that $G'_k\cap h{G'_{\ell}}h^{-1}$ is infinite. This is a contradiction, since $G'_k$ and ${G'_{\ell}}$ are two different vertex groups of a Stallings-Dunwoody decomposition of $G'$.

As a consequence, the first claim is necessarily true (for $f:=j$). There exist a vertex group $G'_k$ together with a homomorphism $f':G'_k\rightarrow G$ which is $\Delta_k$-related to $j_{\vert G'_k}$ and kills an element of $F$. Then $(r\circ f')_{\vert G'_k} : G'_k \rightarrow G'$ is a non-injective preretraction, by Lemma \ref{prop reliés}. It follows from Proposition \ref{étage} that $G'$ is a strict quasi-floor. This is a contradiction, since $G'$ is a quasi-core by hypothesis.
\end{proof}

We can now prove that hyperbolicity is preserved under elementary equivalence.

\begin{prop}\label{isom}Let $\Gamma$ be a hyperbolic group and $G$ a finitely generated group such that $\mathrm{Th}_{\forall\exists}(\Gamma)=\mathrm{Th}_{\forall\exists}(G)$. Let $G'$ be a quasi-core of $G$ and $\Gamma'$ a quasi-core of $\Gamma$. Then every one-ended factor of $G'$ that is not finite-by-orbifold is isomorphic to a one-ended factor of $\Gamma'$. Therefore, $G'$ is hyperbolic.
\end{prop}

Thanks to Proposition \ref{héritage} about inheritance of hyperbolicity, the corollary below is immediate.

\begin{co}Let $\Gamma$ be a hyperbolic group and $G$ a finitely generated group such that $\mathrm{Th}_{\forall\exists}(\Gamma)=\mathrm{Th}_{\forall\exists}(G)$. Then $G$ is hyperbolic.
\end{co}

\begin{proof5}
Let $G'_1,\ldots ,G'_n$ be the one-ended vertex groups of a Stallings-Dunwoody splitting of $G'$ that are not finite-by-orbifold, and let $\Gamma'_1,\ldots ,\Gamma'_m$ be the one-ended vertex groups of a Stallings-Dunwoody splitting of $\Gamma'$ that are not finite-by-orbifold. Since $\Gamma$ is hyperbolic, $\Gamma'$ is hyperbolic (by Proposition \ref{héritage}), so $\Gamma'_1,\ldots ,\Gamma'_m$ are hyperbolic.

Let $f : G'\rightarrow\Gamma'$ and $h:\Gamma'\rightarrow G'$ be the homomorphisms given by Proposition \ref{disjonctionbis2bisbis} (note that $f$ exists because $G$ embeds into a hyperbolic group, by Proposition \ref{sousgroupe}). 

Let $\tau : \llbracket 1,n\rrbracket\rightarrow\llbracket 1,m\rrbracket$ be the map such that, for every $i\in\llbracket 1,n\rrbracket$, $f(G'_i)$ is contained in a conjugate of $\Gamma'_{\tau(i)}$. As in Proposition \ref{disjonctionbis2bisbis}, denote by $I\subset\llbracket 1,n\rrbracket$ the set of indices $i\in\llbracket 1,n\rrbracket$ such that $f(G'_i)=\Gamma'_{\tau(i)}$ up to conjugacy (and $\Gamma'_{\tau(i)}$ is hyperbolic).

Let $\sigma : \llbracket 1,m\rrbracket\rightarrow\llbracket 1,n\rrbracket$ be the map such that, for every $j\in\llbracket 1,m\rrbracket$, $h(\Gamma'_j)$ is contained in a conjugate of $G'_{\sigma(j)}$. Denote by $J\subset\llbracket 1,m\rrbracket$ the set of indices $j\in\llbracket 1,m\rrbracket$ such that $h(\Gamma'_j)=G'_{\sigma(j)}$ up to conjugacy, and $G'_{\sigma(j)}$ is hyperbolic.

If $i\in\llbracket 1,n\rrbracket$ is periodic under $\sigma\circ\tau$, then $f$ induces an isomorphism between $G'_i$ and $f(G'_i)=\Gamma'_{\tau(i)}$, because one-ended hyperbolic groups are co-Hopfian. As a consequence, $i$ belongs to $I$. So it follows from Proposition \ref{disjonctionbis2bisbis} that $i$ is the unique preimage of $\tau(i)$ under the map $\tau$. Similarly, if $j\in\llbracket 1,m\rrbracket$ is periodic under $\tau\circ\sigma$, then $h$ induces an isomorphism between $\Gamma'_j$ and $h(\Gamma'_j)=G'_{\sigma(j)}$. Therefore, $G'_{\sigma(j)}$ is hyperbolic and $j$ lies in $J$. So it follows from Proposition \ref{disjonctionbis2bisbis} that $j$ is the unique preimage of $\sigma(j)$ under the map $\sigma$.

One easily checks that any such pair of maps $\sigma,\tau$ between finite sets are necessarily bijective, and that every element is periodic. Hence, for every $i\in\llbracket 1,n\rrbracket$, $f$ induces an isomorphism between $G'_i$ and $\Gamma'_{\tau(i)}$ (up to conjugacy), which is hyperbolic. This concludes the proof.
\end{proof5}

\subsection{Being hyperbolic and cubulable is a first-order invariant}\label{64}
In this section, we shall prove the following theorems.

\begin{te}\label{hypcubulation}Let $\Gamma$ be a hyperbolic group and let $G$ be a finitely generated group. Suppose that $\mathrm{Th}_{\forall\exists}(\Gamma) = \mathrm{Th}_{\forall\exists}(G)$. Then $\Gamma$ is cubulable if and only if $G$ is cubulable.
\end{te}

\begin{te}\label{hypcubulation2}Let $\Gamma$ and $G$ be two finitely generated groups. Suppose that $\mathrm{Th}_{\forall\exists}(\Gamma) = \mathrm{Th}_{\forall\exists}(G)$. Then $\Gamma$ is hyperbolic and cubulable if and only if $G$ is hyperbolic and cubulable.
\end{te}

Note that Theorem \ref{hypcubulation2} above follows immediately from Theorem \ref{hypcubulation} and from the fact that hyperbolicity is preserved under elementary equivalence, among finitely generated groups (see Theorem \ref{isom}).

We refer the reader to \cite{Sag14} for a definition of $\mathrm{CAT}(0)$ cube complexes.

\begin{de}
\normalfont 
A group $G$ is said to be cubulable if it acts properly and cocompactly on a $\mathrm{CAT}(0)$ cube complex.
\end{de}

Before proving Theorem \ref{hypcubulation}, we shall prove the following proposition.

\begin{prop}\label{propcubu}Let $G$ be a hyperbolic group. If $G$ is a quasi-tower over a group $H$, then $G$ is cubulable if and only if $H$ is cubulable.
\end{prop}

In order to prove this result, we need a combination theorem for hyperbolic and cubulable groups, analogous to that of Bestvina and Feighn for hyperbolic groups (see \ref{BF92}). 

Recall that a subgroup $H$ of a group $G$ is called almost malnormal if $H\cap gHg^{-1}$ is finite for every $g$ in $G\setminus H$. Note that edge groups in a centered graph of groups are almost malnormal in the central vertex group (see Definition \ref{graphecentre}).

In \cite{HW15}, Hsu and Wise proved the following result, as a special case of their main theorem.

\begin{prop}[Combination theorem for hyperbolic and cubulable groups]\label{combicubu}
~\

Let $G=A\ast_CB$ be an amalgamated product such that $A$ and $B$ are hyperbolic and cubulable, and $C$ is virtually cyclic and almost malnormal in $A$ or in $B$. Then $G$ is hyperbolic and cubulable. 

Let $G=A\ast_C$ be an HNN extension such that $A$ is hyperbolic and cubulable, and $C$ is virtually cyclic and almost malnormal in $A$. Then $G$ is hyperbolic and cubulable.

\end{prop}

We also need the following proposition, proved by Haglund, claiming that cubulability is inherited by quasi-convex subgroups.

\begin{prop}[\cite{Hag08}, Corollary 2.29]\label{hag}Let $G$ be a cubulable group. If $H$ is a quasi-convex subgroup of $G$, then $H$ is cubulable.
\end{prop}

Let us now prove Proposition \ref{propcubu} and Theorem \ref{hypcubulation}.

\vspace{1mm}

\begin{proof18}
We shall prove the proposition in the case where $G$ is a quasi-floor over $H$, the general case follows immediately. Since $G$ is assumed to be hyperbolic, by Proposition \ref{héritage}, the group $H$ is hyperbolic.

Let $\Delta_G$ and $\Delta_H$ be the splittings of $G$ and $H$ associated with the quasi-floor structure. By definition, $\Delta_G$ is a centered graph of groups. Let $V_G$ be its set of vertices, and let $v$ be the central vertex. Denote by $V_H^1$ the set of vertices of $\Delta_H$ whose stabilizers are infinite. By definition, there exists a bijection $s : V_G\setminus \lbrace v\rbrace \rightarrow V_H^1$ such that $G_w\simeq H_{s(w)}$ for every $w\in V_G\setminus \lbrace v\rbrace$.

Suppose that $H$ is cubulable. Then, by Proposition \ref{hag}, $H_u$ is cubulable for every vertex $u$ of $\Delta_H$. As a consequence, the vertex groups of $\Delta_G$ are cubulale. Since each edge group of $\Delta_G$ is virtually cyclic and almost malnormal in $G_v$, it follows from the combination theorem \ref{combicubu} that $G$ is cubulable. Conversely, if $G$ is cubulable, we prove in exactly the same way that $H$ is cubulable. 
\end{proof18}

\vspace{1mm}

\begin{proof17}
Let $\Gamma$ be a hyperbolic group and let $G$ be a finitely generated group such that $\mathrm{Th}_{\forall\exists}(\Gamma) = \mathrm{Th}_{\forall\exists}(G)$. Suppose that $\Gamma$ is cubulable. We will prove that $G$ is cubulable. Let us denote by $G'$ a quasi-core of $G$ and by $\Gamma'$ a quasi-core of $\Gamma$. According to Proposition \ref{propcubu}, $\Gamma'$ is cubulable, so each vertex group of a Stallings-Dunwoody splitting of $\Gamma'$ is cubulable, by Proposition \ref{hag}. Moreover, according to Proposition \ref{isom}, every vertex group of a Stallings-Dunwoody splitting of $G'$ (that is not finite-by-orbifold) is isomorphic to a vertex group of a Stallings-Dunwoody splitting of $\Gamma'$. Consequently, by Theorem \ref{combicubu} the group $G'$ is cubulable. Then, it follows from Proposition \ref{propcubu} that $G$ is cubulable. Symmetrically, if $G$ is cubulable, then so is $\Gamma$ (because $G$ is hyperbolic, by Theorem \ref{isom}).\end{proof17}

\addtocontents{toc}{\setcounter{tocdepth}{1}}

\section{From a preretraction to a quasi-floor}\label{section72}

In the previous section, we proved our main theorems by admitting that we can build a strict quasi-floor by means of a non-injective preretraction. This is a precise statement:

\begin{prop}\label{floor7}Let $G$ be a finitely generated $K$-$\mathrm{CSA}$ group that does not contain $\mathbb{Z}^2$. Suppose that $G$ has a one-ended factor $H$ that is not finite-by-orbifold. Let $\Delta$ be the $\mathcal{Z}$-JSJ splitting of $H$. Suppose that there exists a non-injective homomorphism $p:H\rightarrow G$ that is $\Delta$-related to the inclusion of $H$ into $G$. Then $G$ is a strict quasi-floor.
\end{prop}

The goal of this section is to prove the proposition above.

\subsection{A preliminary proposition}

\begin{prop}\label{lemmeperin}Let $G$ be a one-ended finitely generated $K$-$\mathrm{CSA}$ group that does not contain $\mathbb{Z}^2$, so that $G$ has a $\mathcal{Z}$-JSJ splitting denoted by $\Delta$. Let $p$ be an endomorphism of $G$. If $p$ is $\Delta$-related to the identity of $G$ and sends every QH group isomorphically to a conjugate of itself, then $p$ is injective.
\end{prop}

\begin{proof}
Let $T$ be the Bass-Serre tree of $\Delta$. We denote by $V$ the set of vertices of $T$. First of all, let us recall some properties of $\Delta$ that will be useful in the sequel (see Section \ref{25}).
\begin{enumerate}
\item The graph $\Delta$ is bipartite, with every edge joining a vertex carrying a virtually cyclic group to a vertex carrying a non-virtually-cyclic group.
\item Let $v$ be a vertex of $T$, and let $e,e'$ be two distinct edges incident to $v$. If $G_v$ is not virtually cyclic, then the group $\langle G_e,G_{e'}\rangle$ is not virtually cyclic.
\item The action of $G$ on $T$ is 2-acylindrical: if an element $g\in G$ fixes a segment of length $>2$ in $T$, then $g$ has finite order.
\end{enumerate}

If $\Delta$ is reduced to a point, then $p$ is obviously injective. From now on, we will suppose that $\Delta$ has at least two vertices.

As a first step, we build a $p$-equivariant map $f:T\rightarrow T$. Let $v_1,\ldots ,v_n$ be some representatives of the orbits of vertices. For every $1\leq k\leq n$, there exists $g_k\in G$ such that $p(G_{v_k})=g_kG_{v_k}g_k^{-1}$. We let $f(v_k)=g_k\cdot v_k$, so that $p(G_{v_k})=G_{f(v_k)}$. Then we define $f$ on each vertex of $T$ by equivariance. Next, we define $f$ on the edges of $T$ in the following way: if $e$ is an edge of $T$, with endpoints $v$ and $w$, there exists a unique path $e'$ from $f(v)$ to $f(w)$ in $T$. We let $f(e)=e'$.

Now we will prove that $f$ is injective, which allows to conclude that $p$ is injective. Indeed, if $p(g)=1$ for some $g\in G$, then for every vertex $v\in V$ one has $p(g)\cdot f(v)=f(g\cdot v)=f(v)$, thus $g\cdot v=v$ for every $v$. Since the action of $G$ on $\Delta$ is 2-acylindrical, $g$ has finite order. Moreover the restriction of $p$ to every element of finite order is injective (by definition of $\Delta$-relatedness), so $g=1$, which proves that $p$ is injective.

We now prove that $f$ is injective. The proof will proceed in two steps: first, one shows that $f$ sends adjacent vertices on adjacent vertices, then one proves that there are no foldings.

\vspace{1mm}

\textbf{$f$ sends adjacent vertices to adjacent vertices:} let's consider two adjacent vertices $v$ and $w$ of $T$. One has $d(f(v),f(w))\leq 2$, because the action of $G$ on $T$ is 2-acylindrical and $p$ is injective on edge groups, which are virtually cylic and infinite. Since the graph is bipartite, $d(f(v),f(w))$ and $d(v,w)=1$ have the same parity. Hence $d(f(v),f(w))=1$.

\vspace{1mm}

\textbf{There are no foldings:} let $v$ be a vertex of $T$, let $w$ and $w'$ be two distinct vertices adjacent to $v$. Denote by $e$ and $e'$ the edges between $v$ and $w$, and between $v$ and $w'$ respectively. Argue by contradiction and suppose that $f(w)=f(w')$, then $f(e)=f(e')$ since there are no circuits in a tree.

\begin{center}
\begin{tikzpicture}[scale=1]
\node[draw,circle, inner sep=1.7pt, fill, label=below:{$w$}] (A1) at (2,0) {};
\node[draw,circle, inner sep=1.7pt, fill, label=below:{$w'$}] (A2) at (2,2) {};
\node[draw,circle, inner sep=1.7pt, fill, label=below:{$v$}] (A3) at (0,1) {};
\node[draw=none, label=below:{$e$}] (B1) at (1,0.5) {};
\node[draw=none, label=below:{$e'$}] (B2) at (1,2.2) {};
\node[draw=none, label=below:{$f(e)=f(e')$}] (B3) at (7,2) {};
\node[draw,circle, inner sep=1.7pt, fill, label=below:{$f(w)=f(w')$}] (A4) at (8,1) {};
\node[draw,circle, inner sep=1.7pt, fill, label=below:{$f(v)$}] (A5) at (6,1) {};

\draw[-,>=latex] (A3) to (A1) ;
\draw[-,>=latex] (A3) to (A2);
\draw[-,>=latex] (A4) to (A5);
\draw[->,>=latex, dashed] (3,1) to (5,1);
\end{tikzpicture}
\end{center}

If $G_v$ is not virtually cyclic, then $\langle G_e,G_{e'}\rangle$ is not virtually cyclic (see the second property of $\Delta$ recalled above), thus $p(\langle G_e,G_{e'}\rangle)$ is not virtually cyclic since $\langle G_e,G_{e'}\rangle$ is contained in $G_v$, and $p$ is injective on $G_v$. Hence one can assume now that $G_v$ is virtually cyclic. There exists an element $g\in G$ such that $w'=g\cdot w$. Since $f$ is $p$-equivariant, $p(g)\cdot f(w)=f(w)$, i.e.\ $p(g)\in G_{f(w)}=p(G_w)$. As a consequence, there exists an element $h\in G_w$ such that $p(g)=p(h)$. Up to multiplying $g$ by the inverse of ${h}$, one can assume that $p(g)=1$. Then $g$ does not fix a point of $T$, because $p$ is injective on vertex groups and $g\neq 1$. It follows that $g$ is hyperbolic, with translation length equal to 2.

\begin{center}
\begin{tikzpicture}[scale=1]
\node[draw,circle, inner sep=1.7pt, fill, label=below:{$w$}] (A1) at (2,0) {};
\node[draw,circle, inner sep=1.7pt, fill, label=below:{$w'$}] (A2) at (2,2) {};
\node[draw,circle, inner sep=1.7pt, fill, label=below:{$v$}] (A3) at (0,1) {};
\node[draw,circle, inner sep=1.7pt, fill, label=below:{$gv$}] (A4) at (0,3) {};
\node[draw=none, label=below:{$e$}] (B1) at (1,0.5) {};
\node[draw=none, label=below:{$e'$}] (B2) at (1,2.2) {};
\node[draw=none, label=below:{$ge$}] (B2) at (1,3.4) {};
\draw[-,>=latex] (A2) to (A4) ;
\draw[-,>=latex] (A3) to (A1) ;
\draw[-,>=latex] (A3) to (A2);
\end{tikzpicture}
\end{center}

The group $\langle G_{e'},G_{ge}\rangle$ is not virtually cyclic since $G_{w'}$ is not virtually cyclic. It follows that $p(\langle G_{e'},G_{ge}\rangle)$ is not virtually cyclic (indeed, $p$ is injective on $G_{w'}$). On the other hand, $p(\langle G_{e'},G_{ge}\rangle)$ is equal to $p(\langle G_{e'},G_{e}\rangle)$ because $p(g)=1$. Thus $p(\langle G_{e'},G_{e}\rangle)$ is not virtually cyclic. It is a contradiction.
\end{proof}

\subsection{Building a quasi-floor from a non-injective preretraction}

\begin{de}[Maximal pinched set, pinched quotient, pinched decomposition]Let $G$ be a group that splits as a centered splitting $\Delta_G$, with central vertex $v$. The stabilizer $G_v$ of $v$ is a conical finite-by-orbifold group $F\hookrightarrow G_v \twoheadrightarrow\pi_1(O)$. Denote by $q$ the epimorphism from $G_v$ onto $\pi_1(O)$. Let $G'$ be a group, and let $p : G \rightarrow G'$ be a homomorphism. Let $S$ be an essential set of curves on $O$ (see Definition \ref{essential}). Suppose that each  element of $S$ is pinched by $p$ (meaning that $p\left(q^{-1}(\alpha)\right)$ is finite for every $\alpha\in S$), and that $S$ is maximal for this property. The set $S$ is called a maximal pinched set for $p$. Note that $S$ may be empty.

For every $\alpha\in S$, $q^{-1}(\alpha)$ is a virtually cyclic subgroup of $G_v$, isomorphic to $F\rtimes\mathbb{Z}$. Let $N_{\alpha}=\ker\left(p_{\vert q^{-1}(\alpha)}\right)$, and let $N$ be the subgroup of $G$ normally generated by $\lbrace N_{\alpha}\rbrace_{\alpha\in S}$. The quotient group $Q=G/N$ is called the pinched quotient of $G$ associated with $S$. Let $\pi : G\twoheadrightarrow Q$ be the quotient epimorphism. Since each $N_{\alpha}$ has finite index in $q^{-1}(\alpha)$, and since $p$ is injective on finite subgroups, killing $N$ gives rise to new conical points and new QH vertices. The group $Q$ splits naturally as a graph of groups $\Delta_Q$ obtained by replacing the vertex $v$ in $\Delta_G$ by the splitting of $\pi(G_{v})$ over finite groups obtained by killing $N$ (see Figure \ref{pinchedquotient} below). $\Delta_Q$ is called the pinched decomposition of $Q$.
\begin{figure}[!h]
\centering
\includegraphics[scale=0.8]{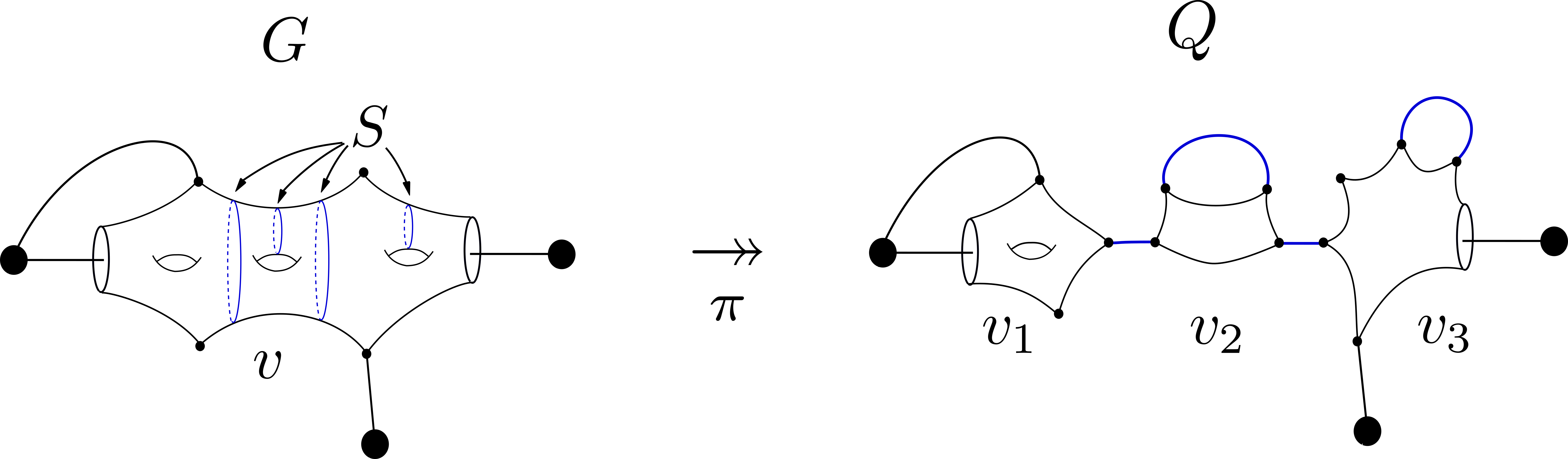}
\caption{For convenience, on this figure, $F$ is trivial. For each $\alpha\in S$, there exists a smallest integer $n\geq 1$ such that $p(\alpha^n)=1$. Killing $\langle\alpha^n\rangle$ gives rise to a conical point of order $n$. The new QH vertices coming from $v$ are denoted by $v_1,v_2,v_3$.}
\label{pinchedquotient}
\end{figure}
\end{de}

We keep the same notations. Suppose that there exists an endomorphism $p$ of $G$ that is $\Delta_G$-related to the identity of $G$. Let $S$ be a maximal pinched set for $p$, let $Q$ be the pinched quotient and let $\Delta_Q$ be its pinched decomposition. Denote by $v_1,\ldots,v_n$ the new vertices coming from $v$ (see Figure \ref{pinchedquotient} above). Let $\pi : G \twoheadrightarrow Q$ be the quotient epimorphism. There exists a unique homomorphism $\phi : Q \rightarrow G$ such that $p=\phi\circ\pi$. This homomorphism is non-pinching since $S$ is assumed to be maximal. We will need a lemma.

\begin{lemme}\label{magique}We keep the same notations. Assume that $p$ does not send $G_v$ isomorphically to a conjugate of itself, and denote by $\mathcal{F}$ the set of edges of $\Delta_Q$ with finite stabilizer. Let $Y$ be a connected component of $\Delta_Q\setminus \mathcal{F}$, and let $Q_Y$ be its stabilizer. Then $\phi(Q_Y)$ is elliptic in the Bass-Serre tree of $\Delta_G$. Moreover, $Y$ contains at most one vertex $w$ different from the vertices $v_k$ coming from the central vertex $v$. If it does, then $\phi(Q_Y)=\phi(Q_w)$.
\end{lemme}

\begin{proof}
First, we shall prove that $\phi(Q_{v_k})$ is elliptic in the Bass-Serre tree $T$ of $\Delta_G$, for every $k\in\llbracket 1,n\rrbracket$. Denote by $O_k$ the underlying orbifold of $Q_{v_k}$. We shall use Proposition \ref{cutting} to split $Q_{v_k}$ as a graph of groups all of whose vertex groups are elliptic in $T$ via $\phi$. If $C$ is an extended boundary subgroup of $Q_{v_k}$, $C$ is of the form $\pi (C')$ where $C'$ stands for an extended boundary subgroup of $G_v$, so $\phi(C)=p(C')$ is elliptic in $T$ by definition of $\Delta_G$-relatedness. Consequently, by Proposition \ref{cutting}, there exists a set $\mathcal{C}_k$ of disjoint simple loops on $O_k$ such that, if $X$ is a connected component of $O_k\setminus\mathcal{C}_k$, and if $H$ is the preimage of $\pi_1(X)$ in $Q_{v_k}$, then $\phi(H)$ is elliptic in the Bass-Serre tree $T$.

Let $\Delta_Q(\mathcal{C}_k)$ be the splitting of $Q$ obtained by replacing $v_k$ in $\Delta_Q$ by the splitting of $Q_{v_k}$ dual to $\mathcal{C}_k$. First, note that if $\mathcal{C}_k$ is empty, then $\phi(Q_{v_k})$ is obviously elliptic in $T$. Assume now that $\mathcal{C}_k$ is non-empty and denote by $v_{k,1},\ldots,v_{k,m}$ the new vertices coming from $v_k$. Let $T'$ be the Bass-Serre tree of $\Delta_Q(\mathcal{C}_k)$. Let $w_{k,i},w_{k,j}\in T'$ be two representatives of $v_{k,i},v_{k,j}\in\Delta_Q(\mathcal{C}_k)$ that are adjacent in $T'$ and linked by an edge with infinite stabilizer. By the previous paragraph, there exists a non-empty subset $I\subset T$ pointwise-fixed by $\phi\left(Q_{w_{k,i}}\right)$, and a non-empty subset $J\subset T$ pointwise-fixed by $\phi\left(Q_{w_{k,j}}\right)$. Let $x\in I$ and $y\in J$ such that $d(x,y)=d(I,J)$, where $d$ is the natural metric on $T$. By definition of a centered splitting, if an element of $G$ of infinite order fixes a segment of length $\geq 2$ in $T$, then this segment has length exactly 2 and its endpoints are translates of the central vertex $v$. Therefore, since $f$ is non-pinching on $Q_{v_k}$, $d(x,y)\in\lbrace 0,1,2\rbrace$, and $d(x,y)=2$ if and only if $x$ and $y$ are translates of $v$. We will prove that $d(x,y)=0$.

First, suppose that $d(x,y)=2$. Then we can assume without loss of generality that $x=v$, i.e. $\phi\left(Q_{w_{k,i}}\right)< G_v$. Since $f$ is non-pinching on $Q_{w_{k,i}}$, the group $\phi\left(Q_{w_{k,i}}\right)$ is infinite, so it is not contained in an extended conical subgroup of $G_v$. If $\phi\left(Q_{w_{k,i}}\right)$ is not contained in an extended boundary subgroup of $G_v$, then it follows from Proposition \ref{complexity} that $k(Q_{w_{k,i}})\geq k(G_v)$, with equality if and only if $f$ induces an isomorphism from $Q_{w_{k,i}}$ to $G_v$. This is a contradiction since the complexity decreases as soon as we cut along a loop or pinch a loop (so $k(Q_{w_{k,i}})\leq k(G_v)$), and $Q_{w_{k,i}}$ is not isomorphic to $G_v$. Therefore, $\phi\left(Q_{w_{k,i}}\right)$ is necessarily contained in an extended boundary subgroup of $G_v$. Then $\phi\left(Q_{w_{k,i}}\right)$ fixes a point $z$ in $T$ such that $d(x,z)=1$. As a consequence, $d(z,y)=1$ or $d(z,y)=3$. This last case is impossible since an element of $G$ of infinite order fixes a segment of length $\leq 2$ in $T$. So $d(z,y)=1$, and this contradicts the definition of $x$.

Now, suppose that $d(x,y)=1$. Since $\Delta_G$ is bipartite, one can assume, up to composing $\phi$ by an inner automorphism and permuting $x$ and $y$, that $x=v$. If $\phi\left(Q_{w_{k,i}}\right)$ is not contained in an extended boundary subgroup of $G_v$, we get a contradiction thanks to Proposition \ref{complexity}, as above. Thus $\phi\left(Q_{w_{k,i}}\right)$ is contained in an extended boundary subgroup of $G_v$. So $\phi\left(Q_{w_{k,i}}\right)$ has a fixed point $z$ in $T$ such that $d(x,z)=1$, so $d(z,y)=0$ or $d(z,y)=2$. This last case is impossible since $z$ and $y$ are not translates of $v$. As a consequence, $d(z,y)=0$, and this contradicts the definition of $x$. 

Hence, we have proved that $d(x,y)=0$. As a conclusion, $\phi(Q_{v_k})$ is elliptic in the Bass-Serre tree $T$ of $\Delta_G$, for every $k\in\llbracket 1,n\rrbracket$.

Now, let $w$ be a vertex of $\Delta_Q$, different from the vertices $v_i$, such that $w$ and $v_k$ are linked by an edge with infinite stabilizer. Let $T''$ be the Bass-Serre tree of $\Delta_Q$. For convenience, we still denote by $w$ and $v_k$ two adjacent representatives of $w$ and $v_k$ in $T''$ linked by an edge with infinite stabilizer. We shall prove that $\phi(Q_{v_k})$ is contained in $\phi(Q_w)$. We have proved the existence of a subset $I\subset T$ pointwise-fixed by $\phi(Q_{v_k})$. Since $p_{\vert G_w}$ is inner, $\phi(Q_w)$ fixes a vertex $y=g\cdot w$ of $T$, and $\phi(Q_w)=G_{y}=gG_w{g}^{-1}$. Let $x$ be a point of $T$ such that $d(x,y)=d(I,y)$. Since $y$ is not a translate of $v$, it follows from the definition of a centered splitting that $d(x,y)\leq 1$. Suppose for the sake of contradiction that $d(x,y)=1$. Then we can assume without loss of generality that $x=v$. Hence, $\phi$ induces a non-pinching morphism of finite-by-orbifold groups from $Q_{v_k}$ to $G_v$. Since $\phi$ is non-pinching, the group $\phi\left(Q_{w_{k,i}}\right)$ is infinite, so it is not contained in an extended conical subgroup of $G_v$. If $\phi\left(Q_{w_{k,i}}\right)$ is not contained in an extended boundary subgroup of $G_v$, it follows from Proposition \ref{complexity} that $k(Q_{w_{k,i}})\geq k(G_v)$, with equality if and only if $\phi$ is an isomorphism. On the other hand, the complexity decreases as soon as we cut along a loop or pinch a loop, so $k(G_v)=k(Q_{w_{k,i}})$ and $p$ sends $G_v$ isomorphically to a conjugate of itself. This contradicts the hypothesis. We have proved that $\phi\left(Q_{w_{k,i}}\right)$ is contained in an extended boundary subgroup of $G_v$, so it fixes a point $z$ in $T$ such that $d(x,z)=1$. As a consequence, $z=y$ or $d(z,y)=2$. This last case is impossible since $y$ and $z$ are not translates of $v$, so $z=y$ and this contradicts the definition of $x$. Hence, we have proved that $\phi\left(Q_{w_{k,i}}\right)$ fixes $y$, i.e.\ $\phi(Q_{v_k})<\phi(Q_w)$. 

Now, let $w_1$ and $w_2$ be two vertices of $\Delta_Q$, different from the vertices $v_i$, such that $w_1$ and $v_k$ are linked by an edge with infinite stabilizer, and $w_2$ and $v_k$ are linked by an edge with infinite stabilizer. We have shown that $\phi(Q_{v_k})<\phi(Q_{w_1})$ and $\phi(Q_{v_k})<\phi(Q_{w_2})$. Remark, in addition, that $Q_{v_k}$ has an extended boundary subgroup $C$ such that $\phi(C)$ is infinite. Hence, $\phi(Q_{w_1})\cap \phi(Q_{w_2})$ is infinite. By definition of a centered splitting, if an element of $G$ of infinite order fixes a segment of length $\geq 2$ in $T$, then this segment has length exactly 2 and its endpoints are translates of the central vertex $v$. Therefore, $w_1=w_2$. This completes the proof. 
\end{proof}
\color{black}
\begin{prop}\label{???}Let $G$ be a group possessing a centered splitting $\Delta_G$, with central vertex $v$. Suppose that $G$ is not finite-by-orbifold, and that there exists an endomorphism $p$ of $G$ that is $\Delta_G$-related to the identity of $G$ and that does not send $G_v$ isomorphically to a conjugate of itself. Suppose that there exists a one-ended subgroup $A$ of $G$ such that $G_v<A$ and $p_{\vert A}$ is non-injective. Then $G$ is a strict quasi-floor.
\end{prop}

\begin{proof}
~\

\textbf{A. The non-pinching case.} 

\vspace{1mm}

Suppose that $p$ is non-pinching on $G_v$. By hypothesis, $p$ does not send $G_v$ isomorphically to a conjugate of itself, so it follows from Lemma \ref{magique} that $\Delta_G$ has only one vertex $w$ different from $v$, and that $p(G_{v})<p(G_w)$. Since $p$ is inner on $G_w$, there exists an element $g\in G$ such that $\iota_g\circ p $ is a retraction from $G$ onto $G_w$. Hence, $G$ is a quasi-floor over $G_w$, and this quasi-floor is strict since, by hypothesis, there exists a one-ended subgroup $A$ of $G$ such that $G_v<A$ and $p_{\vert A}$ is non-injective.

\vspace{3mm}

\textbf{B. The pinching case.} 

\vspace{1mm}

\textbf{Step 1: pinching a maximal set of simple loops.}

\vspace{1mm}

Let $S$ be a maximal pinched set for $p$, let $Q$ be the pinched quotient and let $\Delta_{Q}$ be its pinched decomposition. Denote by $v_1,\ldots,v_n$ the new vertices coming from $v$ (see Figure \ref{pinchedquotient2} below). Let $\pi : G \twoheadrightarrow Q$ be the quotient epimorphism. There exists a unique homomorphism $\phi : Q \rightarrow G$ such that $p=\phi\circ\pi$. This homomorphism is non-pinching since $S$ is assumed to be maximal.

\begin{figure}[!h]
\centering
\includegraphics[scale=0.7]{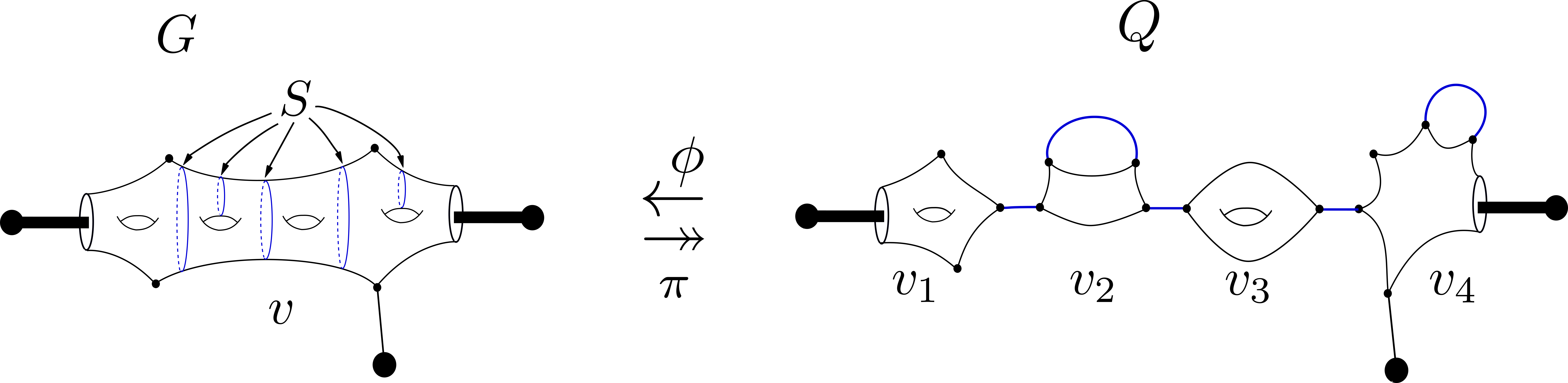}
\caption{Step 1. Edges with infinite stabilizer are depicted in bold.}
\label{pinchedquotient2}
\end{figure}

By Lemma \ref{magique}, for every $k\in\llbracket 1,n\rrbracket$, there exists a vertex $w$ of $\Delta_G$ such that $\phi({Q}_{v_k})$ is contained in a conjugate of ${G_{w}}$. If $w$ is unique, let $w_k:=w$. If $w$ is not unique, then we can assume that $w\neq v$ (since $\Delta_G$ is bipartite), and we let $w_k:=w$.

Our construction consists in eliminating the new vertices $v_1,\ldots ,v_n$ coming from the central vertex $v$. We will illustrate each step of the construction in the case of the example pictured above (Figure \ref{pinchedquotient2}).

\vspace{2mm}

\newpage

\textbf{Step 2: eliminating orbifolds with non-empty boundary.}

\vspace{1mm}

Let $\mathcal{F}$ be the set of edges of $\Delta_Q$ with finite stabilizer. For every $k\in\llbracket 1,n\rrbracket$, we denote by $Y_k$ the connected component of $\Delta_Q\setminus\mathcal{F}$ containing $v_k$.

By Lemma \ref{magique}, $\phi(Q_{Y_k})=\phi\left(Q_{w_k}\right)$ for each $k\in\llbracket 1,n\rrbracket$ such that the underlying orbifold of $Q_{v_k}$ has non-empty boundary. Therefore, the quotient of $Q$ by the subgroup normally generated by \[\Big\lbrace \ker\left(\phi_{\vert Q_{Y_k}}\right) \ \vert \ \text{the underlying orbifold of $Q_{v_k}$ has non-empty boundary}\Big\rbrace\] splits naturally as a graph of groups $\Lambda$ obtained by replacing in $\Delta_Q$ each subgraph $Y_k$ as above by a new vertex labelled by $\phi(Q_{Y_k})= \phi(Q_{w_k})=\phi\circ \pi \left(G_{w_k}\right)=p\left(G_{w_k}\right)=G_{w_k}^g$ for some $g\in G$. For the sake of clarity, this new vertex is still denoted by $w_k$ (see Figure \ref{step2} below). Call $Q'$ the fundamental group of $\Lambda$, and let $\Delta_{Q'}:=\Lambda$. Note that this graph of groups has finite edge groups. Let $\pi' : Q\twoheadrightarrow Q'$ be the quotient epimorphism. There exists a unique homomorphism $\phi' : Q' \rightarrow G$ such that $p=\phi'\circ\pi'\circ\pi$.

\vspace{2mm}

\begin{figure}[!h]
\centering
\includegraphics[scale=0.7]{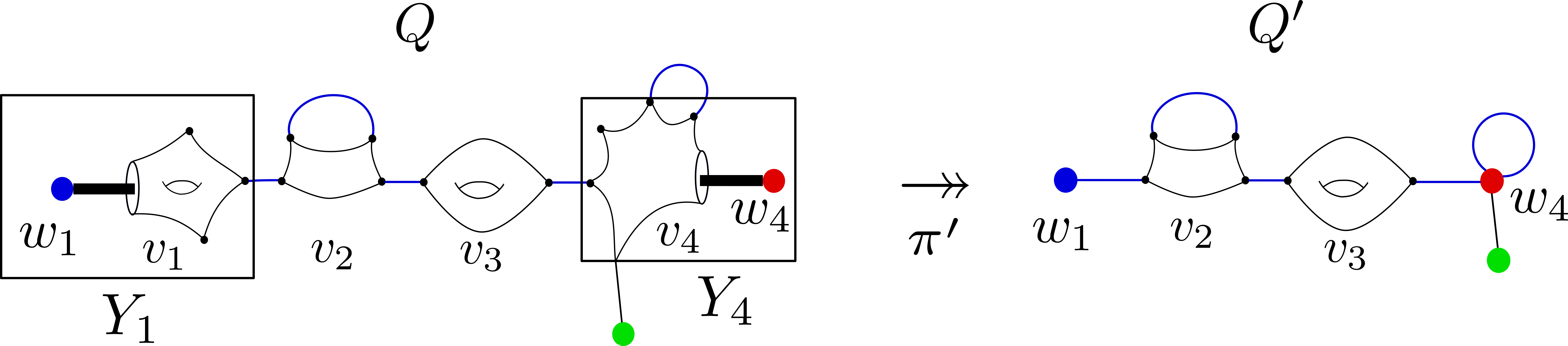}
\caption{Step 2. Edges with infinite stabilizer are depicted in bold. Note that, by construction, $\Delta_{Q'}$ has finite edge groups.}
\label{step2}
\end{figure}

\vspace{2mm}

\textbf{Step 3: eliminating vertices $v_k$ such that $w_k=v$.}

\vspace{1mm}

Let $v_k$ be a vertex such that $w_k=v$. Note that, by definition of $w_k$, $\phi'(Q'_{v_k})$ is not contained in an extended boundary subgroup of $G_v$. Since $\phi'$ is non-pinching on $Q'_{v_k}$ (by maximality of $S$), and since the complexity of $Q'_{v_k}$ is strictly less than the complexity of $G_v$, it follows from Proposition \ref{complexity} that $\phi'(Q'_{v_k})$ is contained in an extended conical subgroup of $G_v^g$. As in the previous step, we replace the vertex $v_k$ by a new vertex labelled by $\phi'(Q'_{v_k})$. This new vertex is called $x_k$ (see Figure \ref{step3} below). We perform the previous operation for each QH vertex $v_k$ such that $w_k=v$. Let $\Lambda$ be the resulting graph of groups. Call $Q''$ its fundamental group, and let $\Delta_{Q''}:=\Lambda$. Let $\pi'' : Q'\twoheadrightarrow Q''$ be the quotient epimorphism. There exists a unique homomorphism $\phi'' : Q'' \rightarrow G$ such that $p=\phi''\circ(\pi''\circ \pi'\circ\pi)$.

\begin{figure}[!h]
\centering
\includegraphics[scale=0.7]{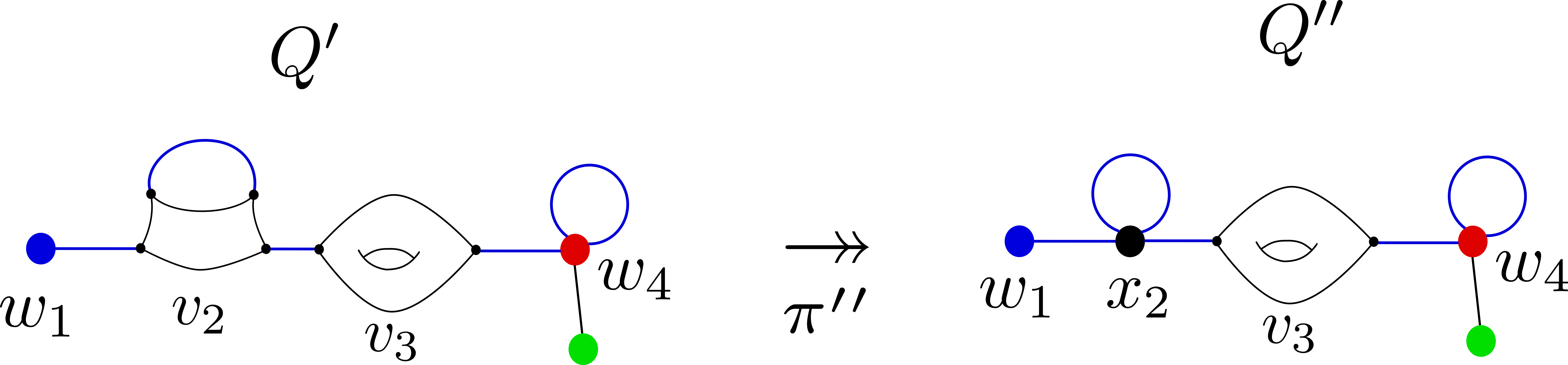}
\caption{Step 3. In this example, $w_2=v$. The vertex $v_2$ is replaced by a vertex $x_2$ labelled by the finite group $\phi'\left(Q'_{v_2}\right)$.}
\label{step3}
\end{figure}

\vspace{2mm}

\textbf{Step 4: eliminating the remaining QH vertices.}

\vspace{1mm}

Denote by $V$ the set of vertices of $\Lambda=\Delta_{Q''}$ coming from $v$ and that have not been treated yet. For each $v_k\in V$, recall that $w_k$ stands for a vertex of $\Delta_G$ such that $\phi({Q}_{v_k})$ is contained in a conjugate of ${G_{w_k}}$. Note that $w_k$ is not a translate of the central vertex $v$ of $\Delta_G$ (see Step 3).

To complete the proof, it is convenient to adopt a topological point of view. Let $X_G$ be a $K(G,1)$ obtained as a graph of spaces using, for each vertex or edge $w$ of $\Delta_G$, a $K(G_w,1)$ denoted by $X_G^w$ (see \cite{SW79}). Let $X_{Q''}$ be a $K(Q'',1)$ obtained in the same way. There exists a continuous map $f : X_{Q''} \rightarrow X_G$ inducing $\phi'': Q''\rightarrow G$ at the level of fundamental groups and such that $f\left(X_{Q''}^{v_k}\right)\subset f\left(X_{Q''}^{w_k}\right)$ for each remaining vertex $v_k$ coming from $v$, and $f$ induces an homeomophism between $X_{Q''}^w$ and $X_G^w$ for each vertex $w$ that does not come from $v$. We define an equivalence relation $\sim$ on $X_{Q''}$ by $x\sim y$ if $x=y$, or if $x\in X_{Q''}^{v_k}$, $y\in X_{Q''}^{w_k}$ and $f(x)=f(y)$. Let $g : X_{Q''}\twoheadrightarrow (X_{Q''}/\sim)$ be the quotient map. There exists a unique continuous function $h : (X_{Q''}/\sim) \rightarrow X_G$ such that $f=h\circ g$. Hence $\phi''=h_{*}\circ g_{*}$. Note that the homomorphism $g_{*}$ is not surjective in general. Call $H$ the fundamental group of $X_{Q''}/\sim$, let $j=h_{*}$ and $r=g_{*}\circ \pi''\circ \pi'\circ\pi$, so that $p=j\circ r$. Note that $X_{Q''}/\sim$ naturally has the structure of a graph of spaces, and denote by $\Delta_H$ the corresponding splitting of $H$. We claim that $G$ is a strict quasi-floor over $H$ (see below).

Let us explain the topological construction above from an algebraic point of view in the case of our example. The only remaining vertex coming from the central vertex $v$ is $v_3$ (see Figure \ref{step3}). Up to replacing $Q''$ by $Q''/\langle\langle\ker(\varphi)\rangle\rangle$, where $\varphi$ stands for the restriction of $\phi''$ to the stabilizer $Q''_{v_3}$ of $v_3$ in $Q''$, we can assume that $\phi''$ is injective on $Q''_{v_3}$. We know that $\phi''(Q''_{v_3})$ is contained in $G_{w_3}^g$ for some $g\in G$. Moreover, $\phi''$ sends $Q''_{w_3}$ isomorphically onto $G_{w_3}^h$ for some $h\in G$. As a consequence, $i:=\left(\phi''\right)^{-1}\circ\iota_{hg^{-1}}\circ \phi'' : {Q''}_{v_3}\rightarrow{Q''}_{w_3}$ is a monomorphism. We add an edge $e$ to the graph of groups $\Delta_{Q''}$ between $v_3$ and $w_3$ identifying ${Q''}_{v_3}$ with its image $i(Q''_{v_3})$ in $Q''_{w_3}$ (see Figure \ref{step4} below). Call $H$ the fundamental group of this graph of groups. In other words, we obtain $H$ by adding a new generator $t$ to $Q''$, as well as the relation $\iota_t(x)=i(x)$ for every $x\in Q''_{v_3}$. Last, we collapse the edge $e$, and we call $\Delta_H$ the resulting splitting of $H$ (see Figure \ref{fin} below). We define $r:G\rightarrow H$ as the composition of $\pi''\circ\pi'\circ\pi : G\rightarrow Q''$ with the natural homomorphism from $Q''$ to $H$. Note that $r$ is not surjective in general. Then, we define a morphism $j$ from $H$ to $G$ that extends $\phi'' : Q''\rightarrow G$ by sending $t$ to $hg^{-1}$. Since $p=\phi''\circ (\pi''\circ\pi'\circ\pi)$, we have $p=j\circ r$.

\begin{figure}[!h]
\centering
\includegraphics[scale=0.7]{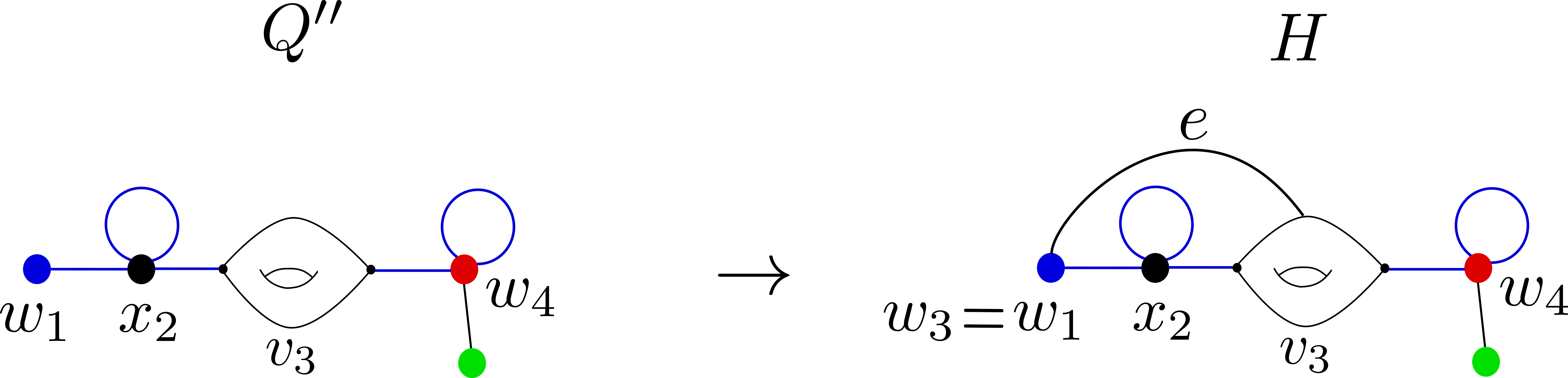}
\caption{In this example, $w_3=w_1$. We define $H$ as the fundamental group of the graph of groups obtained by adding an edge $e$ to the graph $\Delta_{Q''}$, identifying $Q''_{v_3}$ with $i(Q''_{v_3})<Q''_{w_3}$. The natural homomorphism from $Q''$ to $H$ is not surjective in general.}
\label{step4}
\end{figure}

\begin{figure}[!h]
\centering
\includegraphics[scale=0.7]{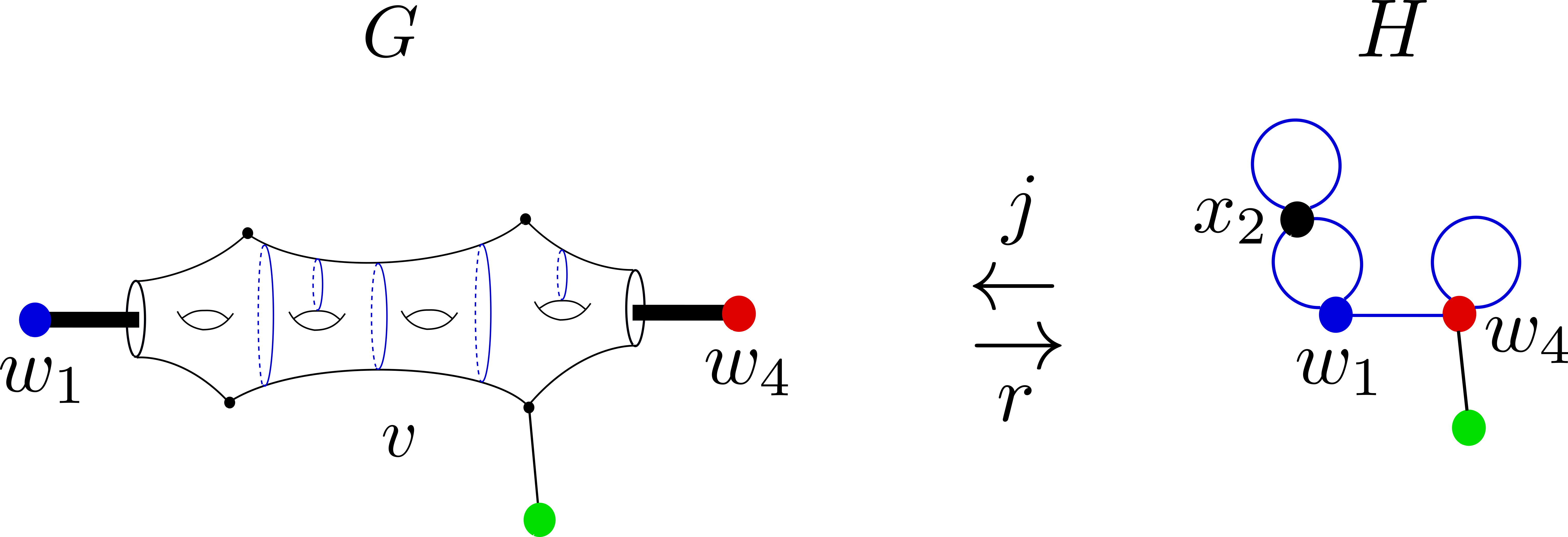}
\caption{After collapsing the edge $e$, we get the desired splitting $\Delta_H$ of $H$.}
\label{fin}
\end{figure}

\newpage

It remains to verify that $(G,H,\Delta_G,\Delta_H,r,j)$ is a strict quasi-floor.
\begin{itemize}
\item[$\bullet$]$j\circ r=p$ is $\Delta$-related to the identity of $G$ by definition of $p$.
\item[$\bullet$]Let $V_G$ be the set of vertices of $\Delta_G$, and let $V_H$ be the set of vertices of $\Delta_H$. Recall that $v$ stands for the central vertex of $\Delta_G$. By construction of $H$ and $\Delta_H$, the homomorphism $r:G\rightarrow H$ induces a bijection $s$ between $V_G\setminus \lbrace v\rbrace$ and a subset $V_1\subset V_H$ such that $r(G_w)={H_{s(w)}}$, for every $w\in V_G\setminus\lbrace v\rbrace$. 
\item[$\bullet$]Let $V_2=V_H\setminus V_1$. By construction (see Step 3 above), for every $w\in V_2$, the vertex group $H_w$ is finite and $j$ is injective on $H_w$.
\item[$\bullet$]By hypothesis, there exists a one-ended subgroup $A$ of $G$ such that $G_v<A$ and $p_{\vert A}$ is non-injective, with $p=j\circ r$. So $A\cap\ker(r)\neq \lbrace 1\rbrace$. Hence, the quasi-floor is strict.
\end{itemize}
\end{proof}

Before proving Proposition \ref{étage}, we need an easy lemma.

\begin{lemme}\label{petitlemme}
Let $G$ be a group with a splitting over finite groups. Denote by $T$ the associated Bass-Serre tree. Let $H$ be a group with a splitting over infinite groups, and let $S$ be the associated Bass-Serre tree. If $p:H\rightarrow G$ is a homomorphism injective on edge groups of $S$, and such that $p(H_v)$ is elliptic in $T$ for every vertex $v$ of $S$, then $p(H)$ is elliptic in $T$.
\end{lemme}

\begin{proof}Consider two adjacent vertices $v$ and $w$ in $S$. Let $H_v$ and $H_w$ be their stabilizers. The group $H_v\cap H_w$ is infinite by hypothesis. Moreover, $p$ is injective on edge groups, thus $p(H_v\cap H_w)$ is infinite. Hence $p(H_v)\cap p(H_w)$ is infinite. Since edge groups of $T$ are finite, $p(H_v)$ and $p(H_w)$ fix necessarily the same unique vertex $x$ of $T$. As a consequence, for each vertex $v$ of $S$, $p(H_v)$ fixes $x$. It follows that the group $p(H)$ fixes the vertex $x$.
\end{proof}

We shall now prove the main result of this section.

\begin{prop2}
Let $G$ be a finitely generated $K$-$\mathrm{CSA}$ group that does not contain $\mathbb{Z}^2$. Suppose that $G$ has a one-ended factor $H$ that is not finite-by-orbifold. Let $\Delta$ be the $\mathcal{Z}$-JSJ splitting of $H$. Suppose that there exists a non-injective homomorphism $p:H\rightarrow G$ that is $\Delta$-related to the inclusion of $H$ into $G$. Then $G$ is a strict quasi-floor.
\end{prop2}

\begin{proof}
Let $\Lambda$ be a Stallings-Dunwoody splitting of $G$ containing a vertex $v_H$ with stabilizer $H$. Let $T$ be its Bass-Serre tree. 

\vspace{1mm}

\textbf{Step 1.} We shall prove that there exists a QH vertex $v$ of $\Delta$ such that $H_v$ is not sent isomorphically to a conjugate of itself by $p$. Suppose for the sake of contradiction that each stabilizer $H_v$ of a QH vertex $v$ of $\Delta$ is sent isomorphically to a conjugate of itself by $p$. As a consequence, $p(H_v)$ is elliptic in $T$, for every QH vertex $v$. On the other hand, if $w$ is a non-QH vertex of $\Delta$, $p(H_w)$ is elliptic in $T$ by definition of $\Delta$-relatedness. Therefore, it follows from Lemma \ref{petitlemme} above that $p(H)$ is elliptic in $T$, because $p$ is injective on edge groups of $\Delta$, and $T$ has finite edge groups. Moreover, since $p$ is inner on non-QH vertices of $\Delta$, $p(H)$ is contained in $gHg^{-1}$ for some $g\in G$ (note that there exists at least one non-QH vertex since $H$ is not finite-by-orbifold by hypothesis). Up to composing $p$ by the conjugation by $g^{-1}$, one can thus assume that $p$ is an endomorphism of $H$. Now, by Proposition \ref{lemmeperin}, $p$ is injective. This is a contradiction. Hence, we have proved that there exists a QH vertex $v$ of $\Delta$ such that $H_v$ is not sent isomorphically to a conjugate of itself by $p$.

\vspace{1mm}

\textbf{Step 2.} We shall complete the proof using Proposition \ref{???}. For this purpose, we shall construct a centered splitting of $G$, together with an endomorphism of $G$ satisfying the hypotheses of Proposition \ref{???}.

First, we refine $\Lambda$ by replacing the vertex $v_H$ by the splitting $\Delta$ of $H$. \textcolor{black}{With a little abuse of notation, we still denote by $v$ the vertex of $\Lambda$ corresponding to the QH vertex $v$ of $\Delta$ defined in the previous step.} Then, we collapse to a point every connected component of the complement of $\mathrm{star}(v)$ in $\Lambda$ (where $\mathrm{star}(v)$ stands for the subgraph of $\Lambda$ constituted of $v$ and all its incident edges). The resulting graph of groups (still denoted by $\Lambda$) is non-trivial, since $H$ is not finite-by-orbifold (by hypothesis). So $\Lambda$ is a centered splitting of $G$, with central vertex $v$. 

\textcolor{black}{The homomorphism $p:H\rightarrow G$ is well-defined on $G_v$ because $G_v=H_v$ is contained in $H$. Moreover, $p$ restricts to a conjugation on each edge $e$ of $\Lambda$ incident to $v$. Indeed, either $e$ is an edge coming from $\Delta$, either $G_e$ is a finite subgroup of $H$ ; in each case, $p_{\vert G_e}$ is a conjugation since $p$ is $\Delta$-related to the inclusion of $H$ into $G$. Therefore, one can define an endomorphism $q : G \rightarrow G$ that coincides with $p$ on $G_v=H_v$ and coincides with a conjugation on every vertex group $G_w$ of $\Lambda$, with $w\neq v$. Hence, the endomorphism $q$ is $\Lambda$-related to the identity of $G$ (in the sense of Definition \ref{reliés}), and $q$ does not send $G_v$ isomorphically to a conjugate of itself, by Step 1.} 

Let us prove that the restriction of $q$ to $H$ is non-injective. In the case where $q$ kills an element of $H_v$, the claim is obvious. If the restriction of $q$ to $H_v$ is injective, then $q$ is \textit{a fortiori} non-pinching on $H_v$, and it follows that $q(H_v)$ is elliptic in $T$. Indeed, by Proposition \ref{cutting}, one can cut the underlying orbifold of $H_v$ into connected components that are elliptic in $T$ via $q$ ; but edge groups of $T$ are finite, and $q$ is non-pinching on $H_v$, so $q(H_v)$ is elliptic in $T$ by Lemma \ref{petitlemme}. Again by Lemma \ref{petitlemme}, $q(H)$ is contained in $H$ (up to conjugacy), so $q$ induces an endomorphism of $H$. Now, we are ready to find a non-trivial element in $\ker(q)\cap H$. Let $\Delta'$ be the splitting of $H$ obtained by collapsing every connected component of the complement of $\mathrm{star}(v)$ in the $\mathcal{Z}$-JSJ splitting $\Delta$ of $H$. \textcolor{black}{With abuse of notation, we still denote by $v$ the vertex of $\Delta'$ coming from the vertex $v$ of $\Delta$.} The splitting $\Delta'$ is centered, with central vertex $v$, and $q_{\vert H}$ is an endomorphism of $H$ that does not send $H_v$ isomorphically to a conjugate of itself. \textcolor{black}{Moreover, $q_{\vert H}$ is $\Delta'$-related to the identity of $H$ (in the sense of Definition \ref{reliés}); indeed, if $w$ is a vertex of $\Delta'$ different from $v$, there exists a vertex $\tilde{w}\in\Lambda$ such that $H_{w}$ is contained in $G_{\tilde{w}}$, and $q$ restricts to a conjugation on $G_{\tilde{w}}$ since $q$ is $\Lambda$-related to the identity of $G$, by construction.} So it follows from Lemma \ref{magique} that $\Delta'$ has only one vertex $w$ different from $v$, and that $q(H_v)<q(H_w)$. Since $q$ is inner on $H_w$, there exists an element $h\in H$ such that $\iota_h\circ q$ is a retraction from $H$ onto $H_w$. Let $x$ be an element of $H_v$ that does not belong to $H_w$. Let $y=\iota_h\circ q(x)$; we have seen that $y$ lies in $H_w$, so $\iota_h\circ q(y)=y$. Hence, $\iota_h\circ q(xy^{-1})=1$, with $xy^{-1}\in H\setminus \lbrace 1\rbrace$. We have proved that the restriction of $q$ to $H$ is non-injective. Now, it follows from Proposition \ref{???} that $G$ is a strict quasi-floor. This concludes the proof.
\end{proof}

\renewcommand{\refname}{References}
\bibliographystyle{alpha}
\bibliography{biblio}

\vspace{5mm}

\textbf{Simon André}

Université de Rennes 1, CNRS, IRMAR - UMR 6625, F-35000 Rennes, France.

E-mail address: \textit{simon.andre@univ-rennes1.fr}

\end{document}